\documentclass[letterpaper, 11pt]{article}
\usepackage{amsmath}
\usepackage{amssymb}
\usepackage{amsthm}
\usepackage{mathrsfs}
\usepackage{wasysym}
\usepackage{mathtools}
\usepackage[shortlabels]{enumitem}
\usepackage{graphicx}
\usepackage[left=1in, right=1in, top=1in, bottom=1in]{geometry}

\usepackage{hyperref}
\usepackage{cleveref}
\usepackage{color}
\usepackage{float}
\usepackage{soul}
\usepackage{caption}
\usepackage{subcaption}
\usepackage{circledsteps}
\usepackage{multirow}
\usepackage[toc,page]{appendix}
\usepackage{algorithm}
\usepackage{algpseudocode}
\usepackage{yhmath}
\usepackage{listings}
\usepackage{makecell}
\usepackage{orcidlink}
\usepackage[blocks]{authblk}

\usepackage{listings}
\definecolor{codegreen}{rgb}{0,0.6,0}
\definecolor{codegray}{rgb}{0.5,0.5,0.5}
\definecolor{codepurple}{rgb}{0.58,0,0.82}
\definecolor{backcolour}{rgb}{0.95,0.95,0.92}

\lstdefinestyle{mystyle}{
    backgroundcolor=\color{backcolour},   
    commentstyle=\color{codegreen},
    keywordstyle=\color{magenta},
    numberstyle=\tiny\color{codegray},
    stringstyle=\color{codepurple},
    basicstyle=\ttfamily\footnotesize,
    breakatwhitespace=false,         
    breaklines=true,                 
    captionpos=b,                    
    keepspaces=true,                 
    numbers=left,                    
    numbersep=5pt,                  
    showspaces=false,                
    showstringspaces=false,
    showtabs=false,                  
    tabsize=2
}

\usepackage[
backend=biber, 
style=alphabetic, 
sorting=ynt 
]{biblatex} 
\theoremstyle{plain}
\newtheorem{theorem}{Theorem}[section]
\newtheorem{lemma}[theorem]{Lemma}

\theoremstyle{definition}

\newcommand{\N}{\mathbb{N}}	
\newcommand{\R}{\mathbb{R}}

\setlist[enumerate]{noitemsep}
\setlist[itemize]{noitemsep}

\title{Unilateral and Equitransitive Tilings by Squares of Sizes 1, 2, 3, and 4}

\author[1]{Layan Arrabi}
\author[1]{Casey Mann \orcidlink{0000-0003-3881-3595}}
\author[2]{Fatima Muniz}
\author[3]{Harmony Vargas}
\author[4]{Rachel Villafranca \orcidlink{0009-0007-2930-252X}}

\affil[1]{University of Washington Bothell}
\affil[2]{University of Wisconsin–Oshkosh}
\affil[3]{University of Nebraska-Lincoln}
\affil[4]{Tufts University}

\date{September 2026}

\begin{document}

\maketitle

\begin{abstract}
In this article we examine a problem stated in Gr\"{u}baum and Shephard's classic book on tiling theory, \emph{Tilings and Patterns}, which asks the reader to find all unilateral and equitransitive tilings by squares of sizes $1, 2, 3, \ldots, n$ where $n \in \N$. The unilaterality condition requires that two squares of the same size in such a tiling may not intersect along an entire side, and the equitransitivity requires that two tiles of the same size in such a tiling must be in the same transitivity class with respect to the symmetry group of the tiling. These two conditions on tilings by squares allows for surprisingly interesting and complex tilings by squares. In this work we provide the answer to Gr\"{u}baum and Shephard's question for the case $n = 4$, as well as developing methods and computer algorithms that may be extended to the cases where $n \geq 5$. \end{abstract}

\section{Introduction}
We begin by defining some terminology. A \emph{tile} is a closed topological disk in $\R^2$, and a \emph{tiling} of the plane $\R^2$ is a covering of $\R^2$ by tiles such that any two distinct tiles in the tiling have disjoint interiors. A \emph{protoset} is a collection of distinct tiles $\mathcal{P}$ and we say $\mathcal{P}$ \emph{admits} a tiling $\mathcal{T}$ if every tile in $\mathcal{T}$ is congruent to a tile in $\mathcal{P}$. In this article, the protosets of interest are those consisting of squares of consecutive integer sizes $1, 2, \ldots, n$.

A \emph{symmetry} of a tiling $\mathcal{T}$ is any rigid motion of the plane $\sigma$ such that $\sigma(\mathcal{T}) = \mathcal{T}$, and the set of all symmetries of a tiling $\mathcal{T}$ is called the \emph{symmetry group} of the tiling and is denoted by $S(\mathcal{T})$. If $S(\mathcal{T})$ contains two nonparallel translations, we say that $\mathcal{T}$ is \emph{periodic}; otherwise, $\mathcal{T}$ is \emph{nonperiodic}. If for two tiles $T_1,T_2 \in \mathcal{T}$ there exists some symmetry $\sigma \in S(\mathcal{T})$ such that $\sigma(T_1) = T_2$, we say $T_1$ is \emph{equivalent} to $T_2$. If a tiling has the property that any two congruent tiles are equivalent, then we say that the tiling is \emph{equitransitive}.

When two tiles in a tiling have nonempty intersection, we say that these tiles are \emph{neighbors}; any two tiles that are neighbors intersect along their boundaries, and such an intersection may consist of some arcs which we call \emph{edges} of the tiling and/or some isolated points. These isolated points along with the endpoints of edges of a tiling are called \emph{vertices} of the tiling. In the special case that a tiling is formed from polygons, we are careful to make a distinction between the edges and vertices of the tiling and the \emph{sides} and \emph{corners} of the polygons of the tiling. See Figure \ref{fig:sides-vs-edges}. A tiling by polygons is said to be \emph{unilateral} if every edge in the tiling is a side of at most one polygon in the tiling. 

\begin{figure}[h]
\centering
\includegraphics[width=.4\textwidth]{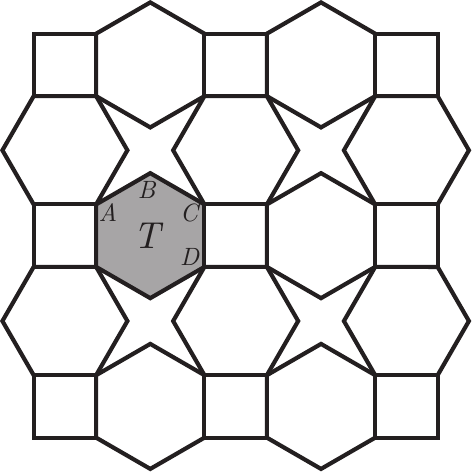}\caption{Arcs $\overline{AB}$, $\overline{BC}$, and $\overline{CD}$ are sides of tile $T$, but $\overline{AB}$ and $\overline{BC}$ are not edges of the tiling. $\overline{CD}$ and $\wideparen{ABC}$ are edges of the tiling. $A$, $B$, $C$, and $D$ are corners of $T$, but only $A$, $C$, and $D$ are vertices of the tiling. $A$, $B$, $C$, and $D$ are all corners of $T$, but only $A$, $C$, and $D$ are vertices of the tiling.}\label{fig:sides-vs-edges}
\end{figure}

In the case of tilings by squares of different sizes, which is the focus of the article, the unilaterality condition simply means that two neighboring squares of the same size can never perfectly line up, meeting along the full length of their common sides. We will refer to unilateral and equitransitive tilings by polygons generically as \emph{UE} tilings, and more specifically, we will refer to UE tilings by squares of $n$ different sizes as \emph{UE-$n$} tilings.

\subsection{Previous Results}
UE-$n$ tilings have been the subject of a few previous articles, which we will briefly overview here. The problem is first discussed in \cite{GS1987}, where several known examples discovered by the authors and other were presented. Later, in \cite{Martini}, the authors state and prove a number of general properties of unilateral tilings by squares of three sizes and demonstrated the existence of 8 different UE-3 tilings. It was later proven in \cite{Schatt}, and independently in \cite{Bolcskei}, that these 8 UE-3 tilings represented a complete classification of the UE-3 tilings. These tilings are shown in Figure \ref{fig:UE3}. In \cite{MannUE4}, the authors purported to classify the UE-4 tilings; however, it has since been revealed that this classification was incomplete, owing somewhat to the complexity of the case-checking that is at the heart of that article's main result, which was done manually.

\begin{figure}[H]
\centering
\begin{subfigure}{0.24\textwidth}
\centering
\includegraphics[width=0.8\textwidth]{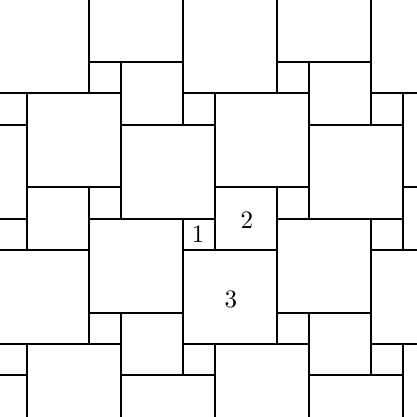} \caption{}
\end{subfigure}
\begin{subfigure}{0.24\textwidth}
\centering
\includegraphics[width=0.8\textwidth]{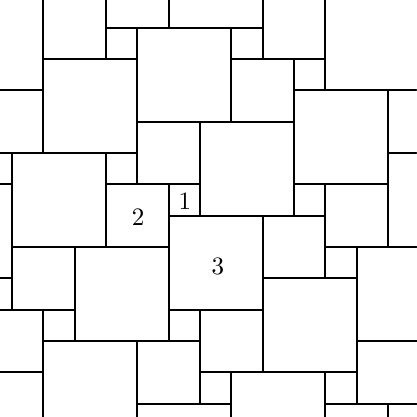} \caption{}
\end{subfigure}
\begin{subfigure}{0.24\textwidth}
\centering
\includegraphics[width=0.8\textwidth]{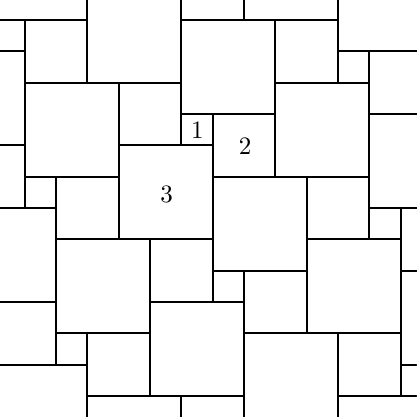} \caption{}
\end{subfigure}
\begin{subfigure}{0.24\textwidth}
\centering
\includegraphics[width=0.8\textwidth]{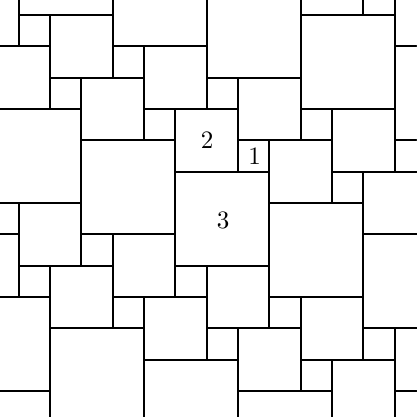} \caption{}
\end{subfigure}
\begin{subfigure}{0.24\textwidth}
\centering
\includegraphics[width=0.8\textwidth]{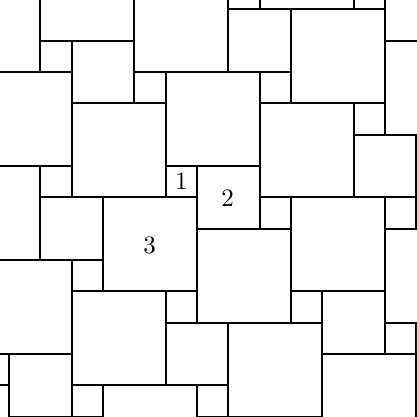} \caption{}
\end{subfigure}
\begin{subfigure}{0.24\textwidth}
\centering
\includegraphics[width=0.8\textwidth]{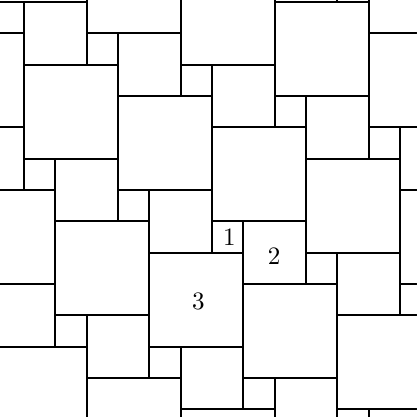} \caption{}
\end{subfigure}
\begin{subfigure}{0.24\textwidth}
\centering
\includegraphics[width=0.8\textwidth]{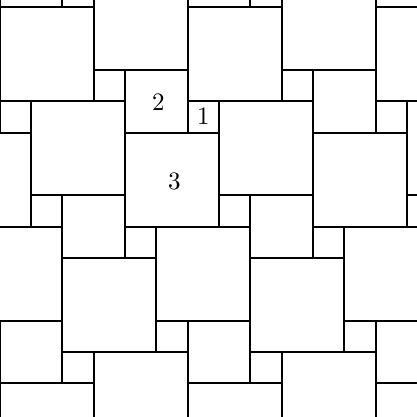} \caption{}
\end{subfigure}
\begin{subfigure}{0.24\textwidth}
\centering
\includegraphics[width=0.8\textwidth]{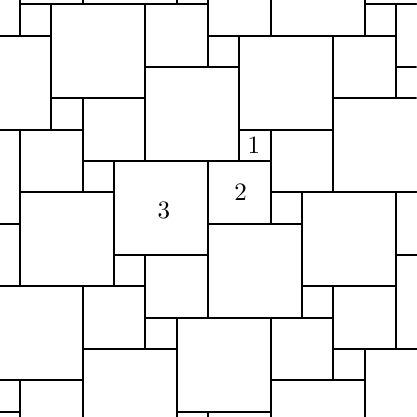} \caption{}
\end{subfigure}
\caption{The eight UE-3 tilings classified in \cite{Schatt}.}\label{fig:UE3}
\end{figure}

The main goals of this article are to (1) describe a computerized process that can be used to help determine all UE tilings by squares of consecutive integer sizes $1, 2, \ldots, n$, and (2) apply this computerized process to the problem when $n = 4$. This problem, for $n \geq 6$, is given in \cite{GS1987} (See Chapter 2.4, Exercise 5), though the authors of this article are unaware of solutions to the $n = 4$ and $n = 5$ case in the literature.  We note that the 8 tilings presented in \cite{Martini}, \cite{Schatt}, and \cite{Bolcskei} are all realizable as UE tilings by squares of sizes 1, 2, and 3, so the problem has been previously settled for $n = 3$. The main theorem of this paper (Theorem \ref{thm:main}) states that for UE-4 tilings by squares of sizes 1, 2, 3, and 4, there are 24 such tilings along with one infinite family of tilings; the tilings in this infinite family all contain a common set of equally spaced fault lines, and the only difference between any two tilings in this family is how the rows of tiles between these fault lines are shifted relative to each other so that these tilings all have a very similar appearance, even though they are technically distinct.

\section{Preliminaries}\label{sec:prelims}

\subsection{Coronas and Corona Signatures}
A \emph{corona} of a square $C$ in a UE-$n$ tiling $\mathscr{T}$ consists of $C$ along with all its neighbors in $\mathscr{T}$. A square with side length $r$ will be referred to as a square of \emph{size} $r$ or a \emph{size}-$r$ square. If $C$ is a square of size $r$, we call a corona of $C$ an \emph{$r$-corona}. An immediate consequence of the equilaterality condition of $\mathscr{T}$ is that all $r$-coronas in $\mathscr{T}$ are congruent, so the process of classifying all UE-$n$ tilings depends on finding and combining all possible $n$ distinct coronas centered on squares of $n$ distinct sizes.

Previous studies on UE-$n$ tilings have classified such tilings in terms of their coronas, and the notation used to describe coronas in these previous studies was a simple list of numbers corresponding to the sizes of those surrounding squares in the order they appear reading either clockwise or counterclockwise around the central square. For example, in Figure \ref{fig:4-corona}, the traditional corona signature for this 4-corona would be 1.3.4.3.4.3.3.2 (starting with the size-1 square on the top side and proceeding clockwise) or, equivalently, 2.3.3.4.3.4.3.1 (starting at the size-2 square on the left side and proceeding counterclockwise). For coronas involving squares of four or more sizes, the traditional notation can be ambiguous. For example, both of the 4-coronas in Figure \ref{fig:inconsistent-notation} have traditional corona signature 1.3.4.3.4.3.3.2, even though these two coronas are clearly noncongruent. Thus, our first task is to introduce a new notation for coronas that removes the ambiguity seen in the older notation.

\begin{figure}[H]
\centering
\begin{subfigure}{0.4\textwidth}
\centering
\includegraphics[width=.85\textwidth]{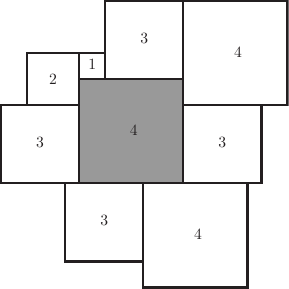}\caption{}\label{fig:4-corona}
\end{subfigure}
\begin{subfigure}{0.4\textwidth}
\centering
\includegraphics[width=.85\textwidth]{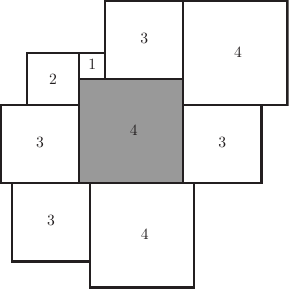}\caption{}\label{fig:4-corona-ii}
\end{subfigure}
\caption{Two different 4-coronas with ambiguous signature 1.3.4.3.4.3.3.2}\label{fig:inconsistent-notation}
\end{figure}

Throughout, we assume that all squares have horizontal and vertical sides and are oriented clockwise in the order top, right, bottom, and left. Suppose that $S$ is a central square of a corona $\mathcal{C}$. If the corners of $S$ are labeled clockwise, starting on the upper left, as $V_0, \ldots, V_3$, as in Figure \ref{fig:exponent}, then $\overline{V_0 V_1},\overline{V_1 V_2}, \overline{V_2 V_3}$ and $\overline{V_3 V_0}$ are the sides of $S$. The set of extended sides of $S$ along the edges of the corona $\mathcal{C}$ is called the \emph{skeletal extension} of $S$ (see Figure \ref{fig:skeleton}). For each side $s$ of $S$, there is a first square (with respect to the direction of the side) lying along the extension of $s$ that contains the initial vertex of $s$; we have labeled these squares by their sizes $a_0, b_0, c_0$, and $d_0$ in Figure \ref{fig:exponent}. For $i \in \{0,1,2,3\}$, let $x_i$ be the length of the intersection between the first square of the $i$-th side and the part of the skeletal extension of $S$ that is outside of $S$; we call $x_i$ the \emph{prehang} of the $i$-th side. Notice that $x_i$ satisfies $0 \leq x_i \leq  q_0$ where $q_0$ is the first square of the $i$-th side.

\begin{figure}[h]
\centering
\begin{subfigure}{0.45\textwidth}
\centering
\includegraphics[width=.85\textwidth]{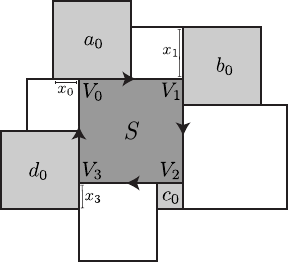}\caption{The corona $\mathcal{C}$ with central square $S$}\label{fig:exponent}
\end{subfigure}
\begin{subfigure}{0.45\textwidth}
\centering
\includegraphics[width=.85\textwidth]{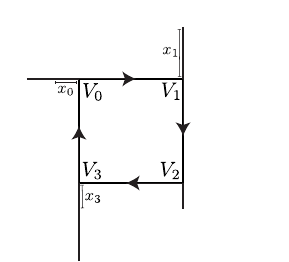}\caption{the skeletal extension of $S$ in $\mathcal{C}$}\label{fig:skeleton}
\end{subfigure}
\caption{}\label{fig:skeletal-extesnion}
\end{figure}

With these preliminaries in place, the new notation for unambiguously describing a corona will take the form \[a_0^{x_0}.a_1.\cdots.a_{n_0}\,\,;\,\,b_0^{x_1}.b_1.\cdots.b_{n_1}\,\,;\,\,c_0^{x_2}.c_1.\cdots.c_{n_2}\,\,;\,\, d_0^{x_3}.d_1.\cdots.d_{n_3}.\] The semicolons in the notation break the signature into 4 sections referred to as the \emph{sides} of the corona. For each side, the final entry ($a_{n_0}$, $b_{n_1}$, $c_{n_2}$, or $d_{n_3}$) is the size of the first square lying along that side (reading clockwise) that contains the terminal corner of that side. For example, the 4-corona of Figure \ref{fig:4-corona-iii} has signature \[1^{0}.3\,;\,4^{4}.3.2\,;\,1^{0}.3\,;\,4^{4}.3.2.\] The first section of this signature begins with $1^0$ since $1$ is the size of the first square along that side that contains the initial vertex of the top side, and it ends at $3$ (instead of $4$) because the size-3 square is the first square along that side containing the terminal corner of that side. Similarly, the second section of the signature (corresponding to the right side) begins with $4^4$ since the first square along the right side is of size $4$ the prehang there is $4$, and ends with 2 because 2 is size of the first square lying along the right side that contains the terminal corner of the right side. And so on. Notice that the exponents in the notation that specify the prehang for each side are the element of the notation that eliminates the ambiguity seen in the older notation. There are instances where we may specify an exponent in the notation with a continuous parameter. For example, in Figure \ref{fig:4-corona-iv}, we have labeled the amount of prehang on the bottom side by $x$ where $x \in \R$ satisfies $0 < x < 3$, and so the bottom side of the signature of that corona is $4^x.3$. 

\begin{figure}[h]
\centering
\begin{subfigure}{0.4\textwidth}
\centering
\includegraphics[width=.85\textwidth]{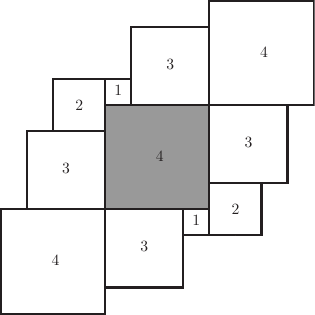}\caption{A 4-corona with signature\\ $1^{0}.3;4^{4}.3.2;1^{0}.3;4^{4}.3.2$.}\label{fig:4-corona-iii}
\end{subfigure}
\begin{subfigure}{0.4\textwidth}
\centering
\includegraphics[width=.85\textwidth]{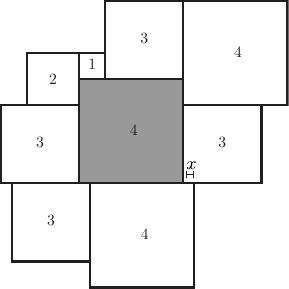}\caption{A 4-corona with signature\\ $1^{0}.3;4^{3}.3;4^{x}.3;3^{0}.2$ where $0 < x < 3$.} \label{fig:4-corona-iv}
\end{subfigure}
\caption{}\label{fig:4-corona-examples}
\end{figure}

\subsection{Skeletal Archetypes}

At each corner of the central square of a corona, the skeletal extension may extend in 1 or 2 directions; if we consider all the possibilities for all 4 corners combined, there are 15 possible types (up to symmetry) of resulting figures, shown in Figure \ref{fig:archetypes}; we call these 15 types the \emph{skeletal archetypes} of coronas. For example, the 4-corona of Figure \ref{fig:4-corona-iii} is of skeletal archetype I and the 4-corona of Figure \ref{fig:4-corona-iv} is of type IX. Note that all coronas having the same size central square in a UE-$n$ tiling must have the same skeletal archetype. Also, notice that the central squares of types I - VI and XI - XV have corners in which the edges extend in two directions; we call these \emph{double corners}. If consecutive corners of an archetype (reading around the central square clockwise) are both double corners, we will say that these double corners are \emph{adjacent}. Observe that archetypes XI - XV have adjacent double corners and the other archetypes do not.

\begin{figure}[h]
\centering
\includegraphics[width=.8\textwidth]{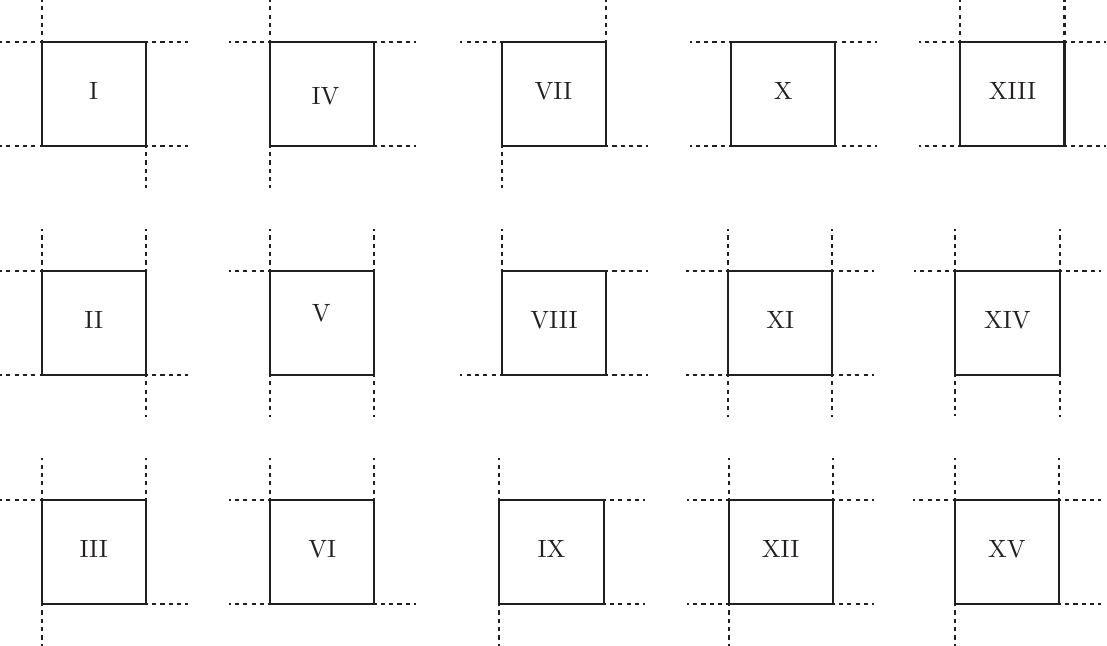}\caption{The 15 skeletal archetypes of coronas.}\label{fig:archetypes}
\end{figure}

Skeletal archetypes of coronas are used in the process of enumerating all possible coronas of each size. Here we provide a few important facts about them that will aid in the enumeration process. Coronas of skeletal archetype VII are of particular importance, and in previous articles such a corona is called a \emph{vortex}; we shall continue with that terminology. The following fact is from \cite{Martini} (See Theorems 1 and 2) speaks to the importance of vortices.

\begin{theorem}[\cite{Martini}] In a UE tiling by squares of two or more distinct sizes, the coronas of the smallest two sizes of squares must be vortices. \label{thm:vortices}\end{theorem}

\noindent We note that Theorem \ref{thm:vortices} does not require the equitransitivity property - it depends only on unilaterality. 

The following lemma points out that, for a corona in a UE-$n$ tiling, a 1 or a 2 can only appear as the first or last entry of a side in a corona signature.

\begin{lemma}
    Let $s_{0}^{x_i}.s_1.s_2.\cdots .s_k$ be a side of a corona signature for a corona in a UE tiling by squares of sizes 1, 2, $\ldots$, $n$. Then $s_i \neq 1$ and $s_i \neq 2$ for all $i$ satisfying $0 < i < k$. In other words, a size-1 or -2 square cannot appear in the middle of the side of a UE tiling by squares of sizes 1, 2, $\ldots$, $n$. \label{lem:no-1-or-2}
\end{lemma}

\begin{proof}
    For the sake of contradiction, let $S$ be a square of size $1$ or $2$ situated along side of a corona with central square $C$ suppose and $S$ is not the first or last square of that side. Note that since $S$ is not the last square of $S$, it does not contain the terminal corner of the side of $C$ along which $S$ is situated, but it may contain the initial vertex. Consider the two cases (1) $S$ does not contain the initial vertex of the side of $C$ along which it lies, or (2) it does. 
    
    For case (1), skeletal extension of $S$ extends in opposite directions along the side of $C$ on which $S$ lies, so the corona of $S$ is not a vortex, in violation of Theorem \ref{thm:vortices}. For case (2), $S$ must be the second square of that side of the corona which implies that this side of the corona must begin with a double corner. Thus, the skeletal extension of $S$ extends in two directions at this double corner of $C$, so the corona of $S$ cannot be a vortex and again we have a violation of Theorem \ref{thm:vortices}. \end{proof}

Next, for UE tilings by squares of integer sizes $1, 2, \ldots, n$ where $n \geq 2$, we have the following fact that somewhat limits the complexity of such tilings when $n$ is relatively small. 

\begin{theorem} For any UE tiling by squares of integer sizes $1, 2, \ldots, n$ where $3 \leq n \leq 6$, the skeletal archetype of any corona cannot contain adjacent double corners. Thus, in such tilings only skeletal archetypes I - X are possible.\label{thm:1-10}\end{theorem}

\begin{proof} Suppose that a corona with central square $C$ of size $c$ has adjacent double corners. By Theorem \ref{thm:vortices}, $c \geq 3$. Let $\mathcal{S}$ be the side of the corona that begins at the first of these adjacent double corners (as always, reading clockwise). Then by Theorem \ref{thm:vortices}, the squares at these adjacent double corners cannot be of size 1 or 2, nor may the second or last squares of $S$ be of size 1 or 2 (these are the squares of $S$ that are adjacent to the double corners). Additionally, by Lemma \ref{lem:no-1-or-2}, none of the other squares of $\mathcal{S}$ are of size 1 or 2, so none of the squares along the side of $C$ between double corners can be of size 1 or 2.

Now consider the size $c$ in relation to the sequence $(s_1, s_2, \ldots, s_n)$ of sizes of squares from the side $\mathcal{S}$ that lie between the double corners of $\mathcal{C}$. Note that this sequence must have more than one term to avoid a violation of unilaterality. If $c = 3$, the only possibilities for $(s_1, s_2, \ldots, s_n)$ in which consecutive terms are not equal (to avoid a violation of unilaterality) are $(1,2)$ and $(2,1)$; if $c = 4$, the possibilities are $(1,3)$, $(3,1)$, and $(1,2,1)$; if $c = 5$, the possibilities are $(1,4)$, $(4,1)$, $(2,3)$,  $(3,2)$, $(1,3,1)$, and $(2,1,2)$; and finally, if $c = 6$, the possibilities are $(1,5)$, $(5,1)$, $(2,4)$, $(4,2)$, permutations of $(1,2,3)$, $(1,4,1)$, $(2,1,2,1)$, and $(1,2,1,2)$. Every possibility presented contains a 1 or a 2, but as was previously established, none of the squares between adjacent double corners can be of size 1 or 2. \end{proof}

It is possible to have adjacent double corners when $n \geq 7$. For example, the sequence of sizes of squares between the double corners could be $(3,4)$ for a central square of size $7$ or $(3,5)$ for a central square of size $8$, and etc.

Observe that if the signatures of the coronas of two distinct sizes of squares involve only those two sizes of squares, then any UE-$n$ with tiling containing coronas with these signatures must have $n = 2$. For example, suppose  that $\mathcal{C}_1$ is a 1-corona containing only squares of size 1 and 2 and suppose that $\mathcal{C}_2$ is a 2-corona containing only squares of size 1 and 2. If $\mathcal{C}_1$ and $\mathcal{C}_2$ were to appear in any UE-$n$ tiling with $n > 2$, then there would be a third size of square somewhere in that tiling that neighbors a size-1 or a size-2 square, contrary to the fact that $\mathcal{C}_1$ and $\mathcal{C}_2$ contain only size-1 and size-2 squares. A similar line of reasoning can be used to see that, in a UE-$n$ tiling with $n > 3$, there cannot be coronas of three distinct sizes of squares (say, $a$, $b$, and $c$) in which those three coronas contain only squares of sizes $a$, $b$, and $c$. We generalize this observation in the following theorem, whose proof follows the same logic.

\begin{theorem}
    Suppose that $\mathcal{T}$ is a UE-$n$ tiling and let $\mathcal{S} = \{\mathcal{C}_1, \mathcal{C}_2, \ldots, \mathcal{C}_n\}$ be the set of distinct corona signatures of coronas in $\mathcal{T}$. For any nonempty subset $\mathcal{R}$ of $\mathcal{S}$, let $N(\mathcal{R})$ denote the total number of different sizes of squares contained in all of the signatures in $\mathcal{R}$, and let $|\mathcal{R}|$ denote the number of signatures in $\mathcal{R}$. If $N(\mathcal{R}) < n$, then  $N(\mathcal{R}) > |\mathcal{R}|$.\label{thm:no-closed-subsets}
\end{theorem}

\section{Algorithms}

The algorithms discussed in this section were developed to produce all possible 4-tuples of signatures of coronas in UE-4 tilings by squares of consecutive integer sizes 1, 2, 3, and 4, and we note that these algorithms are generalized so that they can be applied to the larger problem of classifying all UE-$n$ tilings by squares of consecutive integer sizes $1, 2, \ldots n$ where $n \geq 5$, and so we explain the algorithms in these more general terms. The main steps of the algorithm are the following:

\begin{enumerate}
    \item For each $i \in \{1,2,...,n\}$, generate all $i$-coronas that are viable in that they do not contain local configurations that violate the unilaterality requirement, Theorem \ref{thm:vortices}, Lemma \ref{lem:no-1-or-2}, or Theorem \ref{thm:1-10} (for $n \leq 6$). 
    \item For each tuple $(\mathcal{C}_1, \mathcal{C}_2, \mathcal{C}_3, \ldots \mathcal{C}_n)$ where each $\mathcal{C}_i$ is a viable $i$-corona found in Step 1, determine if these coronas are compatible with each other in terms of equitransitivity. This means that any partial $i$-coronas appearing in $\mathcal{C}_j$ must be found in $\mathcal{C}_i$ for all pairs $(i,j)$ where $i,j \in \{1,...,n\}$.
\end{enumerate}
These two steps can be completely automated. There is a third step that we have not found a way to automate, which is: \begin{enumerate}
    \item[3.] For each tuple of compatible coronas $(\mathcal{C}_1, \mathcal{C}_2, \mathcal{C}_3, \ldots \mathcal{C}_n)$ identified in Step 2, determine if a UE-$n$ tiling(s) can be formed from these coronas.
\end{enumerate} In the case that $n=4$, there were so few quadruples $(\mathcal{C}_1, \mathcal{C}_2, \mathcal{C}_3, \mathcal{C}_4)$ that Step 3 could be done by hand. In order to scale this problem to larger values of $n$, a process for automating Step 3 will need to be developed, but we leave that as an open problem.

With this overview of the general strategy in place, we now go into the details of the algorithms.

\subsection{Generating Viable Coronas} \label{subsec:viable_sides}
We break this part of the algorithm into a two subparts: \begin{enumerate}
    \item[(i)] (Viable Side Generation) Recall that a corona signature has four parts called sides having the form $s_0^{x}.s_1.\ldots.s_{k}$ where $s_i \in \{1,2,\ldots,n\}$ for each $i$. For each size $C$ of central square, the algorithm will generate all possible such sides with integral exponent $x$, being careful that these sides avoid violating the unilaterality requirement, Theorem \ref{thm:vortices}, or Lemma \ref{lem:no-1-or-2}. These sides will come in 4 varieties, as explained in Section \ref{subsubsec:sides}.
    \item[(ii)] (Stitching Viable Sides Together) After calculating all viable sides for each central size $C$ of square, these sides are combined in all possible ways, again being careful to avoid introducing any violations of the unilaterality requirement, Theorem \ref{thm:vortices}, or Theorem \ref{thm:1-10} while removing from consideration any redundant copies of coronas so created. Combining sides to create coronas is organized around skeletal archetypes, as described in Section \ref{subsubsec:stitching}.
\end{enumerate}

\subsubsection{Viable Side Generation}\label{subsubsec:sides}
Recalling the corona notation we introduced in Section \ref{sec:prelims},  the squares along a single side of a corona can be expressed in the form $s_{0}^{x_i}.s_1. \cdots .s_{n_i}$ where the exponent $x_i$ specifies the placement of these squares along a side relative to the central square. In particular, the exponent $x_i$ is the amount of \emph{prehang} of the first square relative to the central square, and there might also be some \emph{afterhang} $y_i$ along a side (see Figure \ref{fig:hang}) which is computed as \begin{equation} y_i = \sum_{k = 0}^{n_i} s_k - x_i -C. \label{eqn:afterhang} \end{equation}

\begin{figure}[h]
\centering
\includegraphics[width=.5\textwidth]{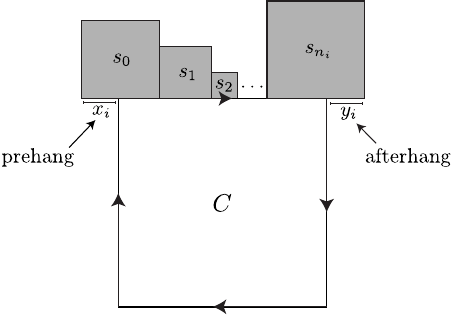}\caption{Prehang and afterhang of side $s_{0}^{x_i}.s_1.s_2. \cdots .s_{n_i}$.}\label{fig:hang}
\end{figure}

We define \emph{side types} 1 - 4 as a function of the values of $x_i$ and $y_i$ for a given side as follows:

\begin{itemize}
    \item Type 1: $x_i = 0$ and $y_i > 0$ (afterhang only)
    \item Type 2: $x_i > 0$ and $y_i = 0$ (prehang only)
    \item Type 3: $x_i = 0$ and $y_i = 0$ (no prehang or afterhang)
    \item Type 4: $x_i > 0$ and $y_i > 0$ (both prehang and afterhang)
\end{itemize}

Note that the afterhang $y_i$ satisfies $0 \leq y_i < s_{n_i}$; we cannot have $y_i = s_{n_i}$ for that would imply that $s_{n_i - 1}$ contains the terminal corner of the corresponding side of the central square, contradicting the convention that the last square $s_{n_i}$ is the first square of a that side containing the terminal corner of that side of the central square. Also note that $x_i = s_0$ if and only if there is a double corner at the beginning of that side. Thus, any side of a corona (as seen in the list of skeletal archetypes in Figure \ref{fig:archetypes}) fits into exactly one of these four side types, with double corners occurring only at the beginning of sides of type 2 or 4 when the exponent equals the size of the first square of such a side.\\

\noindent \textbf{\underline{Type 1 sides}}\\ Given that $x_i = 0$, $y_i > 0$ , and the central square is of size $C$, we see that the sum of the sizes of the squares $s_k$ for a type 1 side must satisfy \begin{equation} C + 1 \leq \sum_{k = 0}^{n_i}s_k \leq C + n - 1\label{eqn:partition_constraint_1} \end{equation} where $n$ is the maximum size square. For each integer $j$ satisfying $C + 1 \leq j \leq C + n - 1$, let $S_j$ be the set of finite sequences of numbers in $\{1,2,\ldots,n\}$ whose sums are exactly $j$ (i.e., $S_j$ is a set of \emph{integer compositions} of $j$ written in list form). If $(s_0, s_1, \ldots, s_{\ell}) \in S_j$, then by construction, $s_{0}^{0}.s_1. \ldots .s_{\ell}$ satisfies Equation \ref{eqn:partition_constraint_1} and is therefore a \emph{potential} type 1 side of a $C$-corona, but it may be invalid for a number of reasons, so we filter out all such invalid sides as follows:
\begin{enumerate}
    \item Remove $s_{0}^{0}.s_1. \ldots .s_{\ell}$ if $s_0 = C$; this would be a violation of unilaterality.
    \item Remove $s_{0}^{0}.s_1. \ldots .s_{\ell}$ if $s_i = s_{i+1}$ for any integer $i$ satisfying $0 \leq i < \ell$; this would be a violation of unilaterality.
    \item $s_{0}^{0}.s_1. \ldots .s_{\ell}$ if $s_i = 1$ or $s_i = 2$ for any integer $i$ satisfying $1 < i < \ell$; this would violate Lemma \ref{lem:no-1-or-2}.
    \item Remove $s_{0}^{0}.s_1. \ldots .p_{\ell}$ if $y_i$ as calculated in Equation \ref{eqn:afterhang} does not satisfy $0 < y_i < s_{\ell}$.
\end{enumerate}

\noindent \textbf{\underline{Type 2 sides}}\\ Producing and filtering the remaining side types are handled in a very similar way as is done for type 1 sides, with minor variations. For type 2 sides, since the prehang $x_i$ satisfies $1 \leq x_i, \leq s_0$, a side $s_{0}^{x_i}.s_1. \cdots .s_{\ell}$ must satisfy \begin{equation} C + 1 \leq \sum_{k = 0}^{n_i}s_k \leq C + n, \label{eqn:partition_constraint_2} \end{equation} prompting us to produce a set $S_j$ of integer compositions of $j$ for each integer $j$ satisfying $C + 1 \leq j \leq C + n$. Since the afterhang $y_i = 0$ for type 2 sides, the value $j-C$ must be the amount of prehang for such a side, so for each such $j$, let $x = j-C$ and let $(s_0, \ldots, s_{\ell}) \in S_j$, then consider the potential type 2 side $s_{0}^{x}.s_1. \cdots .s_{\ell}$ where $x$ satisfies $x \leq s_0$, applying the following filters:

\begin{enumerate}
    \item Remove $s_{0}^{x}.s_1. \cdots .s_{\ell}$ if $s_i = s_{i+1}$ for any integer $i$ satisfying $0 \leq i < \ell$; this would be a violation of unilaterality.
    \item Remove $s_{0}^{x}.s_1. \cdots .s_{\ell}$ if $s_i = 1$ or $s_i = 2$ for any integer $i$ satisfying $0 < i < \ell$; a size-1 or -2 square appearing in the middle of a side would result in a violation of Lemma \ref{lem:no-1-or-2}.
    \item Remove $s_{0}^{x}.s_1. \cdots .s_{\ell}$ if $x = s_0 = 1$ or $2$; a square of size 1 or 2 cannot be the first square of a side beginning with a double corner as this results in a violation of Theorem \ref{thm:vortices}.
    \item Remove $s_{0}^{x}.s_1. \cdots .s_{\ell}$ if $x = s_0$ and $s_1 = C$; this would be a violation of unilaterality.
\end{enumerate}

\noindent \textbf{\underline{Type 3 sides}}\\ Type 3 sides are the simplest to generate and filter. For side type 3, we must have \begin{equation}
    \sum_{k = 0}^{n_i}s_k = C, \label{eqn:partition_constraint_iii}
\end{equation} and then for each $(s_0, \ldots, s_{\ell}) \in S_C$, we have $s_{0}^{0}.s_1. \cdots .s_{\ell}$ is a potential side. We then filter out potential sides as follows:

\begin{enumerate}
    \item Remove $s_{0}^{0}.s_1. \cdots .s_{\ell}$ if $s_0 = C$ or $s_{\ell} = C$; this would be a violation of unilaterality.
    \item Remove $s_{0}^{0}.s_1. \cdots .s_{\ell}$ if $s_i = s_{i+1}$ for any integer $i$ satisfying $0 \leq i < \ell$; this would be a violation of unilaterality.
    \item Remove $s_{0}^{0}.s_1. \cdots .s_{\ell}$ if $s_i = 1$ or $s_i = 2$ for any integer $i$ satisfying $1 < i < \ell$; a size-1 or -2 square appearing in the middle of a side would result in a violation of Lemma \ref{lem:no-1-or-2}.
\end{enumerate}

\noindent \textbf{\underline{Type 4 sides}}\\ For type 4 sides, we begin by noting that the sum of the sizes of the squares $s_k$ must satisfy the equation \begin{equation}
    C+2 \leq \sum_{k = 1}^{n_i}s_k \leq C + 2n - 1.
\end{equation} For each integer $j$ satisfying $C+2 \leq j \leq C + 2n - 1$ we again form the set $S_j$. There is a complication with this side type that makes it fundamentally different than the other three side types; this complication is that for some $(s_0, \ldots, s_{\ell}) \in S_j$, it is possible to have a valid corresponding type 4 side $s_{0}^{x_i}.s_1. \cdots .s_{\ell}$ where $x_i$ is not an integer. The idea here is that, as depicted in Figure \ref{fig:hang}, it may be possible to continuously slide the squares of a type 4 side along the line where it meets the central square, which would allow the prehang exponent $x_i$ to vary continuously. The algorithm for enumerating type 4 sides will depend on $x_i$ being an integer, but as it turns out, the valid integral values of $x_i$ will be sufficient to represent the nonintegral values. A few conditions force $x_i$ to be an integer. First, if any of the $s_i$ equal $1$ or $2$, the fact that size-$1$ and size-$2$ squares must be vortices (Theorem \ref{thm:vortices}) forces $x_i$ to be an integer satisfying $0 < x_i \leq s_0$. Also, if a type 4 side is in a corona such that either the square immediately preceding or succeeding that side is a size-1 or size-2 square, Theorem \ref{thm:vortices} again prevents the type 4 side from having a noninteger prehang $x_i$. For $(s_0, \ldots, s_{\ell}) \in S_j$ with $s_i \neq 1$ and $s_i \neq 2$ for all $i$, a type 4 side of the form $s_{0}^{x_i}.s_1. \cdots .s_{\ell}$ is called a \emph{faulty} type 4 side.  A corona containing a faulty side that is not preceded or succeeded by a size-1 or -2 square will be called a \emph{faulty} corona. The extension of a side of the central square of a faulty corona that bears a faulty type 4 side is called \emph{fault line} of the corona. It will be seen in Section \ref{subsec:faulty_tilings} that for UE-4 tilings by squares of sizes 1, 2, 3, and 4 in which both the 3- and 4-coronas are faulty, the fault lines of these faulty 3- and 4-coronas extend infinitely in both directions in UE-$4$ tiling containing these coronas, forming equally spaced parallel lines in these tilings; we therefore define a UE-4 tiling to be a \emph{faulty tiling} if its 3- and 4-coronas are faulty. We emphasize that the terminology ``faulty" relates to these fault lines, and not that the coronas a defective or wrong in some sense.

Figure \ref{fig:faulty_coronas} shows examples of faulty 3- and 4-coronas with faulty top sides. Notice in Figure \ref{fig:faulty_coronas} that the faulty side of the 3-corona has a representative in which the exponent is an integer (when $x = 1$), as does the faulty side of the 4-corona (when $y = 3$ or $y = 4$). We will see in Section \ref{subsec:faulty_tilings} that it is always the case that faulty 3- and 4-coronas for UE-4 tilings always have representatives with integer exponents in their signatures so that our computer algorithm, which operates by enumerating the exponents as integers, will identify representatives of all instances of faulty coronas.

\begin{figure}[h]
\centering
\begin{subfigure}{0.35\textwidth}
\centering
\includegraphics[width=.95\textwidth]{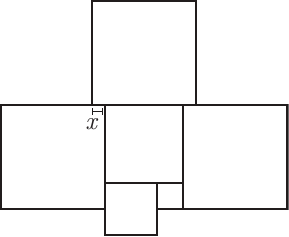}\caption{A faulty 3-corona with faulty top side $4^x$, $0 < x \leq 1$}\label{fig:faulty_3_top}
\end{subfigure}
\hspace{2cm}
\begin{subfigure}{0.35\textwidth}
\centering
\includegraphics[width=.95\textwidth]{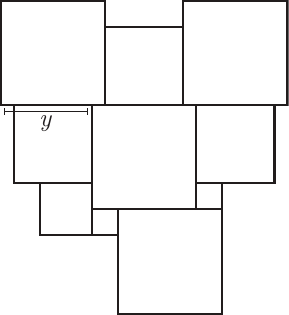}\caption{A faulty 4-corona with faulty top side $4^y.3.4$, $3 < y \leq 4$ }\label{fig:faulty_4}
\end{subfigure}
\caption{Faulty coronas}\label{fig:faulty_coronas}
\end{figure}

Setting aside consideration of faulty type 4 sides for now, the type 4 sides of the form $s_{0}^{x_i}.s_1. \cdots .s_{\ell}$ where $x_i$ is an integer satisfying $0 < x_i \leq s_0$ may be enumerated, and then we apply the following filters to remove invalid instances:
\begin{enumerate}
    \item Remove $s_{0}^{x_i}.s_1. \cdots .s_{\ell}$ if $s_i = s_{i+1}$ for any integer $i$ satisfying $1 \leq i < \ell$; this would be a violation of unilaterality.
    \item Remove $s_{0}^{x_i}.s_1. \cdots .s_{\ell}$ if $s_i = 1$ or $s_i = 2$ for any integer $i$ satisfying $1 < i < \ell$; a size-1 or -2 square appearing in the middle of a side would result in a violation of Lemma \ref{lem:no-1-or-2}.
    \item Remove $s_{0}^{x_i}.s_1. \cdots .s_{\ell}$ if $x_i = s_0 = 1$ or $2$; a square of size 1 or 2 cannot be the first square of a side beginning with a double corner as this results in a violation of Theorem \ref{thm:vortices}.
    \item Remove $s_{0}^{x_i}.s_1. \cdots .s_{\ell}$ if $x_i = s_0$ and $s_1 = C$; this would be a violation of unilaterality.
    \item Remove $s_{0}^{x_i}.s_1. \cdots .s_{\ell}$ if $y_i$ as calculated in Equation \ref{eqn:afterhang} does not satisfy $0 < y_i < s_{\ell}$.
\end{enumerate}

SageMath/Python implementations of the side types 1-4 generators are provided in Appendix \ref{apx:sides}.

\subsubsection{Stitching Together Viable Sides to Get Viable Coronas}\label{subsubsec:stitching}

Notice that each instance of a double corner among the skeletal archetypes (Figure \ref{fig:archetypes}) occurs at the beginning of a type 2 or a type 4 side (sides with positive prehang) such that the exponent is equal to the size of the first square; when the exponent is equal to the first square of a type $2$ or type $4$ side, we will indicate that with $2^{\ast}$ or $4^{\ast}$. With this notation, each of these skeletal archetypes can be viewed as a 4-tuple of side types, as summarized in Table \ref{tab:corona_tuples}. Note that by Theorem \ref{thm:1-10}, the coronas of squares in UE-$n$ tilings by squares of consecutive integer sizes $1, 2, \ldots, n$ can be only from skeletal archetypes I - X when $n = 4$, but when $n >6$ other types may be possible.

\begin{table}[h]
\centering
\begin{tabular}{c|c}
\textbf{Corona Skeletal Type} & \textbf{Side Type 4-Tuple}\\\hline
I & $(4^{\ast},3,4^{\ast},3)$\\
II & $(2^{\ast},2,4^{\ast},3)$\\
III & $(2^{\ast},2,2,2)$\\
IV & $(4^{\ast},3,2,2)$\\
V & $(2^{\ast},4,3,2)$\\
VI & $(2^{\ast},2,4,3)$\\
VII & $(2,2,2,2)$\\
VIII & $(1,3,4,1)$\\
IX & $(1,3,2,4)$\\
X & $(4,3,4,3)$\\
XI & $(2^{\ast},2^{\ast},2^{\ast},2^{\ast})$\\
XII & $(2^{\ast},2^{\ast},2,2^{\ast})$\\
XIII & $(2^{\ast},2^{\ast},4,3)$\\
XIV & $(2^{\ast},4^{\ast},3,2)$\\
XV &  $(2^{\ast},2^{\ast},2,2)$\end{tabular} \caption{Corona skeletal archetypes described as tuples of side types.}\label{tab:corona_tuples}
\end{table} 

Let the set of all viable sides of type $i$ (for a central square of some fixed size $C$) be denoted by $\mathcal{S}_i$  and let $\mathcal{S}_{2}^{\ast}$ and $\mathcal{S}_{4}^{\ast}$ be the subsets of $\mathcal{S}_2$ and $\mathcal{S}_4$ consisting of sides with exponents equal to the size of the first square (e.g., the sides beginning with double corners). Then, for example, we can compute all potential coronas of skeletal type V by computing the cartesian product $\mathcal{S}_{2}^{\ast} \times \mathcal{S}_{4} \times \mathcal{S}_{3} \times \mathcal{S}_{2}$. However, blindly combining 4 sides in such a manner can result in violations of unilaterality and in violations of Theorem \ref{thm:vortices} in the corona so formed. Additionally, there may be symmetric duplicates produced (for example, a corona of skeletal archetype I can have rotational symmetry through an angle of $\pi$). Thus, after computing these cartesian products, we remove nonviable or redundant 4-tuples $(s_0, s_1, s_2, s_3)$ from a cartesian product as follows:
\begin{enumerate}
    \item Remove $(s_0, s_1, s_2, s_3)$ if for some $i \in \{0,1,2,3\}$ the last entry of $s_i$ is equal to both the first entry of $s_{i+1}$ and the exponent of $s_{i+1}$ (where $i+1$ is computed modulo $4$). This ensures that unilaterality is not violated at the ends where two sides meet.
    
    \item For $i \in \{0,1,2,3\}$, let $x_i$ be the prehang of side $s_i$, let $C$ be the size of the center square, let $a$ be the size of the first square in $s_i$, and let $w$ be the size of the last square of the previous side $s_{i-1}$ (where $i-1$ is computed modulo 4). If any of the following conditions hold, we remove $(s_0, s_1, s_2, s_3)$:
    \begin{enumerate}
        \item $w = 1$ and $x_i > 1$.
        \item $w = 2$ and $x_i > 2$.
        \item $C \neq 1$ and $x_i = 1$ and $a = 2$ and $w \neq 1$.
    \end{enumerate} This ensures that Theorem \ref{thm:vortices} is not violated by squares of size 1 or 2 adjacent to the prehang line at the beginning of side $s_i$.
    
    \item For $i \in \{0,1,2,3\}$, let $y_i$ be the afterhang of side $s_i$, let $C$ be the size of the center square, let $b$ be the size of the last square in $s_i$, and let $z$ be the size of the first square of the next side $s_{i+1}$ (where $i+1$ is computed modulo 4). If any of the following conditions hold, we remove $(s_0, s_1, s_2, s_3)$:
    \begin{enumerate}
        \item $z = 1$ and $y_i > 1$. 
        \item $z = 2$ and $y_i > 2$.
        \item $C \neq 1$ and $y_i = 1$ and $b = 2$ and $z \neq 1$. 
    \end{enumerate} This ensures that Theorem \ref{thm:vortices} is not violated by squares of size 1 or 2 in adjacent to the afterhang line at the end of side $s_i$.

    \item Remove $(s'_0, s'_1, s'_2, s'_3)$ if it is a symmetric copy of a previously 4-tuple $(s_0, s_1, s_2, s_3)$. In particular, given $(s'_0, s'_1, s'_2, s'_3)$, compute the reflection of $(s'_0, s'_1, s'_2, s'_3)$ over the vertical and call this $(s''_0, s''_1, s''_2, s''_3)$, then compute the rotation by angles of $\pi/2$, $\pi$, and $3\pi/2$ of both $(s'_0, s'_1, s'_2, s'_3)$ and $(s''_0, s''_1, s''_2, s''_3)$. If any such transformation $(s'_0, s'_1, s'_2, s'_3)$ is equal to a previous 4-tuple $(s_0, s_1, s_2, s_3)$ in the list of 4-tuples of coronas, then remove $(s'_0, s'_1, s'_2, s'_3)$ from further consideration.
\end{enumerate}
In Table \ref{tab:coronas} we summarize the numbers of 1-, 2-, 3-, and 4-coronas found for each skeletal archetype after removing nonviable and redundant symmetric copies of coronas as described above. 

\begin{table}[h]
\centering
\begin{tabular}{c||c|c|c|c|c|c|c|c|c|c||c}
Center size & I & II & III & IV & V & VI & VII & VIII & IX & X & Total\\\hline
1 & 0 & 0 & 0 & 0 & 0 & 0 & 15 & 0 & 0 & 0 & 15\\
2 & 0 & 0 & 0 & 0 & 0 & 0 & 24 & 0 & 0 & 0 & 24\\
3 & 0 & 0 & 0 & 6 & 0 & 0 & 37 & 46 & 22 & 8 & 117\\
4 & 1 & 1 & 5 & 4 & 5 & 8 & 29 & 72 & 41 & 17 & 180\end{tabular}\caption{Summary of numbers of coronas of each skeletal archetype.}\label{tab:coronas}\end{table}

\subsection{Corona Compatibility Checking}

The last major step of our computerized process for identifying UE tilings by squares is consider every possible $n$-tuple of coronas $(\mathcal{C}_1, \mathcal{C}_2, \ldots, \mathcal{C}_n)$ (where $\mathcal{C}_i$ is an $i$-corona) and determine if this tuple corresponds to a tiling. A necessary condition is that the coronas in the tuple must be \emph{compatible}, which we define here: 

For a corona $\mathcal{C}$ central square $m$, let $s$ be any square in $\mathcal{C} \setminus \{m\}$. We define the \emph{partial corona} of $s$ in $\mathcal{C}$ to be the set of all squares in $\mathcal{C}$ that are neighbors of $s$ (see Figure \ref{fig:partial_coronas}). Let $\mathcal{C}_i$ be an $i$-corona and let $\mathcal{C}_j$ be an $j$-corona (allowing the possibility that $\mathcal{C}_i = \mathcal{C}_j$). We say $\mathcal{C}_i$ and $\mathcal{C}_j$ are \emph{compatible} if for each square $s \in \mathcal{C}_j$ of size $i$, a congruent copy of the partial corona of $s$ in $\mathcal{C}_j$ is also present in $\mathcal{C}_i$, and vice versa: for each square $t \in \mathcal{C}_i$ of size $j$, a congruent copy of the partial corona of $t$ in $\mathcal{C}_i$ is also present in $\mathcal{C}_j$. If $\mathcal{C}_i$ and $\mathcal{C}_j$ are not compatible, we say they are \emph{incompatible}. For example, the two coronas in  Figure \ref{fig:incompatible} are incompatible.

\begin{figure}[h]
\centering
\begin{subfigure}{0.45\textwidth}
\centering
\includegraphics[width=.85\textwidth]{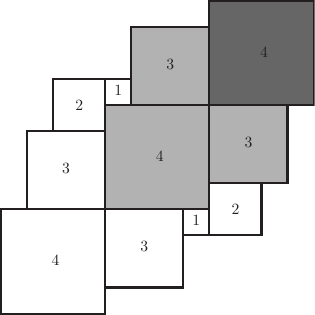}\caption{}\label{fig:partial_a}
\end{subfigure}
\begin{subfigure}{0.45\textwidth}
\centering
\includegraphics[width=.85\textwidth]{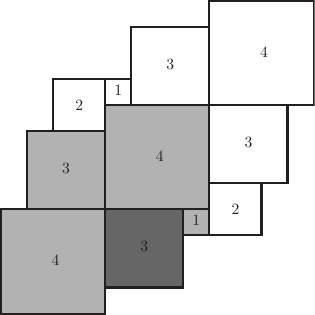}\caption{}\label{fig:partial-b}
\end{subfigure}
\caption{Partial coronas of squares in a corona can consist of 4 squares (as in the partial 4-corona shaded Figure \ref{fig:partial_a}) or 5 squares (as in the partial 3-corona in Figure \ref{fig:partial-b}).}\label{fig:partial_coronas}
\end{figure}

\begin{figure}[h]
\centering
\includegraphics[width=.75\textwidth]{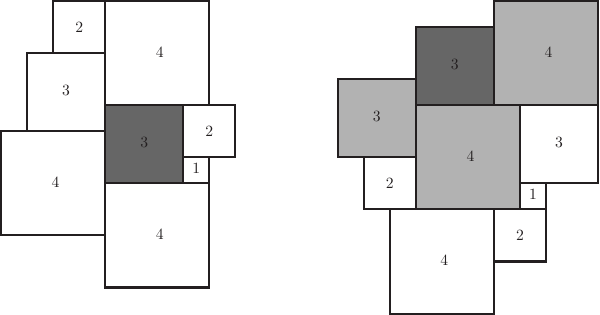}\caption{These coronas are incompatible because the shaded partial 3-corona inside the 4-corona is not present in the 3-corona.}\label{fig:incompatible}
\end{figure}

It is clear from the definition of equitransitivity that in any UE-$n$ tiling, any two coronas (including two congruent coronas) must be compatible. Therefore, the algorithm should remove from consideration any $n$-tuple of coronas $(\mathcal{C}_1, \mathcal{C}_2, \ldots, \mathcal{C}_n)$ (where $\mathcal{C}_i$ is an $i$-corona) containing an incompatible pair of coronas, $\mathcal{C}_i$ and $\mathcal{C}_j$, for such an $n$-tuple of coronas cannot be the coronas of a UE-$n$ tiling. The computer algorithm for filtering out such incompatible tuples of coronas involves several cases based on how the center square of a partial corona can be positioned along a side of a corona and whether or not that partial corona contains 4 squares (as in Figure \ref{fig:partial_a}) or 5 squares (as in Figure \ref{fig:partial-b}). The algorithm for compatibility checking can be examined in Appendix \ref{apx:compat}; we only note here that it is straightforward (if tedious with several cases): Given an $i$-corona and a $j$-corona, to check if they are compatible, the algorithm finds all partial $j$-coronas in the given $i$-corona and determines for each if that partial $j$-corona is present in the given $j$-corona, and vice versa.

In addition to filtering out any incompatible tuples of coronas, the algorithm applies Theorem \ref{thm:no-closed-subsets} to filter out tuples $(\mathcal{C}_1,\mathcal{C}_2,\mathcal{C}_3, \ldots \mathcal{C}_n)$ for which any number $k$ of these coronas with $k<n$ involves only $k$ sizes of squares.

\subsubsection{Summary of Compatible Coronas with Integral Exponents for \texorpdfstring{$n = 4$}{n=4}}
From Table \ref{tab:coronas}, it is seen that there are $15 \times 24 \times 117 \times 180 = 7,581,600$ 4-tuples of coronas with integral exponents to check for compatibility. After doing so (using the afore mentioned computer program in Appendix \ref{apx:compat}), there were only 29 such compatible 4-tuples of coronas with integral exponents, which are given in Table \ref{tab:29-tuples} and in Figures \ref{fig:cor-1} - \ref{fig:cor-29}. Some of these tuples have faulty 3-coronas (rows 23-29) and are discussed in more detail in Subsection \ref{subsec:faulty_tilings}. We have reordered these 4-tuples (compared to the order in which they were produced by the computer algorithm); in particular, rows 17-29 are a reordering of the original computer output which was done to group faulty coronas together by some common characteristics and moving the 4-tuples containing faulty coronas to appear last in the Table \ref{tab:29-tuples}.

\begin{table}[h]
    \centering
    \begin{tabular}{l|l|l|l|l}
    & 1-corona & 2-corona & 3-corona & 4-corona \\\hline
        1 & $2^1;2^1;3^2;4^3$ & $3^1;3^1;4^3.1;2^1.1$ & $3^3.4;1^0.2;3^2.2;4^1$ & $4^4.3.2;1^0.3;4^4.3.2;1^0.3$     \\
        2 & $2^1;2^1;3^2;4^3$ & $3^1;3^1;4^3.1;2^1.1$ & $3^3.4;1^0.2;3^2.2;4^1$ & $4^4.3.1;2^1.3;4^4.3.2;1^0.3$     \\
        3 & $2^1;2^1;3^2;4^3$ & $3^1;4^2;4^3.1;2^1.1$ & $4^0;1^0.2;4^2.3;4^0$ & $3^0.4;3^0.1;2^1.4;2^0.3$     \\
        4 & $2^1;3^2;2^1;3^2$ & $3^1;4^2;4^2;3^2.1$ & $4^2.2;1^0.2;4^2.2;1^0.2$ & $3^1.2;4^2.2;3^1.2;4^2.2$     \\
        5 & $2^1;3^2;3^2;3^2$ & $3^1;4^2;4^2;3^2.1$ & $1^0.3;1^0.2;4^2.2;1^0.3$ & $3^1.2;4^2.2;3^1.2;4^2.2$     \\
        6 & $2^1;3^2;3^2;4^3$ & $3^1;4^2;4^2;4^3.1$ & $1^0.3;1^0.2;4^2.3;4^0$ & $2^0.3;3^0.1;2^1.4;2^0.4$     \\
        7 & $2^1;3^2;3^2;4^3$ & $3^1;4^3.1;3^1;4^3.1$ & $3^3.4;2^0.1;3^1.1;4^1$ & $4^4.3.2;1^0.3;4^4.3.2;1^0.3$     \\
        8 & $2^1;3^2;3^2;4^3$ & $3^1;4^3.1;3^1;4^3.1$ & $3^3.4;2^0.1;3^1.1;4^1$ & $4^4.3.1;2^1.3;4^4.3.2;1^0.3$     \\
        9 & $2^1;3^2;4^3;3^2$ & $3^1;4^2;4^2;3^2.1$ & $4^0;1^0.2;4^2.2;1^0.4$ & $2^0.4;3^0.1;3^1.4;2^0.3$     \\
        10 & $2^1;3^2;4^3;3^2$ & $3^1;4^2;4^2;3^2.1$ & $4^0;1^0.2;4^2.1;2^1.4$ & $2^0.3;1^0.3;4^2.3;2^0.4$     \\
        11 & $2^1;3^2;4^3;3^2$ & $3^1;4^2;4^2;3^2.1$ & $4^0;1^0.2;4^2.1;2^1.4$ & $2^0.3;1^0.3;4^2.2;3^1.4$     \\
        12 & $2^1;3^2;4^3;3^2$ & $3^1;3^2.1;3^1;3^2.1$ & $4^0;1^0.2;3^1.2;1^0.4$ & $4^2.3;1^0.3;4^2.3;1^0.3$     \\
        13 & $2^1;3^2;4^3;4^3$ & $3^1;3^1;4^2;4^3.1$ & $4^0;1^0.2;3^2.4;2^0.3$ & $2^0.3;3^0.1;4^1.2;1^0.4$     \\
        14 & $2^1;3^2;4^3;4^3$ & $3^1;3^1;4^2;4^3.1$ & $4^0;1^0.2;3^2.4;2^0.3$ & $1^0.4;1^0.3;3^1.2;4^1.2$     \\
        15 & $2^1;3^2;4^3;4^3$ & $3^1;3^1;4^2;4^3.1$ & $4^0;1^0.2;3^2.2;4^2.3$ & $2^0.3;3^0.1;4^1.2;1^0.4$     \\
        16 & $2^1;3^2;4^3;4^3$ & $3^1;3^1;4^2;4^3.1$ & $4^0;1^0.2;3^2.2;4^2.3$ & $1^0.4;1^0.3;3^1.2;4^1.2$     \\
        17 & $3^2;3^2;3^2;4^3$ & $4^2;4^2;4^2;4^2$ & $4^1;4^2.1;3^1.1;3^1.1$ & $2^0.4;3^0.1;3^1.4;2^0.4$     \\
        18 & $3^2;3^2;4^3;4^3$ & $3^1;4^2;3^1;4^2$ & $4^1;4^3.2;4^2.1;3^1.1$ & $1^0.4;3^0.1;4^1.3;2^0.3$     \\
        19 & $3^2;3^2;4^3;4^3$ & $3^1;4^2;3^1;4^2$ & $4^1;4^3.2;4^2.1;3^1.1$ & $1^0.4;1^0.3;4^1.3;2^0.3$     \\
        20 & $3^2;4^3;3^2;4^3$ & $3^1;3^1;4^2;4^2$ & $4^1;4^3.2;3^2.2;4^2.1$ & $2^0.3;1^0.3;4^1.3;2^0.4$     \\
        21 & $3^2;4^3;3^2;4^3$ & $3^1;4^2;3^1;4^2$ & $4^1;4^1;4^3.2;4^2.1$ & $2^0.3;1^0.3;4^3.3;4^1.3$     \\
        22 & $3^2;4^3;3^2;4^3$ & $4^2;4^2;4^2;4^2$ & $4^1;4^2.1;4^1;4^2.1$ & $2^0.4;3^0.1;3^1.4;2^0.4$     \\   
        23 & $2^1;3^2;4^3;4^3$ & $3^1;4^3.1;3^1;4^3.1$ & $4^0;2^0.1;4^1;4^3.3$ & $1^0.4;3^0.1;4^1.3;3^0.2$     \\
        24 & $2^1;3^2;4^3;4^3$ & $3^1;4^3.1;3^1;4^3.1$ & $4^0;2^0.1;4^1;3^1.4$ & $1^0.4;1^0.3;3^2.4;3^0.2$     \\
        25 & $2^1;3^2;4^3;4^3$ & $3^1;4^3.1;3^1;4^3.1$ & $4^0;2^0.1;4^1;3^2.4$ & $1^0.4;1^0.3;3^1.4;3^0.2$    \\
        26 & $2^1;3^2;4^3;4^3$ & $3^1;4^3.1;3^1;4^3.1$ & $3^3.4;1^0.2;4^1;4^1$ & $4^4.3.1;4^1.1;2^1.3;4^3.3$     \\
        27 & $2^1;3^2;4^3;4^3$ & $3^1;4^3.1;3^1;4^3.1$ & $4^0;2^0.1;4^1;4^2.3$ & $1^0.4;1^0.3;4^2.3;3^0.2$     \\
        28 & $2^1;3^2;4^3;4^3$ & $3^1;4^3.1;3^1;4^3.1$ & $4^0;2^0.1;4^1;4^3.3$ & $1^0.4;1^0.3;4^1.3;3^0.2$     \\
        29 & $2^1;3^2;4^3;4^3$ & $3^1;4^3.1;3^1;4^3.1$ & $3^3.4;2^0.1;4^1;4^1$ & $4^4.3.4;3^0.1;4^1.1;2^1.3$    \\
    \end{tabular}
    \caption{The 29 4-tuples of Compatible Coronas Generated by Computer Search}
    \label{tab:29-tuples}
\end{table}

\section{Classification of All UE-4 Tilings By Squares of Sizes \texorpdfstring{$1$}{1}, \texorpdfstring{$2$}{2}, \texorpdfstring{$3$}{3}, and \texorpdfstring{$4$}{4}}

The compatible coronas with integral exponents of Table \ref{tab:29-tuples} serve as a baseline for determining all possible UE-4 tilings by squares of sizes $1$, $2$, $3$, and $4$. In this analysis, we find that rows 1-23 of Table \ref{tab:29-tuples} each correspond to a unique UE-4 tiling having exactly those coronas. These coronas can be characterized as all possible 4-tuples of coronas in which the 4-corona is not faulty. We discuss this in Section \ref{subsec:non-faulty}. 

For the coronas of row 24-29, both the 3- and the 4-coronas are faulty. It is found that these coronas are all very closely related and, in fact, may all be thought of as particular instances of a larger phenomenon. In particular, there is a single infinite family $\{\mathcal{T}_y\}$ of UE-4 tilings with $0 < y < 7$ such that each of the 4-tuples of coronas in rows 24-29 describes the coronas of the tilings $\mathcal{T}_y$ for integral values of $y$ in $0 < y < 7$. The tilings in $\{\mathcal{T}_y\}$ are all very similar in appearance and construction, with the difference between any two of them can be described as a uniform shift of rows along fault lines in the tilings. That is, one such tiling can be obtained from the other by simply shifting parallel rows of tiles. Each of rows 24-29 corresponds to coronas seen in $\mathcal{T}_y$ when $y$ is an integer value in $0 < y < 7$. Furthermore, it turns out that any possible 4-tuple of coronas with non-integral exponents in both the 3- and 4-coronas must be the coronas of exactly one tiling in $\{\mathcal{T}_y\}$ with $0 < y < 7$. In this way, we obtain the remainder of the complete set of potential 4-tuples of coronas for UE-4 tilings. Moreover, every such possible 4-tuple of coronas gives rise to a unique tiling \emph{except for one}! In particular, the coronas corresponding to $\mathcal{T}_{7/2}$ give rise to two distinct tilings. We explain this in detail in Section \ref{subsec:faulty_tilings}.

\subsection{Rows 1-23: Existence and Uniqueness of Tilings with These Coronas}\label{subsec:non-faulty}

 Consider any particular row from rows 1-23 of Table \ref{tab:29-tuples} and let $\mathcal{C}_4$ be the 4-corona of that row. It is our observation that any tiling $\mathcal{T}$ containing 4-coronas congruent to $\mathcal{C}_4$ must contain a uniquely determined infinite strip $\mathcal{S}$ of tiles formed entirely from congruent translated copies of $\mathcal{C}_4$. The strip $\mathcal{S}$ may then be expanded in a unique way by completing the coronas of squares on the boundary of the strip (according to what 4-tuple of coronas are under examination) and continuing in this manner until two parallel copies $\mathcal{S}'$ and $\mathcal{S}''$ of $\mathcal{S}$ are completed, one on either side of $\mathcal{S}$. From there, $\mathcal{S}'$ and $\mathcal{S}''$ expand uniquely to form additional parallel strips congruent to $\mathcal{S}$. Because congruent copies of $\mathcal{S}$ are produced and the process of expanding $\mathcal{S}$ was unique, the process repeats itself ad infinitum in a unique way, producing a unique periodic tiling with the coronas used in its construction. One then checks that this particular periodic tiling is equitransitive by checking that any two squares of the same size in a fundamental region of the tiling are in the same transitivity class with respect to the symmetries of the tiling; there are typically only a small number of squares in any such fundamental region, so this checked manually in a straightforward manner. Lastly, unilaterality is assured because the coronas of the 4-tuples of compatible coronas contain no violations of unilaterality by construction.

There are 23 cases to be checked separately, but for the sake of brevity, we will show why this is true of one of these cases, the others being very similar. To that end, consider the 4-tuple $(\mathcal{C}_1,\mathcal{C}_2,\mathcal{C}_3,\mathcal{C}_4)$ of coronas of row 2 of Table \ref{tab:29-tuples} (see Figure \ref{fig:cor-2}). Starting from a single copy of $\mathcal{C}_4$, notice that the coronas of size-2 and size-3 squares within this copy of $\mathcal{C}_4$ can be completed in a unique way so that these completed coronas are congruent to $\mathcal{C}_2$ and $\mathcal{C}_3$ (which is necessary for the tiling to be equitransitive); we have shaded the squares added to complete these coronas at the end of Step 1 in Figure \ref{fig:2-construction-step-1}. Next, referring to Steps 2-5 of Figure \ref{fig:2-construction-step-1}, it is seen that the coronas of the size-4 squares can be completed in a unique way and that these 4-coronas are translates of the original 4-corona. The newly introduced size-4 squares must then be completed using translates of the original 4-corona, and this continues forever, implying that any tiling with these coronas must have an infinite strip congruent to the one depicted at left in Figure \ref{fig:2-construction-step-2}. By inspection, one sees that the 3-coronas of this strip are completed in a unique way that maintains congruence to $\mathcal{C}_3$, and when these 3-coronas are completed, it is seen that the newly introduced size-4 squares can be completed in a unique way, resulting in two new strips of 4-coronas that are congruent to the original strip of 4-coronas (as at right in Figure \ref{fig:2-construction-step-2}). These two new strips of 4-coronas are both rotated through an angle of measure $\pi$ relative to the central strip of 4-coronas. Again, we may now repeat this process of strip building ad infinitum, resulting in a tiling of alternating strips of copies of 4-coronas. Because every step of this process was uniquely determined by the restriction that the coronas of the tiling be congruent to $\mathcal{C}_1$, $\mathcal{C}_2$, $\mathcal{C}_3,$ and $\mathcal{C}_4$, the resulting tiling is unique in having those coronas. 

To see that the tiling is equitransitive, we observe that since the tiling is periodic, there is a finite fundamental region $\mathcal{F}$ with respect to the group of translations of the tiling. It is then checked that each square of the same size in $\mathcal{F}$ is in the same transitivity class with respect to the full symmetry group of the tiling. In Figure \ref{fig:2-fd-copy}, copies of a fundamental region are outlined in bold, and as an example, we have shaded two size-2 squares within a copy of the fundamental region. It is quickly checked by inspection that there is a glide reflection mapping one of the shaded size-2 squares to the other while preserving the tiling. In a similar way, it is checked that the other pair of same size squares are equivalent in the fundamental region, so the tiling as a whole is equitransitive.

\begin{figure}[h]
\centering
\includegraphics[width=.75\textwidth]{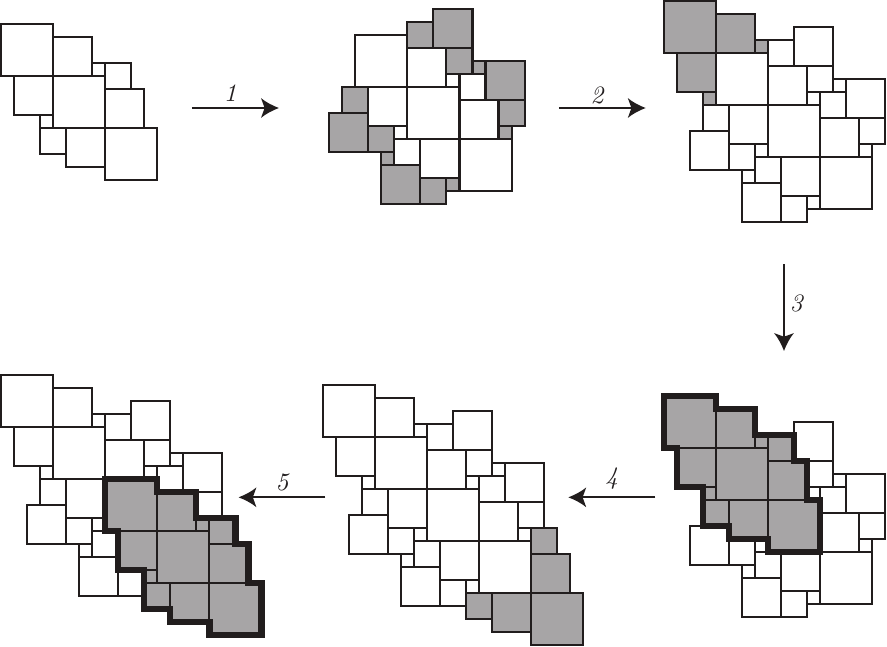}\caption{Construction of the unique tiling with coronas of row 2 of Table \ref{tab:29-tuples}: An infinite strip formed from translated copies of 4-coronas must exist in any tiling with these coronas.}\label{fig:2-construction-step-1}
\end{figure}

\begin{figure}[h]
\centering
\includegraphics[width=.75\textwidth]{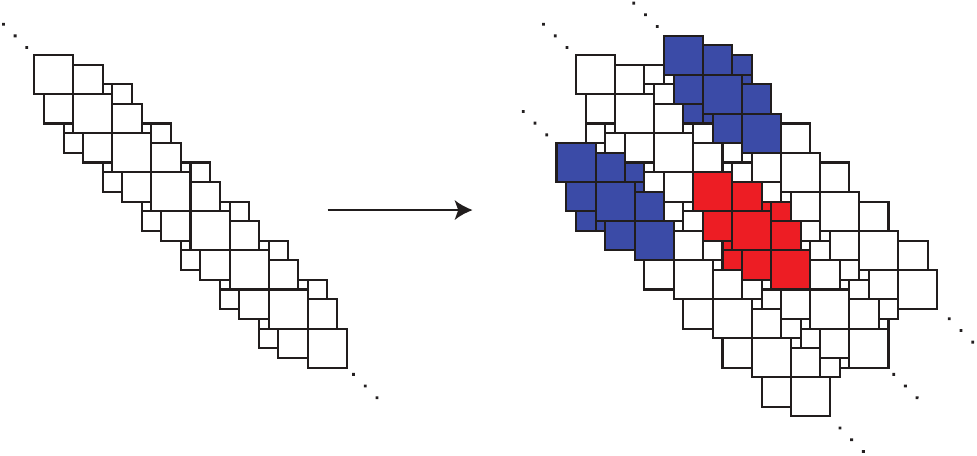}\caption{Construction of the unique tiling with coronas of row 2 of Table \ref{tab:29-tuples}: The infinite strip of translated copies of 4-coronas must expand to include two additional congruent copies of this strip. In this case, the two new strips are rotated through an angle of measure $\pi$ relative to the original central strip (the different rotational orientations of the 4-coronas are indicated by coloring).}\label{fig:2-construction-step-2}
\end{figure}

\begin{figure}[h]
\centering
\includegraphics[width=.4\textwidth]{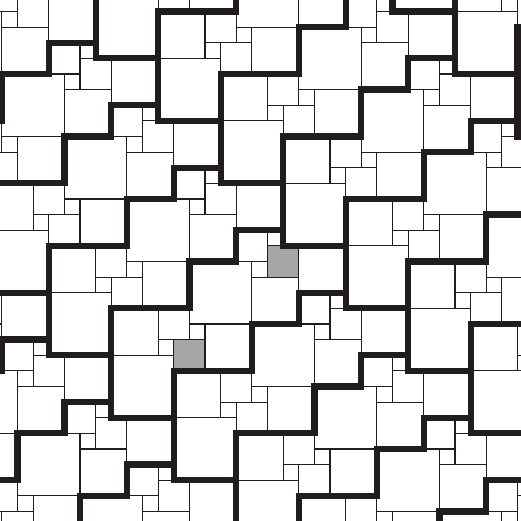}\caption{A portion of the unique UE-4 tiling having the coronas of row 2 of Table \ref{tab:29-tuples} emphasizing a fundamental domain for the group of translations of the tiling. The glide reflection mapping the coronas of shaded size-2 squares is also a symmetry of the tiling.}\label{fig:2-fd-copy}
\end{figure}

\subsection{Rows 24-29 and the Infinite Family \texorpdfstring{$\{\mathcal{T}_y\}$}{Ty} of Faulty UE-4 Tilings}\label{subsec:faulty_tilings}

In rows 24-29 of Table \ref{tab:29-tuples}, it is seen that the 1- and 2-coronas are the same throughout and that the 3- and 4-coronas are both faulty. We will show that these rows are the complete set of instances of 4-tuples with integral exponents in which both the 3- and 4-coronas are faulty. In fact, we will further show that the set of \emph{all} 4-tuples of coronas in which both the 3- and 4-coronas are faulty can be described concisely as exactly those coronas of a family of UE-4 tilings $\{\mathcal{T}_y\}$ that are parametrized by a single real variable $y$ satisfying $0 < y < 7$. In doing so, we will establish the complete set of all 4-tuples of compatible coronas (with both integral and nonintegral exponents), with the computer algorithm providing all the 4-tuples with integral solutions and the argument to follow providing all compatible 4-tuples that have nonintegral exponents. In doing so, we simultaneously find all remaining UE-4 tilings.

Let $\mathcal{C}_3$ be a faulty 3-corona and let $\mathcal{S}$ be a faulty side of $\mathcal{C}_3$. By definition of faulty side, $\mathcal{S}$ does not contain a size-1 or size-2 square, nor may $\mathcal{S}$ be immediately preceded or succeeded by a size-1 or size-2 square. Thus, the squares immediately preceding and succeeding $\mathcal{S}$ must be of size 4, and $\mathcal{S}$ must consist of size 3 or size 4 squares only. It follows that $\mathcal{S}$ be of the form $3^x.4$ ($0 < x \leq 3$), $4^x.3$ ($1 < x < 4$), or $4^x$ ($0 < x \leq 1$). We see that $\mathcal{S}$ along with the size-4 squares immediately preceding and succeeding $\mathcal{S}$  account for three of the four sides of the corona, and the remaining fourth side can only be completed with exactly one size-1 square and one size-2 square. Thus, up to symmetry, all possible faulty 3-coronas are of the following three types (see Figure \ref{fig:faulty_3}): \begin{itemize}
    \item The faulty side is of the form $3^x.4$ and the corona has signature $3^x.4 ; 4^0 ; 2^0.1 ; 4^1$ with $0 < x \leq 3$.
    \item The faulty side is of the form $4^x.3$ and the corona has signature $4^x.3 ; 4^0 ; 2^0.1 ; 4^1$ with $1 < x < 4$.
    \item The faulty side is of the form $4^x$ and the corona has signature  
        \begin{itemize}
            \item  $4^x ; 4^0 ; 2^0.1 ; 4^1$ when $0 < x < 1$ or
            \item  $4^1 ; 3^3.4 ; 2^0.1 ; 4^0$ when $x = 1$.
        \end{itemize}
\end{itemize}

\begin{figure}[h]
\centering
\includegraphics[width=.75\textwidth]{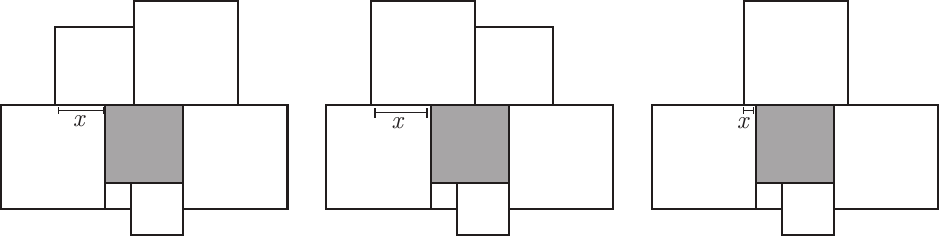}\caption{Faulty 3-coronas}\label{fig:faulty_3}
\end{figure}

Notice that each type of faulty 3-corona has at least one representative with all integer exponents. Now that the faulty 3-coronas are fully described, we point out that the compatibility of a 3-corona with a 1-corona or a 2-corona is not affected by the faulty side of a 3-corona due to the absence of size-1 or size-2 squares along the fault line. Thus, any 1- or 2-corona that is compatible with a faulty 3-corona is compatible with every other faulty 3-corona as well. Moreover, it can be checked that there exist a unique 1-corona and a unique 2-corona that are mutually compatible with any of the three types of faulty 3-coronas; the compatible 1-corona has signature $2^1;3^2;4^3;4^3$ and the 2-corona has signature $3^1;4^3.1;3^1;4^3.1$. It is also seen that all faulty 3-coronas are self-compatible. Lastly, it can be checked that since every faulty 3-corona has a unique compatible 1- and 2-corona, it follows that there is a unique 4-corona that joins them to form a compatible 4-tuple of coronas, with the exception of exactly one faulty 3-corona: For faulty 3-corona with faulty side $4^3.3$, there are exactly two 4-coronas that are compatible, one being faulty and the other being non-faulty. The singular 4-tuple in which the 3-corona is faulty but the 4-corona is non-faulty is seen in row 23 of Table \ref{tab:29-tuples}. For all other 4-tuples of compatible coronas in which the 3-corona is faulty, we see that the 4-corona is also faulty. The 4-tuples of coronas in which the 3-corona is faulty are described in Table \ref{tab:faulty_tuples} (which includes those of rows 23-29 of Table \ref{tab:29-tuples}). 

Next we can consider any 4-tuples of compatible coronas in which the 4-corona is faulty. For a side of a 4-corona to be faulty, the squares preceding and  succeeding this side in the faulty 4-corona must be of size 3. From this observation, it is quickly seen that in order to be compatible with a faulty 4-corona, a 3-corona must also be faulty. Since we have already determined all 4-tuples of coronas in which the 3-corona is faulty, we conclude that Table \ref{tab:faulty_tuples} describes all 4-tuples of compatible coronas in which either the 3- or 4-coronas are faulty. Thus, Tables \ref{tab:29-tuples} and \ref{tab:faulty_tuples} together give a complete accounting of all 4-tuples of compatible coronas for squares of sizes 1, 2, 3, and 4.

\begin{table}[h]
    \centering
    \footnotesize
    \begin{tabular}{c|c|c|c|c|c}
    1-corona & 2-corona & 3-corona & 4-corona(s) & Restrictions & \makecell{Representative \\  in Table \ref{tab:29-tuples}}\\\hline
    $2^1;3^2;4^3;4^3$ & $3^1;4^3.1;3^1;4^3.1$ & $3^x.4 ; 4^0 ; 2^0.1 ; 4^1$ & $4^x .3 ; 3^0.1 ; 4^1.1 ; 2^1.3$ & $0 < x < 3$ & 24 and 25 \\ 
    $2^1;3^2;4^3;4^3$ & $3^1;4^3.1;3^1;4^3.1$ & $3^3.4 ; 4^0 ; 2^0.1 ; 4^1$ & $4^3.3 ; 4^4.3.1 ; 4^1.1 ; 2^1.3$ & $x = 3$   & 26\\\hline
    $2^1;3^2;4^3;4^3$ & $3^1;4^3.1;3^1;4^3.1$ & $4^x.3 ; 4^0 ; 2^0.1 ; 4^1$    & $3^{x-1}.4 ; 3^0.1 ; 4^1.1 ; 2^1.3$    & $1 < x < 4$, $x \neq 3$     & 27\\
    \multirow{3}{*}{$2^1;3^2;4^3;4^3$} & \multirow{3}{*}{$3^1;4^3.1;3^1;4^3.1$} &  \multirow{3}{*}{$4^3.3 ; 4^0 ; 2^0.1 ; 4^1$}   & $3^2.4 ; 1^0.3 ; 4^1.1 ; 2^1.3$     &      & 23\\
    & & & or & $x = 3$  & \\
     & & & $3^2.4 ; 3^0.1 ; 4^1.1 ; 2^1.3$ &  & 28\\\hline
    $2^1;3^2;4^3;4^3$ & $3^1;4^3.1;3^1;4^3.1$ & $4^x ; 4^0 ; 2^0.1 ; 4^1$ & $4^{x+3}.3.4 ; 3^0.1 ; 4^1.1 ; 2^1.3$ & $0 < x < 1$     & \\
    $2^1;3^2;4^3;4^3$ & $3^1;4^3.1;3^1;4^3.1$ & $4^1 ; 3^3.4 ; 2^0.1 ; 4^0$ & $4^{4}.3.4 ; 3^0.1 ; 4^1.1 ; 2^1.3$ & $x = 1$     & 29 \\\hline
     \end{tabular}\caption{4-tuples of coronas containing faulty 3-coronas}\label{tab:faulty_tuples}
\end{table}

Next, we may examine the 4-tuples of compatible coronas in which both the 3- and 4-coronas are faulty to determine which give rise to any UE-4 tilings. What we find is that almost every such 4-tuple gives rise to a unique UE-4 tiling. The only exception to this the 4-tuple \begin{equation}( 2^1;3^2;4^3;4^3\,\, , \,\,3^1;4^3.1;3^1;4^3.1\,\, , \,\,4^{1/2};4^0;2^0.1;4^1\,\, , \,\,4^{7/2}.3.4;3^0.1;4^1.1;2^1.3 ),\label{eqn:special-4-tup}\end{equation} which is a special case of the 4-tuple represented by row 29 of Table \ref{tab:29-tuples} having a 3-corona with faulty side $4^{1/2}$. This 4-tuple gives rise to exactly two distinct tilings, as will be demonstrated later in this section.

All UE-4 tilings having faulty 3- and 4-coronas give rise to an infinite family of tilings that can be parametrized by a single real valued parameter $y$ satisfying $0 < y < 7$. To see how this works, consider a single copy of the 2-corona $\mathcal{C}$ with signature $3^1;4^3.1;3^1;4^3.1$ with points $A$ and $B$ marked at top left in Figure \ref{fig:family-1-3-rows}. Let $C$ be the point on $\mathcal{C}$ located a distance of $y$ horizontally to the left of $B$, then let $\mathbf{v} = \overrightarrow{AC}$. Let $\sigma_{\mathbf{v}}$ be the translation determined by $\mathbf{v}$. Next, consider an infinite row $R$ formed from congruent copies of $\mathcal{C}$, as depicted at top right in Figure \ref{fig:family-1-3-rows}. A UE-4 tiling is then formed by taking the union of the translates $\sigma_{\mathbf{v}}^{i}(R)$ over all integers $i$. If we keep track of the corona signatures of the 3- and 4-coronas as $y$ varies in $0 < y < 7$, the family of tilings $\{\mathcal{T}_y\}$ are in 1-1 correspondence with all possible 4-tuples in which the 3- and 4-coronas are both faulty. Specifically, for $0 < y \leq 3$, the coronas of $\mathcal{T}_y$ exactly those of the first two rows of Table \ref{tab:faulty_tuples} (with $y = x$); for $3 < y \leq 4$, the coronas of $\mathcal{T}_y$ are exactly those of the last two rows of Table \ref{tab:faulty_tuples} (with $x = y-3$); and for $4 < y < 7$, the coronas of $\mathcal{T}_y$ are exactly those of the middle two rows of Table \ref{tab:faulty_tuples}  (with $x = y-3$ and not counting the one instance in which the 4-corona is not faulty there). Note that the integral values of $y$ in $0 < y < 7$ correspond exactly to rows 24-29 of Table \ref{tab:29-tuples}. Also, note that if $y \neq y'$, then $\mathcal{T}_y$ is not equivalent to $\mathcal{T}_{y'}$. 

\begin{figure}[h]
\centering
\includegraphics[width=.75\textwidth]{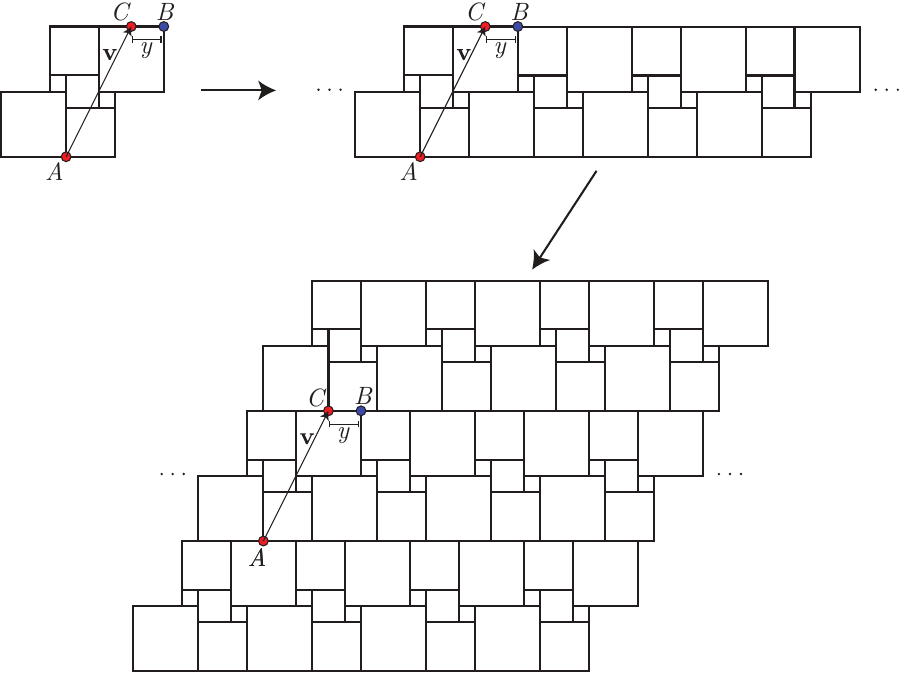}\caption{How tilings $\mathcal{T}_y$ are formed.}\label{fig:family-1-3-rows}
\end{figure}

Within the family of tilings $\{\mathcal{T}_y\}$, the corresponding 4-tuple of coronas for each value of $y$ gives rise to exactly one tiling - namely, $\mathcal{T}_y$, except for one special case: When $y = 7/2$, the corresponding 4-tuple of coronas are exactly those seen in Equation \ref{eqn:special-4-tup}, and for this 4-tuple of coronas, we find that besides the tiling $\mathcal{T}_{7/2}$, these coronas can give rise to one additional tiling. The argument that the 4-tuple of coronas that give rise to $\mathcal{T}_y$ when $y \neq 7/2$ do so uniquely and that $\mathcal{T}_y$ is equitransitive and unilateral is similar to the argument we previously applied to 4-tuple of row 2 of Table \ref{tab:29-tuples}.

To see why the 4-tuple $(\mathcal{C}_1,\mathcal{C}_2,\mathcal{C}_3,\mathcal{C}_4)$ of coronas corresponding to $\mathcal{T}_{7/2}$ (Equation \ref{eqn:special-4-tup}) gives rise to two distinct tilings (one of which is $\mathcal{T}_{7/2}$), start by placing a copy of $\mathcal{C}_4$ in the plane. This copy of $\mathcal{C}_4$ contains three size-3 squares, and we note that the size-3 squares on the left and right side of the 4-corona can be completed in a unique way (while maintaining congruence with $\mathcal{C}_3$), but the size-3 square on the top side of the 4-corona can be completed in two ways - as a directly congruent copy of $\mathcal{C}_3$ and as in indirectly (reflected) copy of $\mathcal{C}_3$. Each of these two choices of orientation for the 3-corona forces any UE-4 tiling with these coronas to have a corresponding infinite row of tiles containing translated copies of $\mathcal{C}_3$ in the same orientation as that original choice, as seen in Figures \ref{fig:17_1_A} and \ref{fig:17_1_B}, where we have colored the size-1 and -2 squares of the 3-corona to emphasize the orientation to the 3-corona. Along the top and bottom sides of these infinite rows there are size-3 squares whose coronas may be completed in two ways. Once a choice is made for orientation of one of the top-side size-3 squares and for one of the bottom-side size-3 squares, the presence of a corresponding infinite row of squares containing translated copies of $\mathcal{C}_3$ in the same orientation is forced, resulting in a taller infinite row with size-3 squares along the top and bottom sides whose coronas can be completed in two ways. In Figure \ref{fig:17_2_AB} we exhibit how one such choice results in a taller row. This process can be continued ad infinitum, and so it may seem that there is an uncountable infinity of distinct UE-4 tilings corresponding to orientation choice for each row. However, most of these tilings are not equitransitive. To see how the equitransitivity constraint restricts the number of distinct tilings, notice that each row produced as discussed above contains arrangements of two size-1 squares and one size-2 square that come in two orientations (shaded red and blue in Figures \ref{fig:17_1_A} - \ref{fig:17_2_AB}). The choice of orientation used for the 3-corona that determines how each row of the tiling is completed also determines the orientation of that arrangement of size 1 and 2 squares. If we label these rows by the orientations of the patches of size-1 and size-2 squares contained within by $A$ and $B$, we see that each unilateral tiling formed from these coronas corresponds to a bi-infinite sequence in the symbols $A$ and $B$. Consideration of this purely symbolic bi-infinite sequence informs the question of whether or not the corresponding 
tiling is also equitransitive. In particular, in order for the corresponding tiling to be equitransitive, it is necessary that the sequence of A's and B's representing the orientations of the rows is transitive with respect to color preserving and color reversing isometries of the sequences (where ``color preserving isometry" is just shift translation of the sequence and ``color reversing isometry" is reflecting the sequence and reversing the symbols $A$ and $B$). It turns out that there are exactly 4 such transitive sequences: \begin{enumerate}
    \item $\ldots, A, A, A, \ldots$,
    \item $\ldots , B, B, B, \ldots$,
    \item $\ldots, A, B, A, B, \ldots$, and 
    \item $\ldots A, A, B, B, A, A, B, B, \ldots$.
\end{enumerate} The two constant sequences are equivalent to each other via reflection, so that leaves exactly three potential UE-4 tilings. By inspection, we determine that tilings corresponding to the sequences $\ldots, A, A, A, \ldots$ and  constant sequence and $\ldots, A, B, A, B, \ldots$ are equitransitive while the tiling corresponding to the sequence $\ldots A, A, B, B, A, A, B, B, \ldots$ is not equitransitive. The two equitransitive tilings are depicted in Figures \ref{fig:24-half} and \ref{fig:25}.

\begin{figure}[h]
\begin{center}
\includegraphics[width=.75\textwidth]{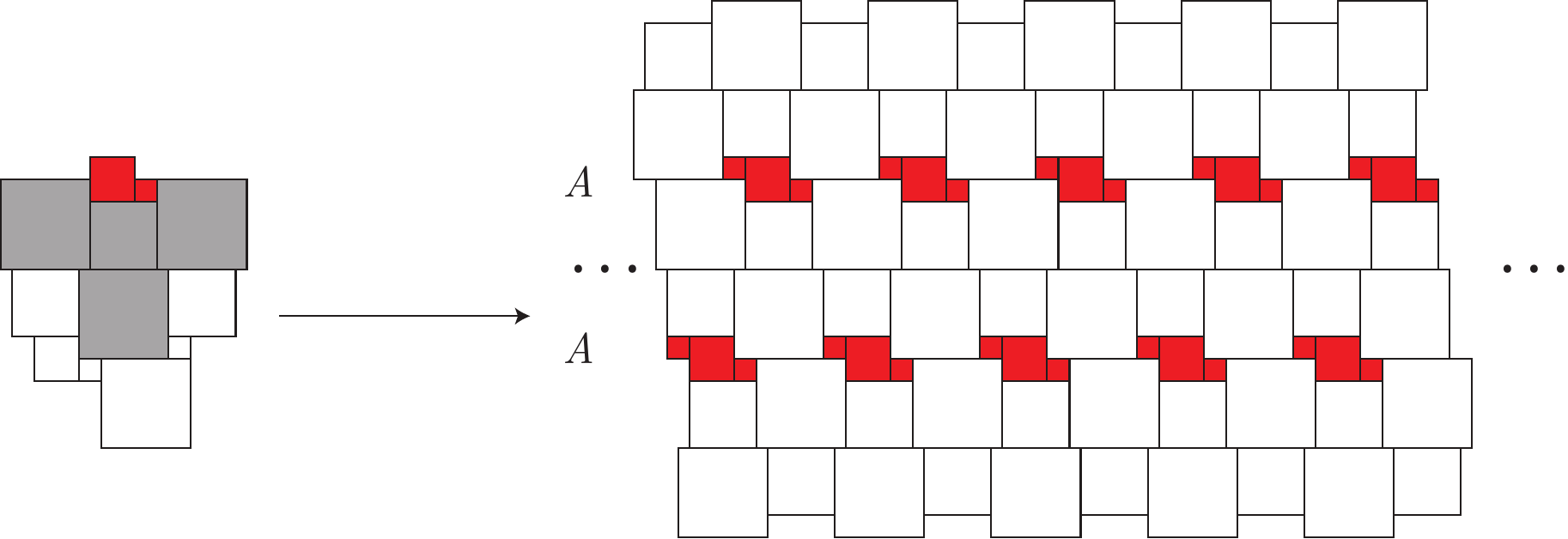}\caption{A choice of orientation the completion of the corona of the size-3 square at the top of a 4-corona (at left) forces an infinite row of tiles in which the arrangement of the size-1 and -2 squares have the same orientation.}\label{fig:17_1_A}
\end{center}
\end{figure}

\begin{figure}[h]
\centering
\includegraphics[width=.75\textwidth]{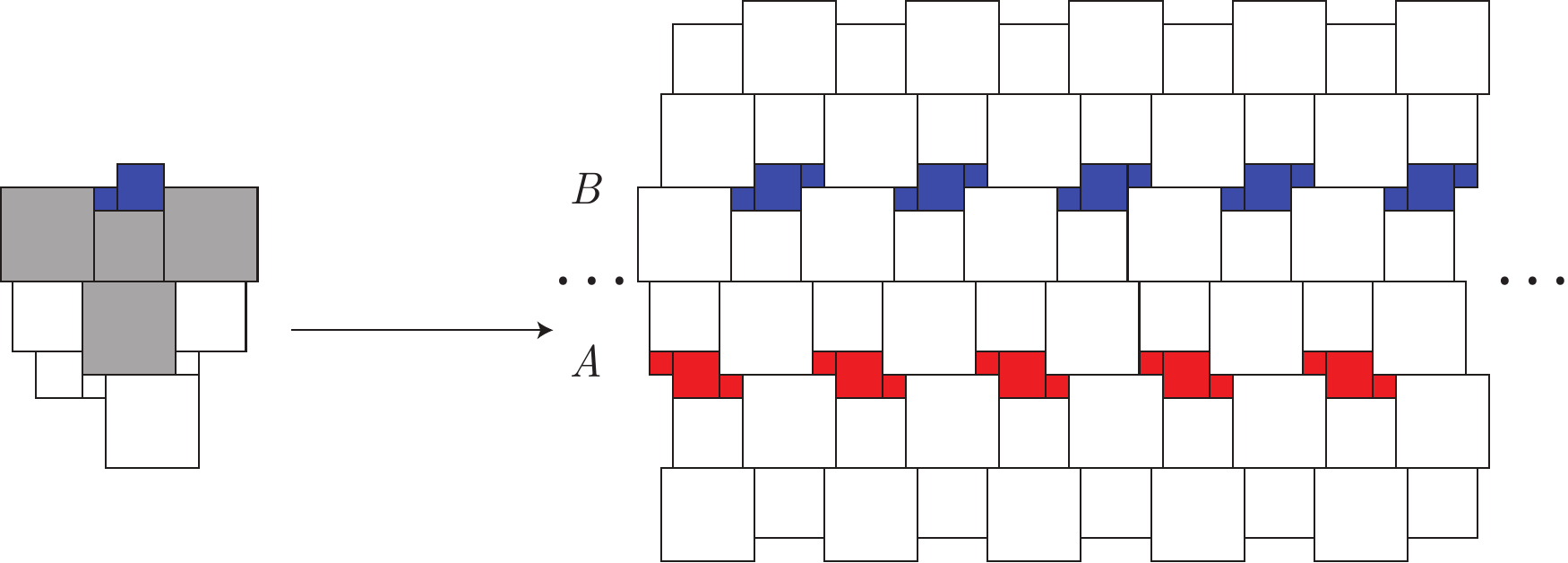}\caption{A choice of orientation the completion of the corona of the size-3 square at the top of a 4-corona (at left) forces an infinite row of tiles in which the arrangement of the size-1 and -2 squares have opposite orientations.}\label{fig:17_1_B}
\end{figure}

\begin{figure}[h]
\centering
\includegraphics[width=.75\textwidth]{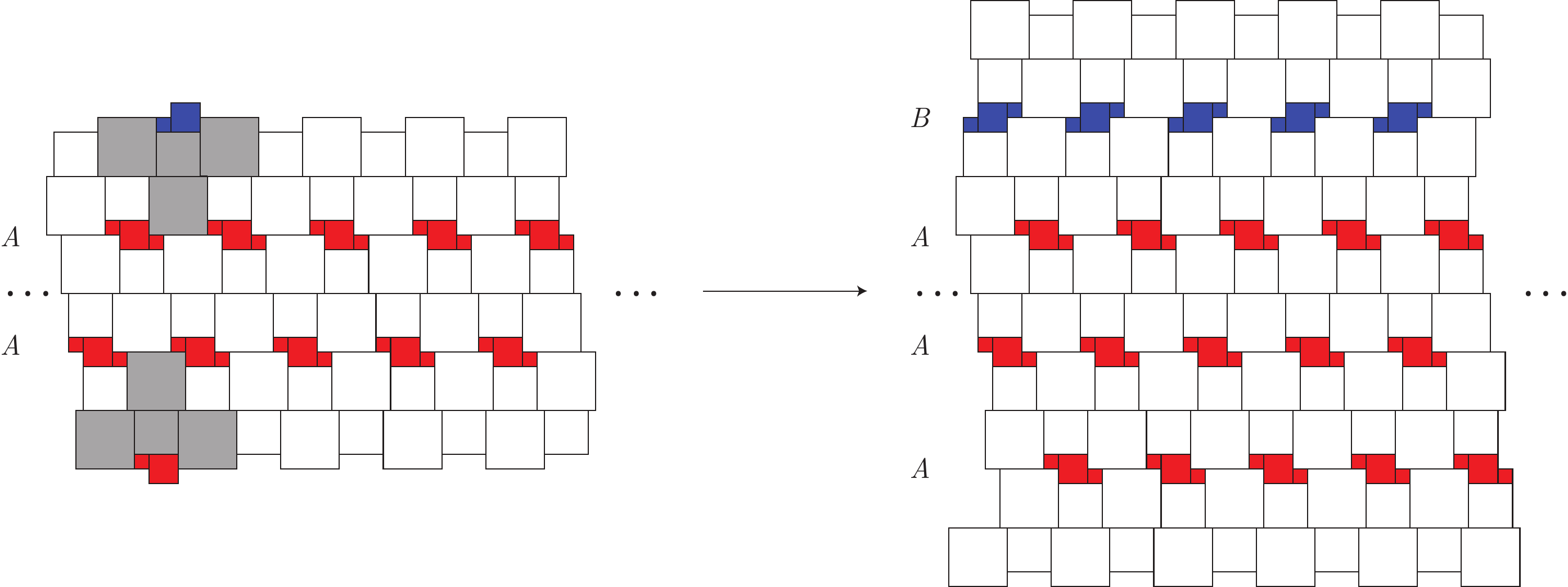}\caption{A choice of orientation for the completion of any one 3-corona on the top side or bottom side of a row forces a vertical extension of the row, adding a new strip with the chosen orientation.}\label{fig:17_2_AB}
\end{figure}

\subsection{Main Theorem}\label{subsec:main}

We have established that all possible 4-tuples of compatible coronas for UE-4 tilings by squares of sizes 1, 2, 3, and 4 are given in Tables \ref{tab:29-tuples} and \ref{tab:faulty_tuples}, and we have established that every such 4-tuple except one gives rise to a unique UE-4 tiling, and that the exceptional 4-tuple of Equation \ref{eqn:special-4-tup} gives rise to exactly three distinct UE-4 tilings.  We summarize these results as the main theorem classifying all UE-4 tilings by squares of sizes 1, 2, 3, and 4:

\begin{theorem}\label{thm:main}
    Let $\mathcal{T}$ be UE-4 tiling by squares of sizes 1, 2, 3, and 4. Then exactly one of the following holds:
    \begin{enumerate}
        \item $\mathcal{T}$ is the unique tiling whose 1-, 2-, 3-, and 4-coronas have signatures corresponding to one of rows 1-23 of Table \ref{tab:29-tuples} (these are the rows in which the 4-corona is not faulty). See Figures \ref{fig:1} - \ref{fig:23}.
        \item $\mathcal{T} = \mathcal{T}_y$ for exactly one $y$ satisfying $0 < y < 7$. A representative tiling in the family $\{\mathcal{T}_y\}$ is seen in Figure \ref{fig:24}. In the family $\{\mathcal{T}_y\}$, both the 3- and 4-coronas are faulty.
        \item $\mathcal{T}$ is the tiling depicted in Figure \ref{fig:25}. The corona signature for this tiling is the same as that of $\mathcal{T}_{7/2}$ (Equation \ref{eqn:special-4-tup}). 
    \end{enumerate}
    Thus, there are 25 types of UET-4 tilings, one of which is an infinite family in which any two tilings are row-shifted versions of each other.
\end{theorem}

\section{Toward Classifying UE-5 Tilings by Squares of Sizes 1, 2, 3, 4, and 5}

The computer algorithm provided in Appendix \ref{apx:code} can be used to generate all viable 1-, 2-, ..., and 5-coronas. Doing so, the following numbers of viable coronas are found: \begin{itemize}
        \item 1-coronas: 39
        \item 2-coronas: 70
        \item 3-coronas: 383
        \item 4-coronas: 1282
        \item 5-coronas: 4393
        \item Total 5-tuples of coronas: 5,888,580,947,340
    \end{itemize}
    
For sizes 1, 2, 3, and 4, the total number of 4-tuples of coronas was a mere 7,581,600. Thus, the number of tuples to check for compatibility increases by a factor of about 77,669 when moving from four sizes of squares to five. Checking compatibility for four sizes took about 5 minutes to process on a standard desktop computer. Thus, on the same computer it would take at least 269 days to process on a single computer, but in reality it will take longer than that, because the maximum number of compatibility pairs to be checked for each tuple grows from 16 to 25. That said, a computing cluster with 1000 CPUs could determine the set of compatible 5-tuples of compatible coronas in about a day, so extending the problem to five sizes of squares is attainable. The subsequent problem of determining which of these 5-tuples give rise to UE-5 tilings, and which are actually just representatives of infinite families like the family $\{\mathcal{T}_y\}$ discussed in Section \ref{subsec:faulty_tilings} may be more challenging as there is no known way to automate this part of the process.

\section*{Acknowledgments}
The authors would like to thank Doris Schattschneider for her helpful comments on an early draft of this article. This work was partially supported by NSF grant \#2150511.

\clearpage

\begin{appendices}

\section{Figures of the 25 UE-4 Tilings by Squares of Sizes \texorpdfstring{$1$}{1}, \texorpdfstring{$2$}{2}, \texorpdfstring{$3$}{3}, and \texorpdfstring{$4$}{4}}\label{apx:all-tilings}

\renewcommand{\thefigure}{T\arabic{figure}}
\setcounter{figure}{0}

\begin{figure}[H]
\centering
\begin{subfigure}{0.45\textwidth}
\centering
\includegraphics[width=.75\textwidth]{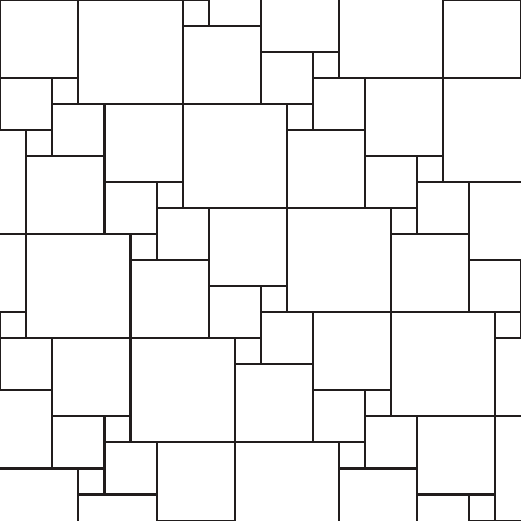}\caption{}\label{fig:1-plain}
\end{subfigure}
\begin{subfigure}{0.45\textwidth}
\centering
\includegraphics[width=.75\textwidth]{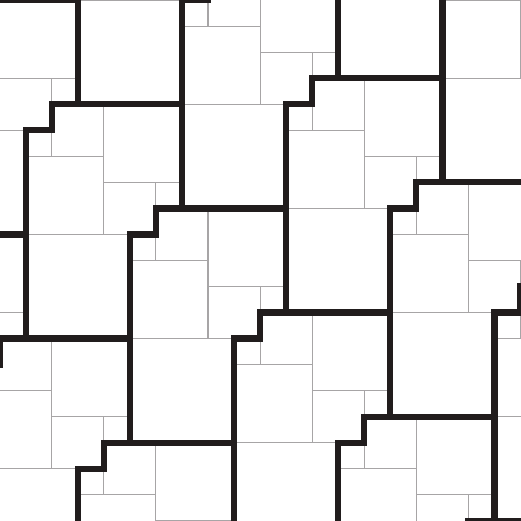}\caption{}\label{fig:1-fd}
\end{subfigure}
\caption{$( 2^1;2^1;3^2;4^3\,\, , \,\,3^1;3^1;4^3.1;2^1.1\,\, , \,\,3^3.4;1^0.2;3^2.2;4^1\,\, , \,\,4^4.3.2;1^0.3;4^4.3.2;1^0.3 )$}\label{fig:1}
\end{figure}

\begin{figure}[H]
\centering
\begin{subfigure}{0.45\textwidth}
\centering
\includegraphics[width=.75\textwidth]{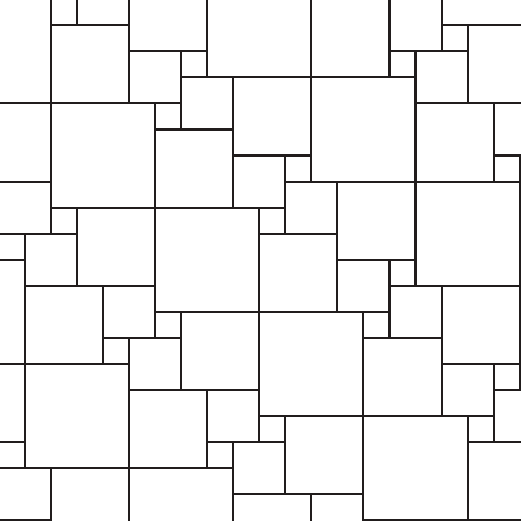}\caption{}\label{fig:2-plain}
\end{subfigure}
\begin{subfigure}{0.45\textwidth}
\centering
\includegraphics[width=.75\textwidth]{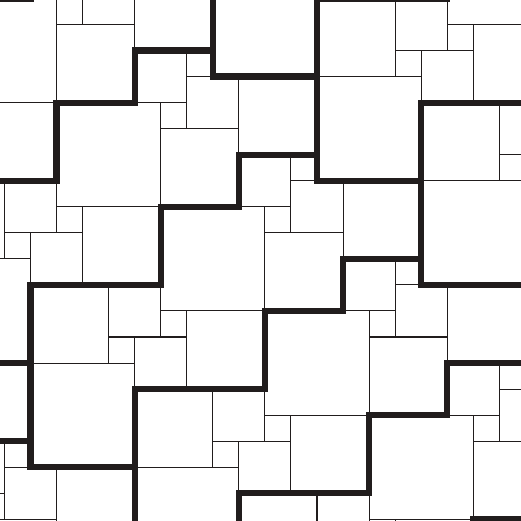}\caption{}\label{fig:2-fd}
\end{subfigure}
\caption{$( 2^1;2^1;3^2;4^3\,\, , \,\,3^1;3^1;4^3.1;2^1.1\,\, , \,\,3^3.4;1^0.2;3^2.2;4^1\,\, , \,\,4^4.3.1;2^1.3;4^4.3.2;1^0.3 )$}\label{fig:2}
\end{figure}

\begin{figure}[H]
\centering
\begin{subfigure}{0.45\textwidth}
\centering
\includegraphics[width=.75\textwidth]{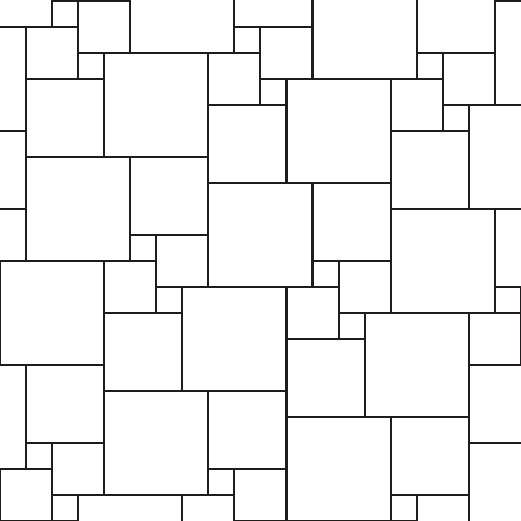}\caption{}\label{fig:3-plain}
\end{subfigure}
\begin{subfigure}{0.45\textwidth}
\centering
\includegraphics[width=.75\textwidth]{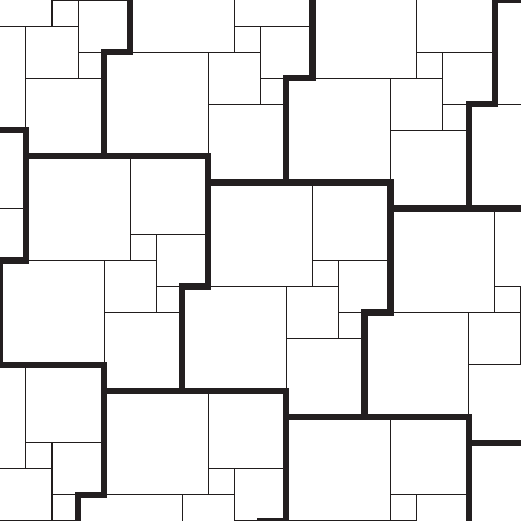}\caption{}\label{fig:3-fd}
\end{subfigure}
\caption{$( 2^1;2^1;3^2;4^3\,\, , \,\,3^1;4^2;4^3.1;2^1.1\,\, , \,\,4^0;1^0.2;4^2.3;4^0\,\, , \,\,3^0.4;3^0.1;2^1.4;2^0.3 )$}\label{fig:3}
\end{figure}

\begin{figure}[H]
\centering
\begin{subfigure}{0.45\textwidth}
\centering
\includegraphics[width=.75\textwidth]{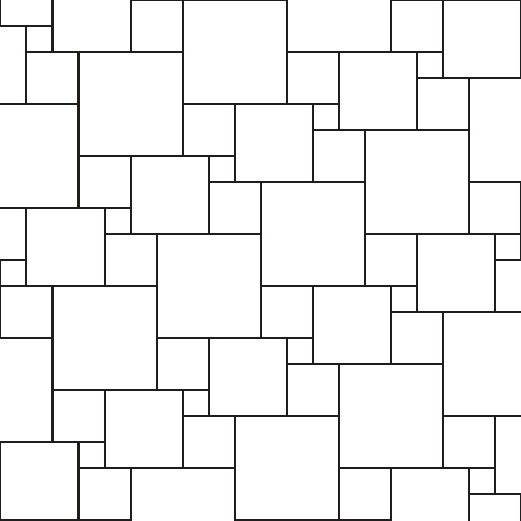}\caption{}\label{fig:4-plain}
\end{subfigure}
\begin{subfigure}{0.45\textwidth}
\centering
\includegraphics[width=.75\textwidth]{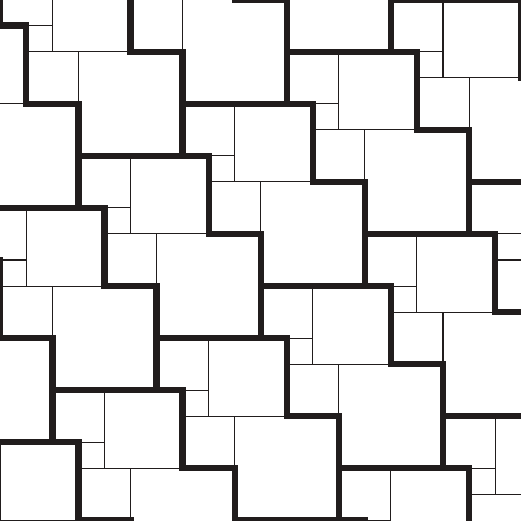}\caption{}\label{fig:4-fd}
\end{subfigure}
\caption{$( 2^1;3^2;2^1;3^2\,\, , \,\,3^1;4^2;4^2;3^2.1\,\, , \,\,4^2.2;1^0.2;4^2.2;1^0.2\,\, , \,\,3^1.2;4^2.2;3^1.2;4^2.2 )$}\label{fig:4}
\end{figure}

\begin{figure}[H]
\centering
\begin{subfigure}{0.45\textwidth}
\centering
\includegraphics[width=.75\textwidth]{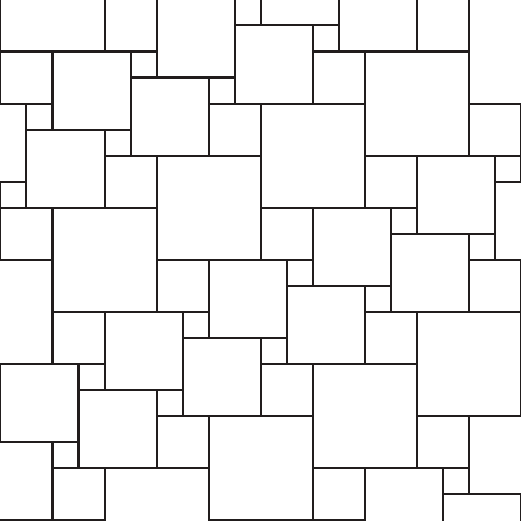}\caption{}\label{fig:5-plain}
\end{subfigure}
\begin{subfigure}{0.45\textwidth}
\centering
\includegraphics[width=.75\textwidth]{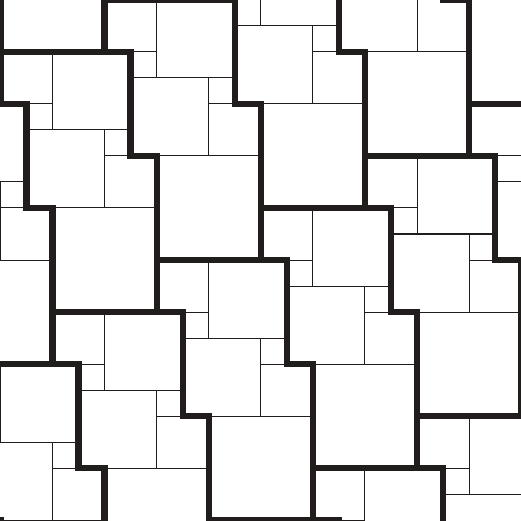}\caption{}\label{fig:5-fd}
\end{subfigure}
\caption{$( 2^1;3^2;3^2;3^2\,\, , \,\,3^1;4^2;4^2;3^2.1\,\, , \,\,1^0.3;1^0.2;4^2.2;1^0.3\,\, , \,\,3^1.2;4^2.2;3^1.2;4^2.2 )$}\label{fig:5}
\end{figure}

\begin{figure}[H]
\centering
\begin{subfigure}{0.45\textwidth}
\centering
\includegraphics[width=.75\textwidth]{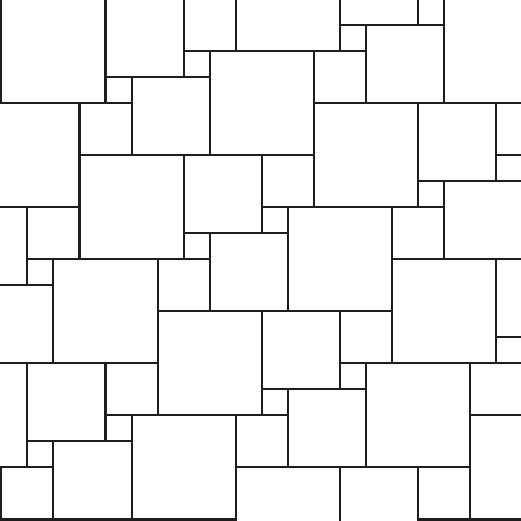}\caption{}\label{fig:6-plain}
\end{subfigure}
\begin{subfigure}{0.45\textwidth}
\centering
\includegraphics[width=.75\textwidth]{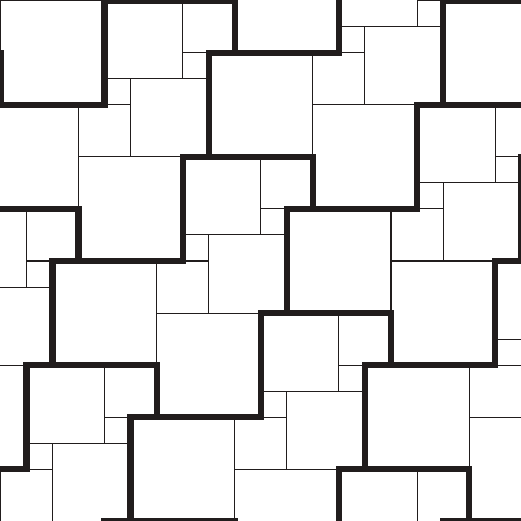}\caption{}\label{fig:6-fd}
\end{subfigure}
\caption{$( 2^1;3^2;3^2;4^3\,\, , \,\,3^1;4^2;4^2;4^3.1\,\, , \,\,1^0.3;1^0.2;4^2.3;4^0\,\, , \,\,2^0.3;3^0.1;2^1.4;2^0.4 )$}\label{fig:6}
\end{figure}

\begin{figure}[H]
\centering
\begin{subfigure}{0.45\textwidth}
\centering
\includegraphics[width=.75\textwidth]{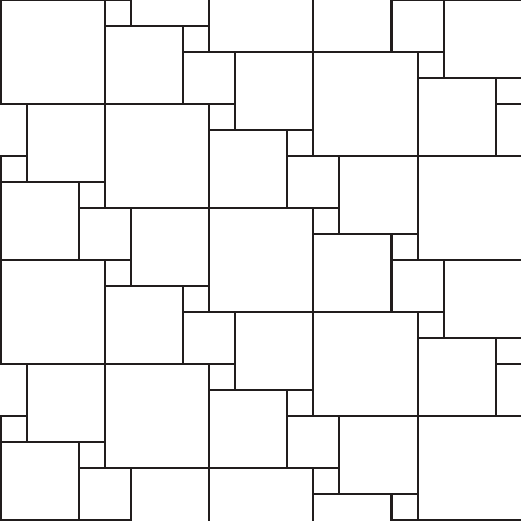}\caption{}\label{fig:7-plain}
\end{subfigure}
\begin{subfigure}{0.45\textwidth}
\centering
\includegraphics[width=.75\textwidth]{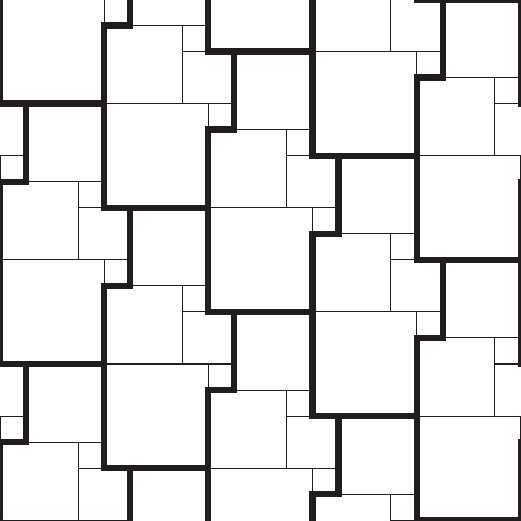}\caption{}\label{fig:7-fd}
\end{subfigure}
\caption{$( 2^1;3^2;3^2;4^3\,\, , \,\,3^1;4^3.1;3^1;4^3.1\,\, , \,\,3^3.4;2^0.1;3^1.1;4^1\,\, , \,\,4^4.3.2;1^0.3;4^4.3.2;1^0.3 )$}\label{fig:7}
\end{figure}

\begin{figure}[H]
\centering
\begin{subfigure}{0.45\textwidth}
\centering
\includegraphics[width=.75\textwidth]{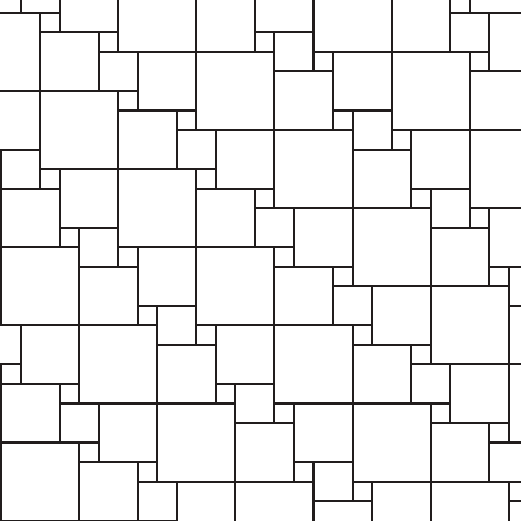}\caption{}\label{fig:8-plain}
\end{subfigure}
\begin{subfigure}{0.45\textwidth}
\centering
\includegraphics[width=.75\textwidth]{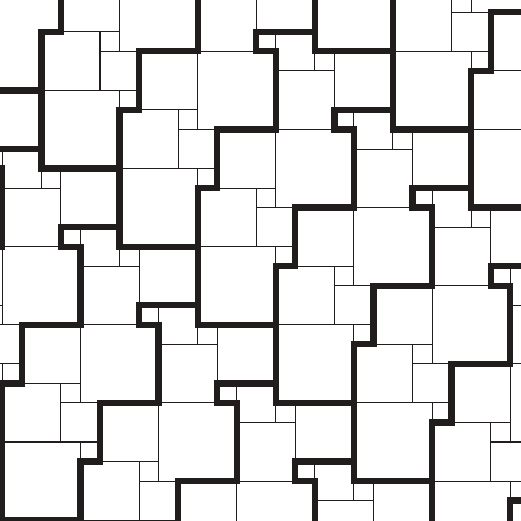}\caption{}\label{fig:8-fd}
\end{subfigure}
\caption{$( 2^1;3^2;3^2;4^3\,\, , \,\,3^1;4^3.1;3^1;4^3.1\,\, , \,\,3^3.4;2^0.1;3^1.1;4^1\,\, , \,\,4^4.3.1;2^1.3;4^4.3.2;1^0.3 )$}\label{fig:8}
\end{figure}

\begin{figure}[H]
\centering
\begin{subfigure}{0.45\textwidth}
\centering
\includegraphics[width=.75\textwidth]{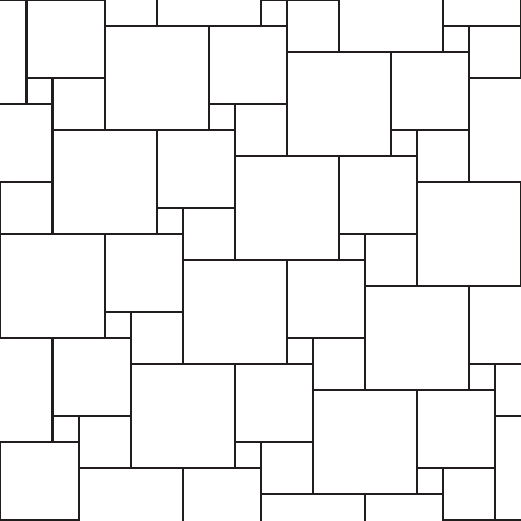}\caption{}\label{fig:9-plain}
\end{subfigure}
\begin{subfigure}{0.45\textwidth}
\centering
\includegraphics[width=.75\textwidth]{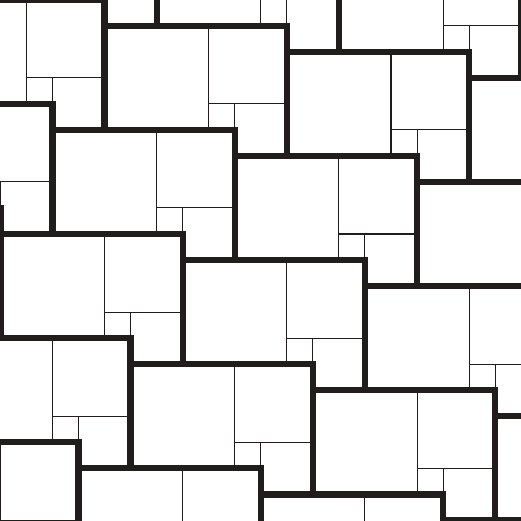}\caption{}\label{fig:9-fd}
\end{subfigure}
\caption{$( 2^1;3^2;4^3;3^2\,\, , \,\,3^1;4^2;4^2;3^2.1\,\, , \,\,4^0;1^0.2;4^2.2;1^0.4\,\, , \,\,2^0.4;3^0.1;3^1.4;2^0.3 )$}\label{fig:9}
\end{figure}

\begin{figure}[H]
\centering
\begin{subfigure}{0.45\textwidth}
\centering
\includegraphics[width=.75\textwidth]{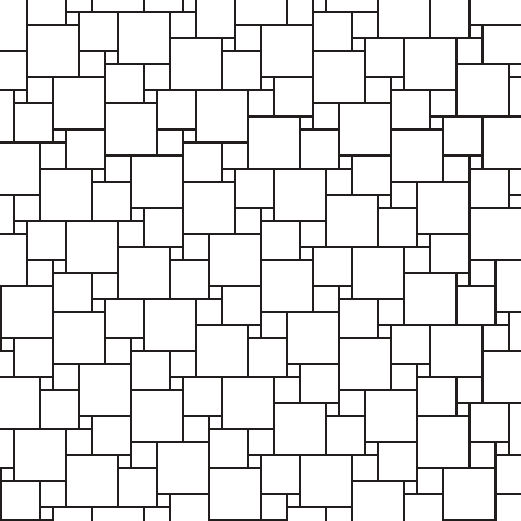}\caption{}\label{fig:10-plain}
\end{subfigure}
\begin{subfigure}{0.45\textwidth}
\centering
\includegraphics[width=.75\textwidth]{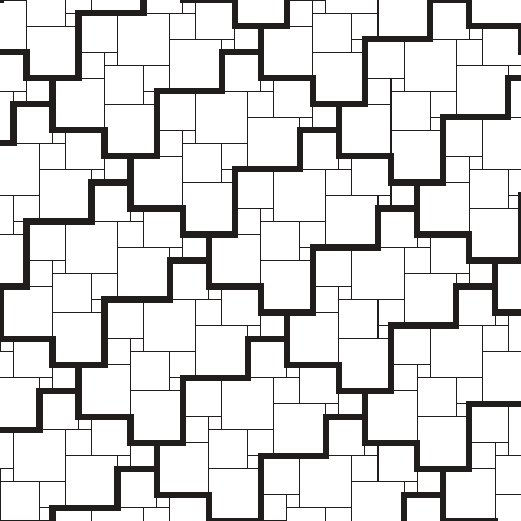}\caption{}\label{fig:10-fd}
\end{subfigure}
\caption{$( 2^1;3^2;4^3;3^2\,\, , \,\,3^1;4^2;4^2;3^2.1\,\, , \,\,4^0;1^0.2;4^2.1;2^1.4\,\, , \,\,2^0.3;1^0.3;4^2.3;2^0.4 )$}\label{fig:10}
\end{figure}

\begin{figure}[H]
\centering
\begin{subfigure}{0.45\textwidth}
\centering
\includegraphics[width=.75\textwidth]{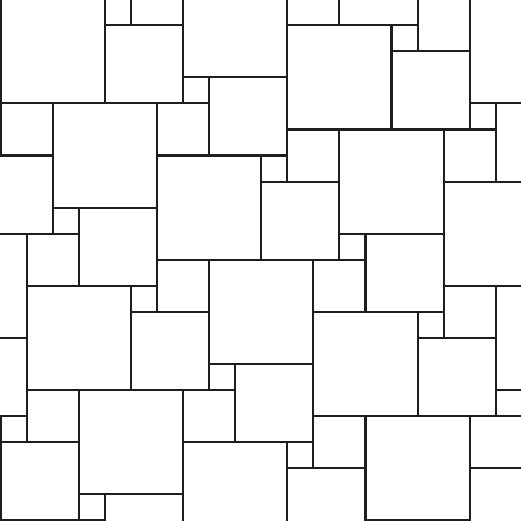}\caption{}\label{fig:11-plain}
\end{subfigure}
\begin{subfigure}{0.45\textwidth}
\centering
\includegraphics[width=.75\textwidth]{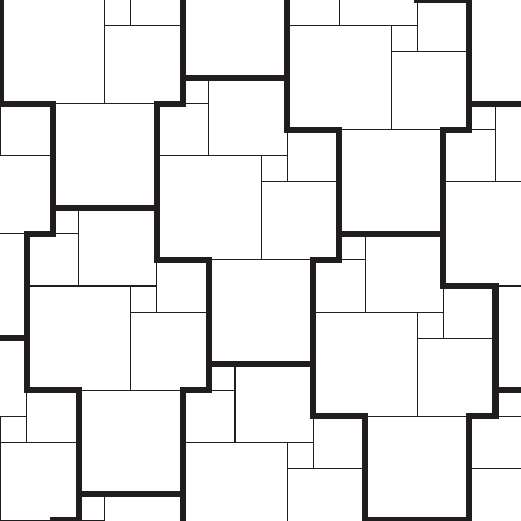}\caption{}\label{fig:11-fd}
\end{subfigure}
\caption{$( 2^1;3^2;4^3;3^2\,\, , \,\,3^1;4^2;4^2;3^2.1\,\, , \,\,4^0;1^0.2;4^2.1;2^1.4\,\, , \,\,2^0.3;1^0.3;4^2.2;3^1.4 )$}\label{fig:11}
\end{figure}

\begin{figure}[H]
\centering
\begin{subfigure}{0.45\textwidth}
\centering
\includegraphics[width=.75\textwidth]{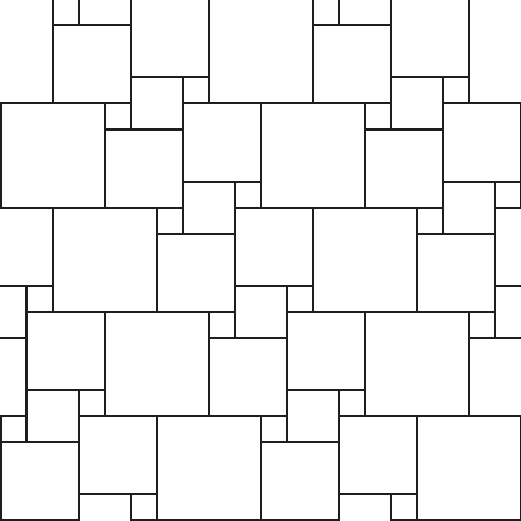}\caption{}\label{fig:12-plain}
\end{subfigure}
\begin{subfigure}{0.45\textwidth}
\centering
\includegraphics[width=.75\textwidth]{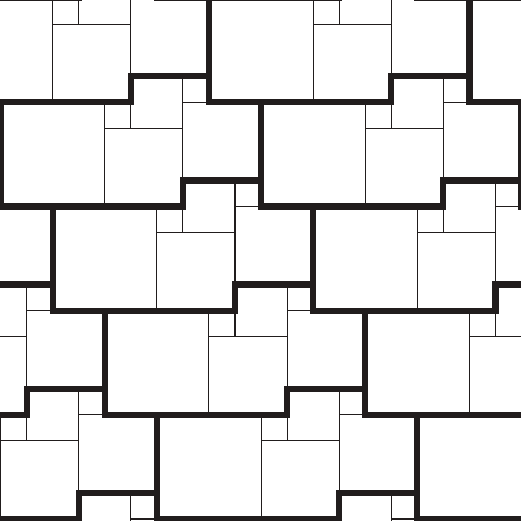}\caption{}\label{fig:12-fd}
\end{subfigure}
\caption{$( 2^1;3^2;4^3;3^2\,\, , \,\,3^1;3^2.1;3^1;3^2.1\,\, , \,\,4^0;1^0.2;3^1.2;1^0.4\,\, , \,\,4^2.3;1^0.3;4^2.3;1^0.3 )$}\label{fig:12}
\end{figure}

\begin{figure}[H]
\centering
\begin{subfigure}{0.45\textwidth}
\centering
\includegraphics[width=.75\textwidth]{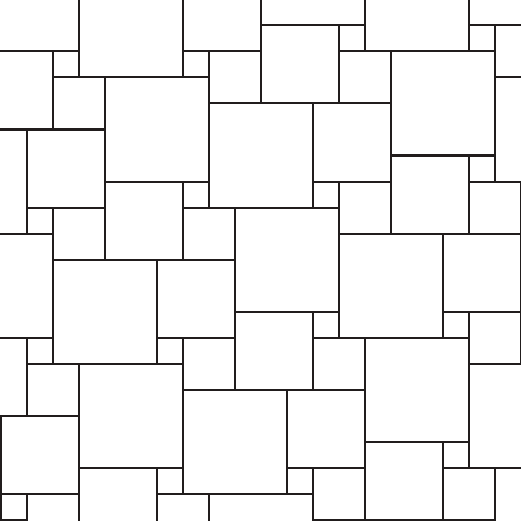}\caption{}\label{fig:13-plain}
\end{subfigure}
\begin{subfigure}{0.45\textwidth}
\centering
\includegraphics[width=.75\textwidth]{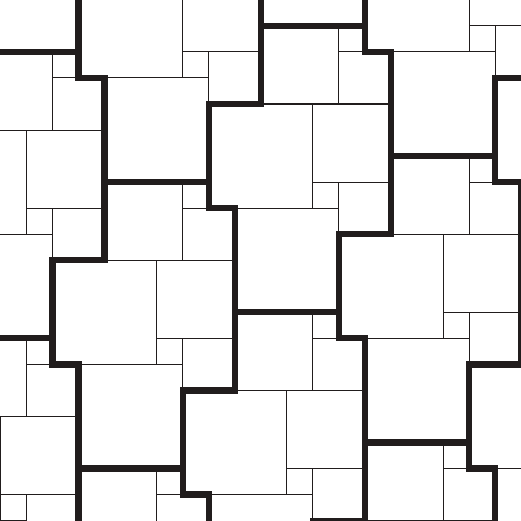}\caption{}\label{fig:13-fd}
\end{subfigure}
\caption{$( 2^1;3^2;4^3;4^3\,\, , \,\,3^1;3^1;4^2;4^3.1\,\, , \,\,4^0;1^0.2;3^2.4;2^0.3\,\, , \,\,2^0.3;3^0.1;4^1.2;1^0.4 )$}\label{fig:13}
\end{figure}

\begin{figure}[H]
\centering
\begin{subfigure}{0.45\textwidth}
\centering
\includegraphics[width=.75\textwidth]{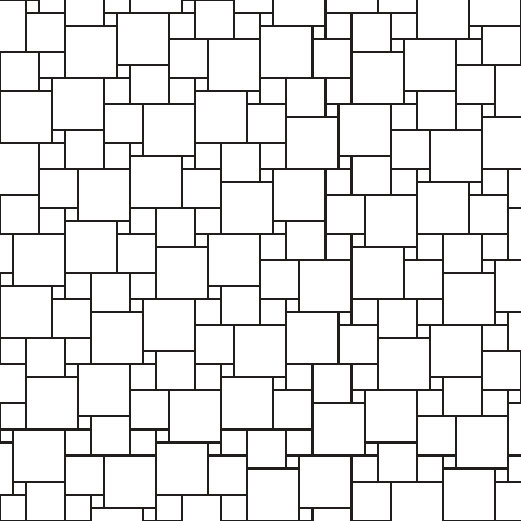}\caption{}\label{fig:14-plain}
\end{subfigure}
\begin{subfigure}{0.45\textwidth}
\centering
\includegraphics[width=.75\textwidth]{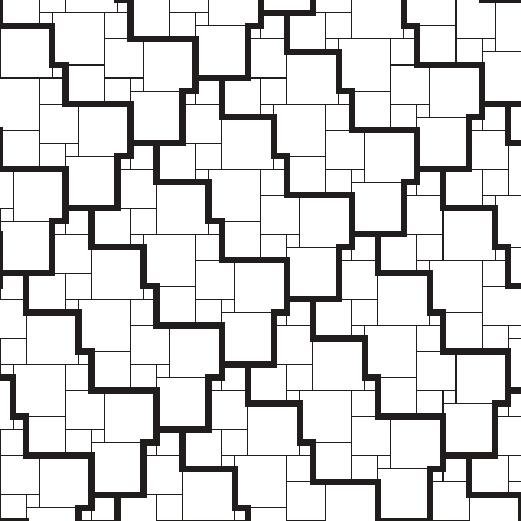}\caption{}\label{fig:14-fd}
\end{subfigure}
\caption{$( 2^1;3^2;4^3;4^3\,\, , \,\,3^1;3^1;4^2;4^3.1\,\, , \,\,4^0;1^0.2;3^2.4;2^0.3\,\, , \,\,1^0.4;1^0.3;3^1.2;4^1.2 )$}\label{fig:14}
\end{figure}

\begin{figure}[H]
\centering
\begin{subfigure}{0.45\textwidth}
\centering
\includegraphics[width=.75\textwidth]{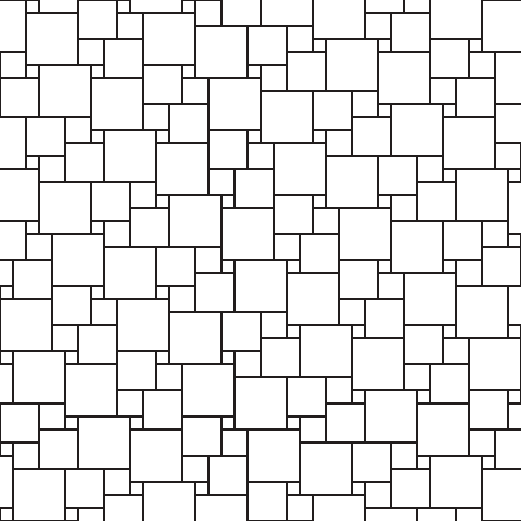}\caption{}\label{fig:15-plain}
\end{subfigure}
\begin{subfigure}{0.45\textwidth}
\centering
\includegraphics[width=.75\textwidth]{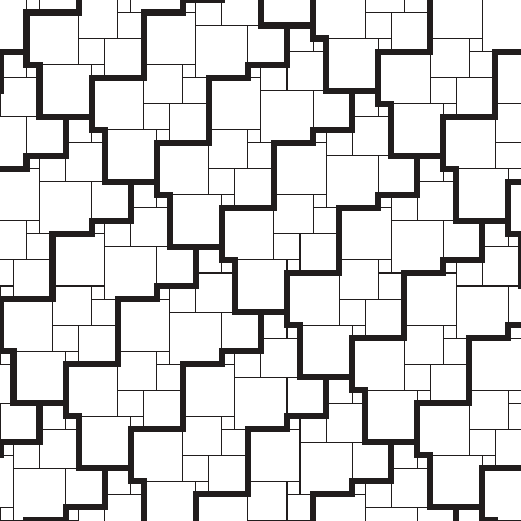}\caption{}\label{fig:15-fd}
\end{subfigure}
\caption{$( 2^1;3^2;4^3;4^3\,\, , \,\,3^1;3^1;4^2;4^3.1\,\, , \,\,4^0;1^0.2;3^2.2;4^2.3\,\, , \,\,2^0.3;3^0.1;4^1.2;1^0.4 )$}\label{fig:15}
\end{figure}

\begin{figure}[H]
\centering
\begin{subfigure}{0.45\textwidth}
\centering
\includegraphics[width=.75\textwidth]{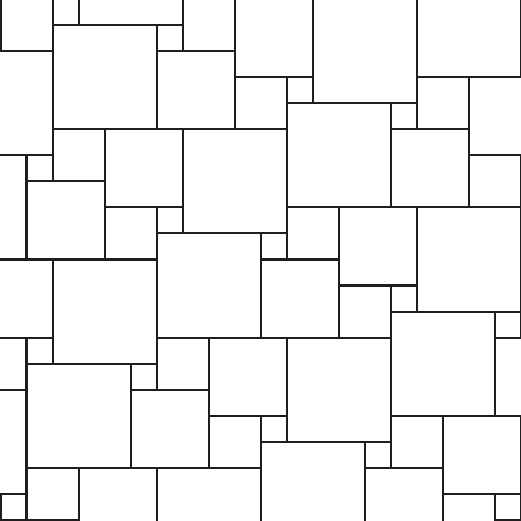}\caption{}\label{fig:16-plain}
\end{subfigure}
\begin{subfigure}{0.45\textwidth}
\centering
\includegraphics[width=.75\textwidth]{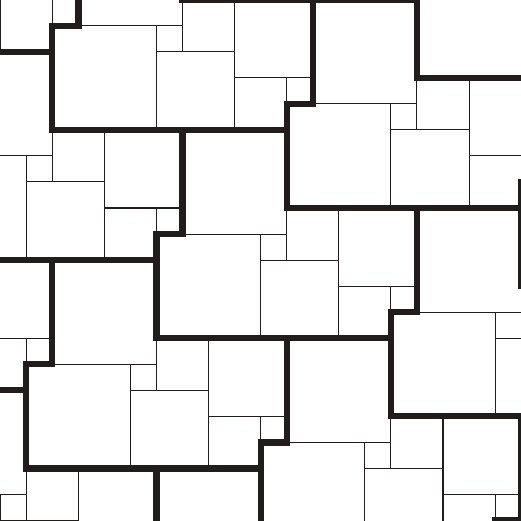}\caption{}\label{fig:16-fd}
\end{subfigure}
\caption{$( 2^1;3^2;4^3;4^3\,\, , \,\,3^1;3^1;4^2;4^3.1\,\, , \,\,4^0;1^0.2;3^2.2;4^2.3\,\, , \,\,1^0.4;1^0.3;3^1.2;4^1.2 )$}\label{fig:16}
\end{figure}

\begin{figure}[H]
\centering
\begin{subfigure}{0.45\textwidth}
\centering
\includegraphics[width=.75\textwidth]{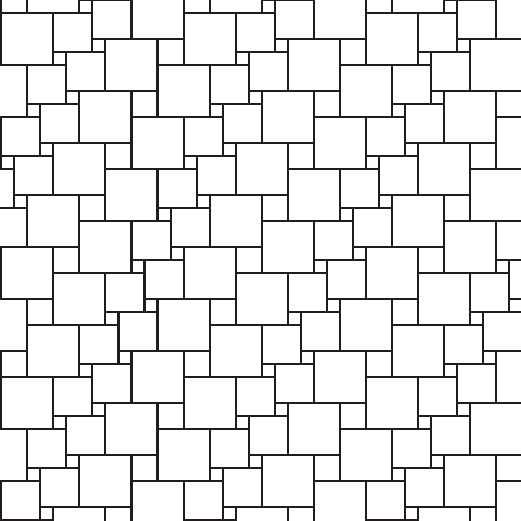}\caption{}\label{fig:17-plain}
\end{subfigure}
\begin{subfigure}{0.45\textwidth}
\centering
\includegraphics[width=.75\textwidth]{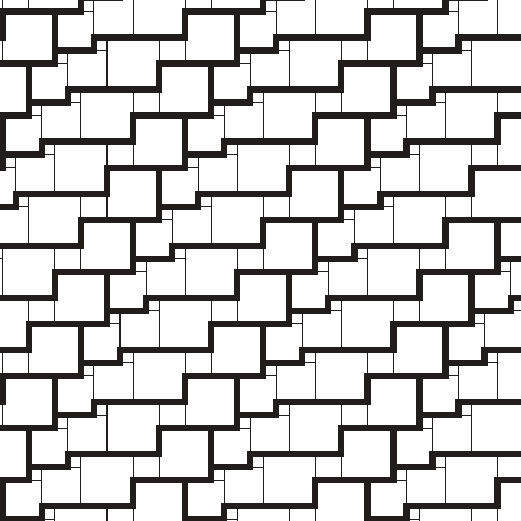}\caption{}\label{fig:17-fd}
\end{subfigure}
\caption{$( 3^2;3^2;3^2;4^3\,\, , \,\,4^2;4^2;4^2;4^2\,\, , \,\,4^1;4^2.1;3^1.1;3^1.1\,\, , \,\,2^0.4;3^0.1;3^1.4;2^0.4 )$}\label{fig:17}
\end{figure}

\begin{figure}[H]
\centering
\begin{subfigure}{0.45\textwidth}
\centering
\includegraphics[width=.75\textwidth]{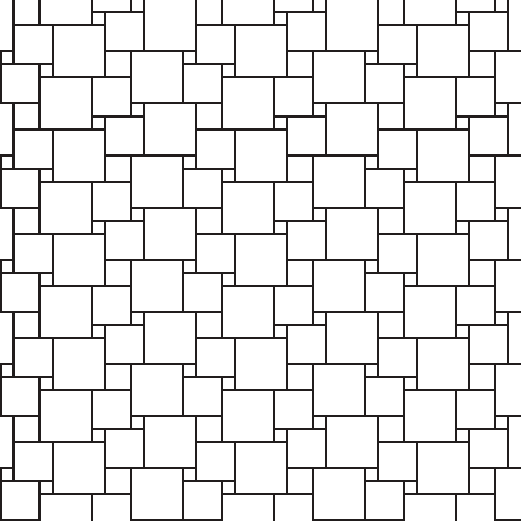}\caption{}\label{fig:18-plain}
\end{subfigure}
\begin{subfigure}{0.45\textwidth}
\centering
\includegraphics[width=.75\textwidth]{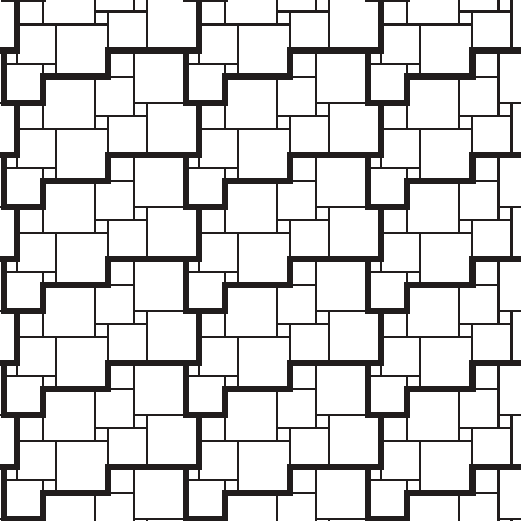}\caption{}\label{fig:18-fd}
\end{subfigure}
\caption{$( 3^2;3^2;4^3;4^3\,\, , \,\,3^1;4^2;3^1;4^2\,\, , \,\,4^1;4^3.2;4^2.1;3^1.1\,\, , \,\,1^0.4;3^0.1;4^1.3;2^0.3 )$}\label{fig:18}
\end{figure}

\begin{figure}[H]
\centering
\begin{subfigure}{0.45\textwidth}
\centering
\includegraphics[width=.75\textwidth]{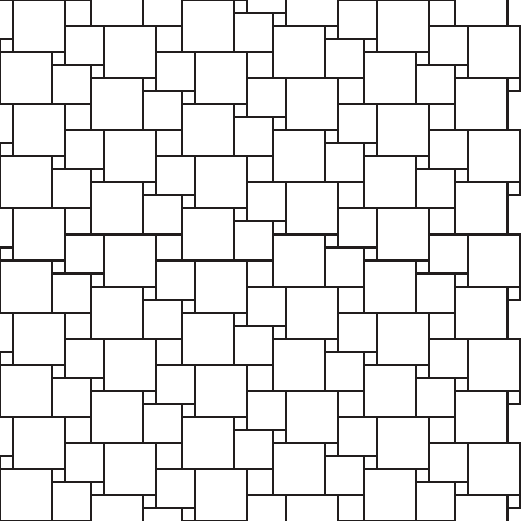}\caption{}\label{fig:19-plain}
\end{subfigure}
\begin{subfigure}{0.45\textwidth}
\centering
\includegraphics[width=.75\textwidth]{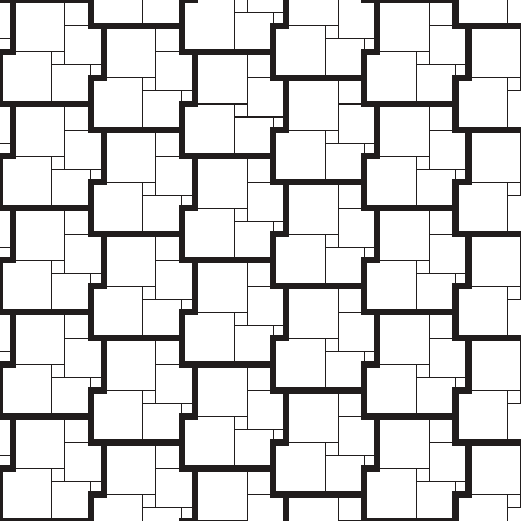}\caption{}\label{fig:19-fd}
\end{subfigure}
\caption{$( 3^2;3^2;4^3;4^3\,\, , \,\,3^1;4^2;3^1;4^2\,\, , \,\,4^1;4^3.2;4^2.1;3^1.1\,\, , \,\,1^0.4;1^0.3;4^1.3;2^0.3 )$}\label{fig:19}
\end{figure}

\begin{figure}[H]
\centering
\begin{subfigure}{0.45\textwidth}
\centering
\includegraphics[width=.75\textwidth]{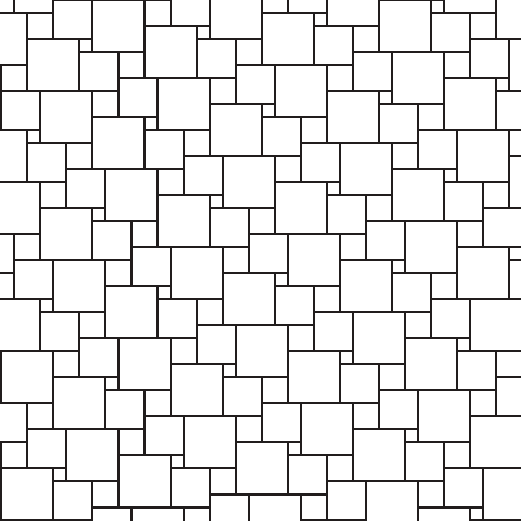}\caption{}\label{fig:20-plain}
\end{subfigure}
\begin{subfigure}{0.45\textwidth}
\centering
\includegraphics[width=.75\textwidth]{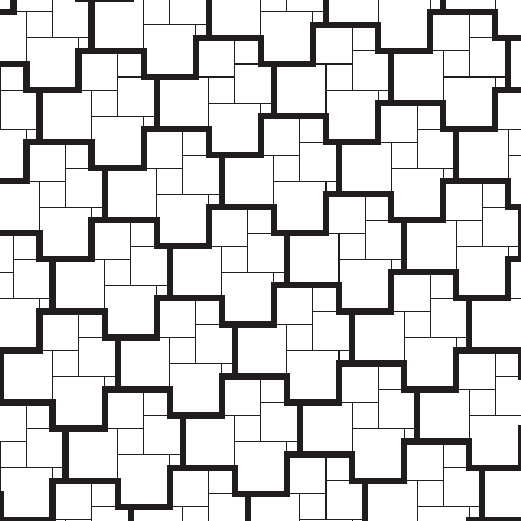}\caption{}\label{fig:20-fd}
\end{subfigure}
\caption{$( 3^2;4^3;3^2;4^3\,\, , \,\,3^1;3^1;4^2;4^2\,\, , \,\,4^1;4^3.2;3^2.2;4^2.1\,\, , \,\,2^0.3;1^0.3;4^1.3;2^0.4 )$}\label{fig:20}
\end{figure}

\begin{figure}[H]
\centering
\begin{subfigure}{0.45\textwidth}
\centering
\includegraphics[width=.75\textwidth]{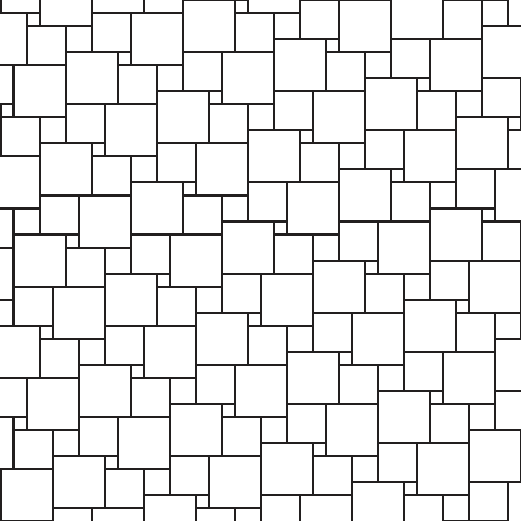}\caption{}\label{fig:21-plain}
\end{subfigure}
\begin{subfigure}{0.45\textwidth}
\centering
\includegraphics[width=.75\textwidth]{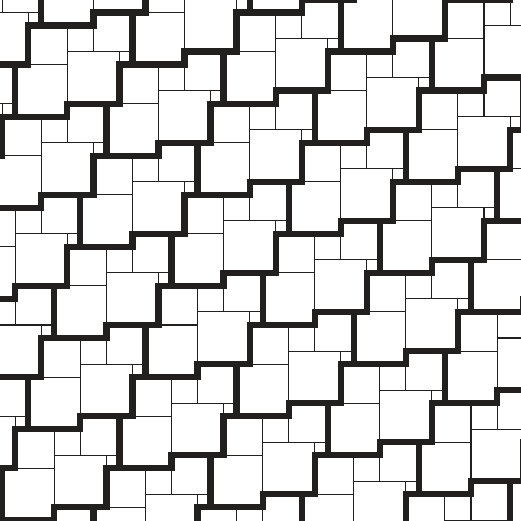}\caption{}\label{fig:21-fd}
\end{subfigure}
\caption{$( 3^2;4^3;3^2;4^3\,\, , \,\,3^1;4^2;3^1;4^2\,\, , \,\,4^1;4^1;4^3.2;4^2.1\,\, , \,\,2^0.3;1^0.3;4^3.3;4^1.3 )$}\label{fig:21}
\end{figure}

\begin{figure}[H]
\centering
\begin{subfigure}{0.45\textwidth}
\centering
\includegraphics[width=.75\textwidth]{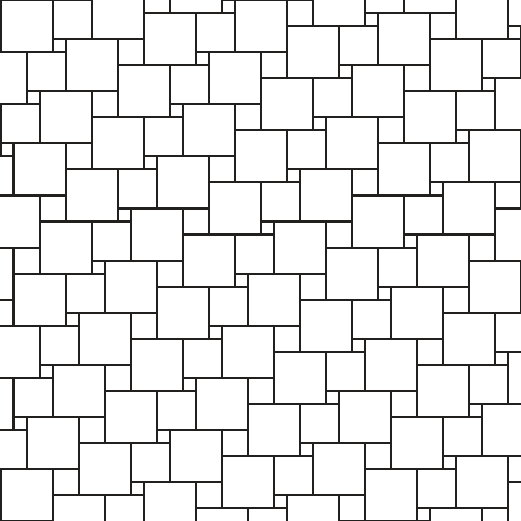}\caption{}\label{fig:22-plain}
\end{subfigure}
\begin{subfigure}{0.45\textwidth}
\centering
\includegraphics[width=.75\textwidth]{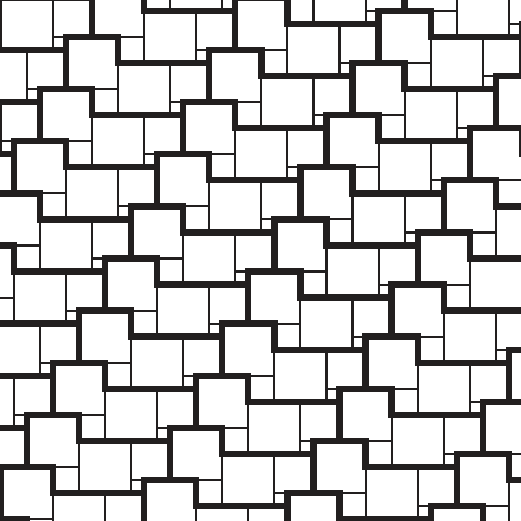}\caption{}\label{fig:22-fd}
\end{subfigure}
\caption{$( 3^2;4^3;3^2;4^3\,\, , \,\,4^2;4^2;4^2;4^2\,\, , \,\,4^1;4^2.1;4^1;4^2.1\,\, , \,\,2^0.4;3^0.1;3^1.4;2^0.4 )$}\label{fig:22}
\end{figure}

\begin{figure}[H]
\centering
\begin{subfigure}{0.45\textwidth}
\centering
\includegraphics[width=.75\textwidth]{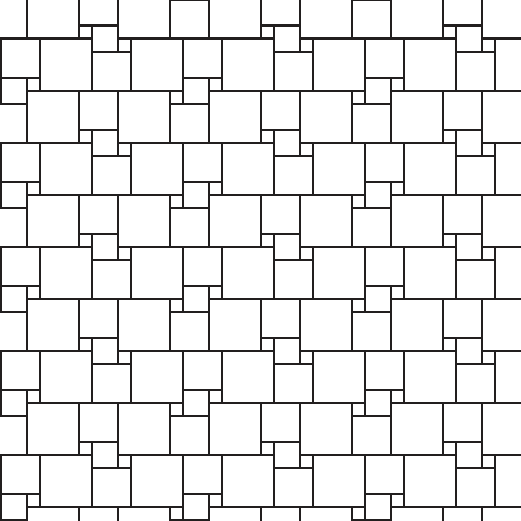}\caption{}\label{fig:23-plain}
\end{subfigure}
\begin{subfigure}{0.45\textwidth}
\centering
\includegraphics[width=.75\textwidth]{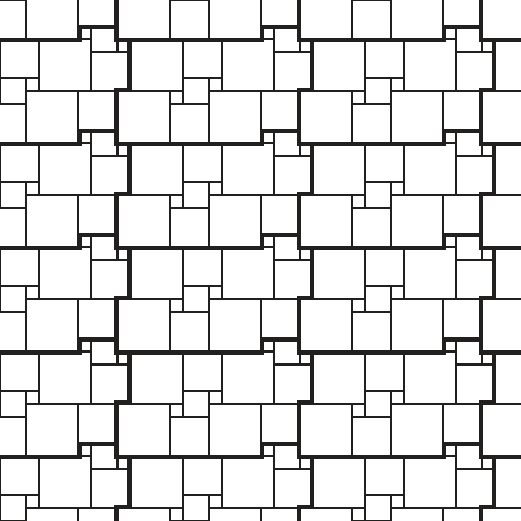}\caption{}\label{fig:23-fd}
\end{subfigure}
\caption{$( 2^1;3^2;4^3;4^3\,\, , \,\,3^1;4^3.1;3^1;4^3.1\,\, , \,\,4^0;2^0.1;4^1;4^3.3\,\, , \,\,1^0.4;3^0.1;4^1.3;3^0.2 )$}\label{fig:23}
\end{figure}

\begin{figure}[H]
\centering
  \begingroup
  \renewcommand{\thefigure}{T24(a)} 
\begin{subfigure}{0.45\textwidth}
\centering
\includegraphics[width=.75\textwidth]{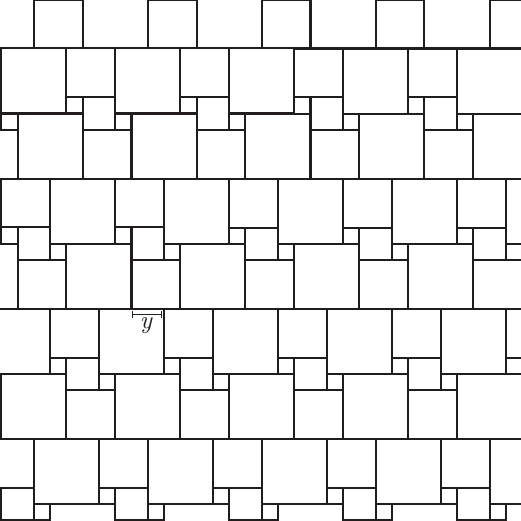}\caption{}\label{fig:24-plain}
\end{subfigure}
\begin{subfigure}{0.45\textwidth}
\centering
\includegraphics[width=.75\textwidth]{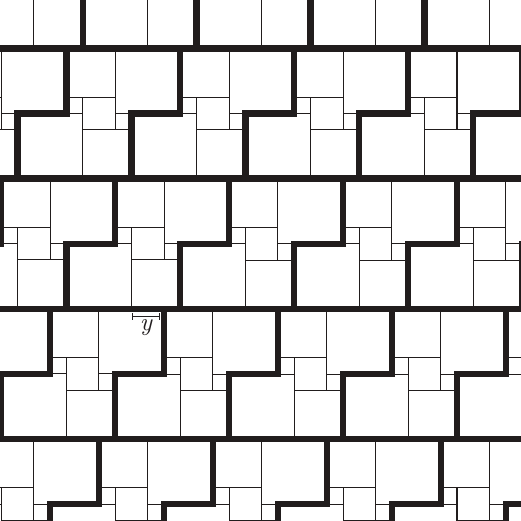}\caption{}\label{fig:24-fd}
\end{subfigure} 
\caption[T24(a)]{A tiling from the infinite family $\{\mathcal{T}_y\}$ of tilings in which both the 3- and 4-coronas are faulty. The parameter $y$ can vary continuously in $0 < y < 7$.}\label{fig:24}
\par
\endgroup 
\end{figure}

\begin{figure}[H]
\centering
  \begingroup
  \renewcommand{\thefigure}{T24(b)} 
\begin{subfigure}{0.45\textwidth}
\centering
\includegraphics[width=.75\textwidth]{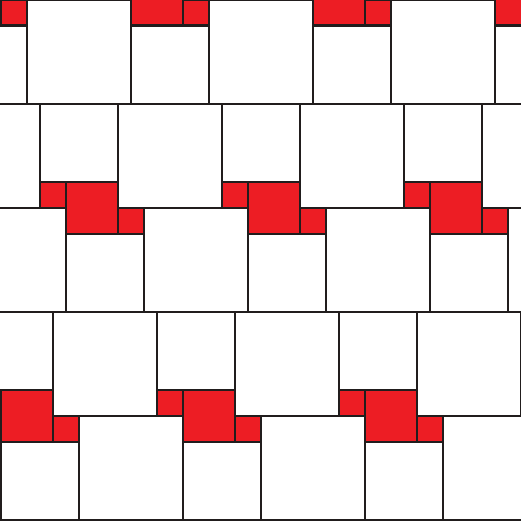}\caption{}\label{fig:24-half-plain}
\end{subfigure}
\begin{subfigure}{0.45\textwidth}
\centering
\includegraphics[width=.75\textwidth]{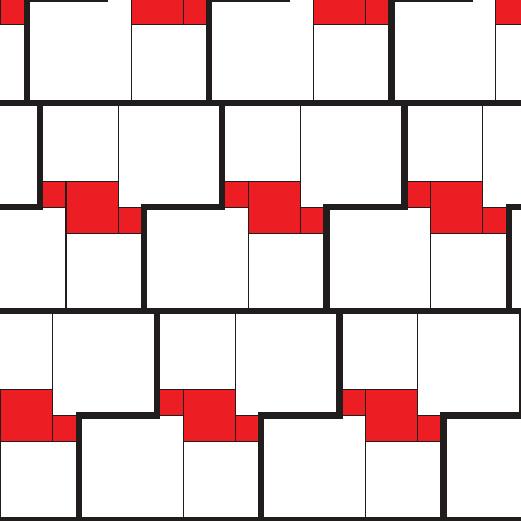}\caption{}\label{fig:24-half-fd}
\end{subfigure}
\caption[T24(b)]{With $y = 7/2$ in the family $\{\mathcal{T}_{y}\}$, the corresponding 4-tuple is 
$( 2^1;3^2;4^3;4^3\,\, , \,\,3^1;4^3.1;3^1;4^3.1\,\, , \,\,4^{1/2};4^0;2^0.1;4^1\,\, , \,\,4^{7/2}.3.4;3^0.1;4^1.1;2^1.3 )$}\label{fig:24-half}
 \endgroup 
\end{figure}

\begin{figure}[H]
\centering
  \begingroup
  \renewcommand{\thefigure}{T25} 
\begin{subfigure}{0.45\textwidth}
\centering
\includegraphics[width=.75\textwidth]{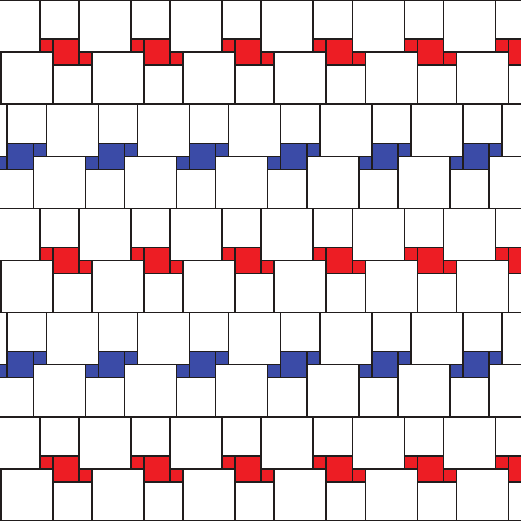}\caption{}\label{fig:25-ABAB-plain}
\end{subfigure}
\begin{subfigure}{0.45\textwidth}
\centering
\includegraphics[width=.75\textwidth]{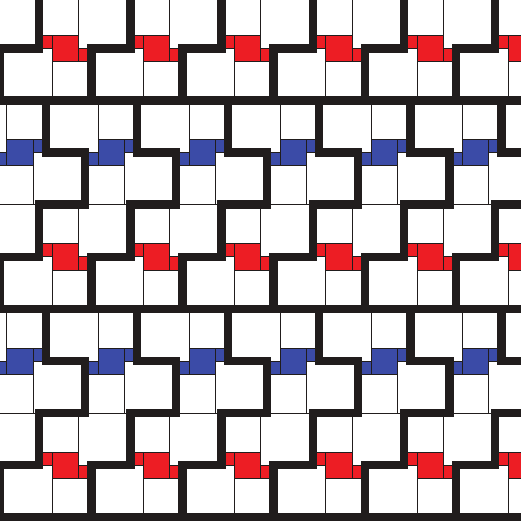}\caption{}\label{fig:25-ABAB-fd}
\end{subfigure}
\caption[T25]{The alternating row orientation tiling generated by the 4-tuple $( 2^1;3^2;4^3;4^3\,\, , \,\,3^1;4^3.1;3^1;4^3.1\,\, , \,\,4^{1/2};4^0;2^0.1;4^1\,\, , \,\,4^{7/2}.3.4;3^0.1;4^1.1;2^1.3 )$}\label{fig:25}
 \endgroup 
\end{figure}

\pagebreak

\section{Coronas found by the computer algorithm (Rows 1-29 of Table \ref{tab:29-tuples})}\label{apx:all-coronas}

\renewcommand{\thefigure}{C\arabic{figure}}
\setcounter{figure}{0}

\begin{figure}[H]
\centering
\includegraphics[width=.75\textwidth]{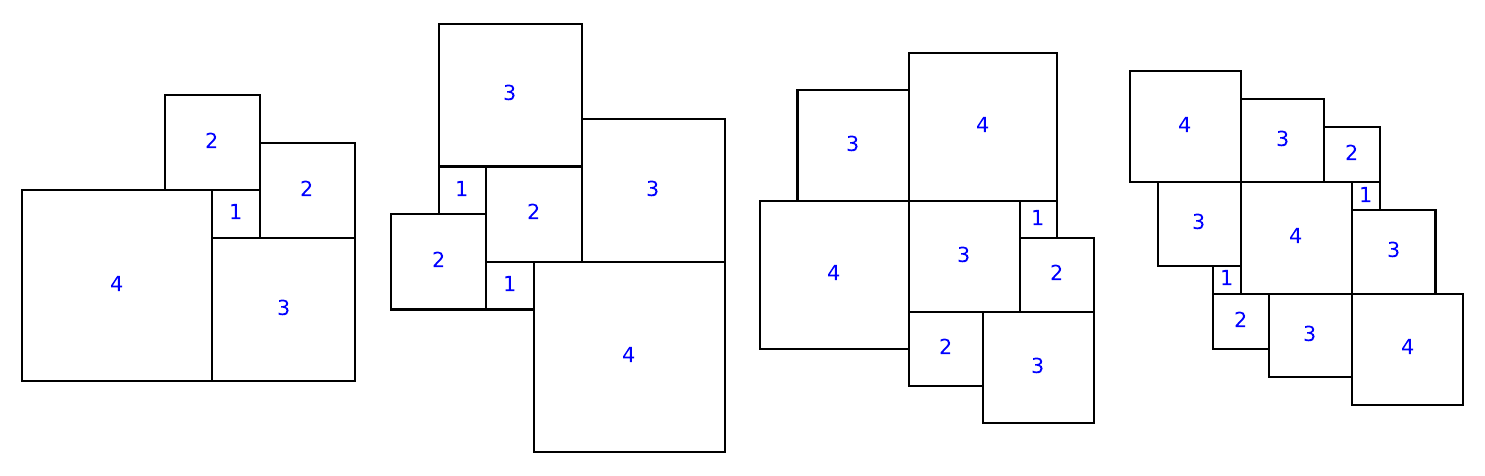}\caption{$( 2^1;2^1;3^2;4^3\,\, , \,\,3^1;3^1;4^3.1;2^1.1\,\, , \,\,3^3.4;1^0.2;3^2.2;4^1\,\, , \,\,4^4.3.2;1^0.3;4^4.3.2;1^0.3 )$}\label{fig:cor-1}
\end{figure}

\begin{figure}[H]
\centering
\includegraphics[width=.75\textwidth]{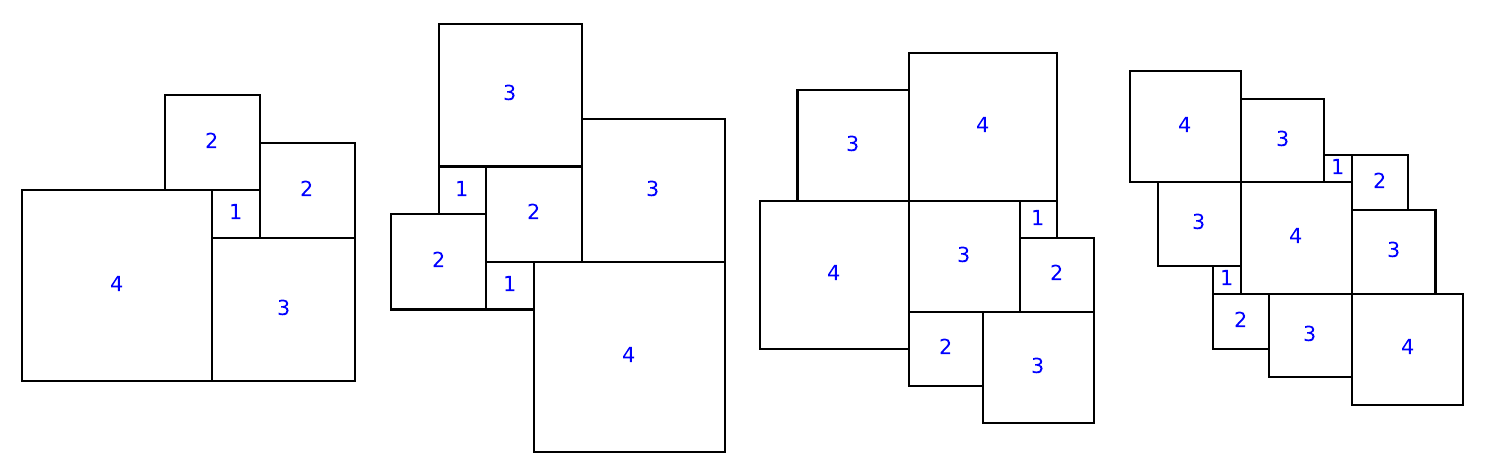}\caption{$( 2^1;2^1;3^2;4^3\,\, , \,\,3^1;3^1;4^3.1;2^1.1\,\, , \,\,3^3.4;1^0.2;3^2.2;4^1\,\, , \,\,4^4.3.1;2^1.3;4^4.3.2;1^0.3 )$}\label{fig:cor-2}
\end{figure}

\begin{figure}[H]
\centering
\includegraphics[width=.75\textwidth]{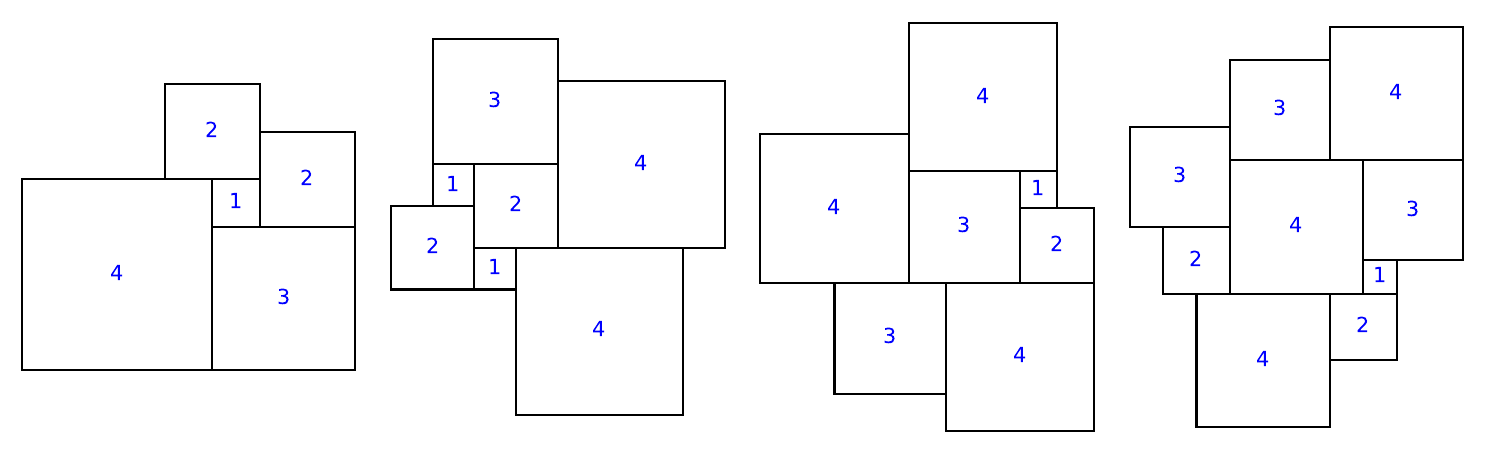}\caption{$( 2^1;2^1;3^2;4^3\,\, , \,\,3^1;4^2;4^3.1;2^1.1\,\, , \,\,4^0;1^0.2;4^2.3;4^0\,\, , \,\,3^0.4;3^0.1;2^1.4;2^0.3 )$}\label{fig:cor-3}
\end{figure}

\begin{figure}[H]
\centering
\includegraphics[width=.75\textwidth]{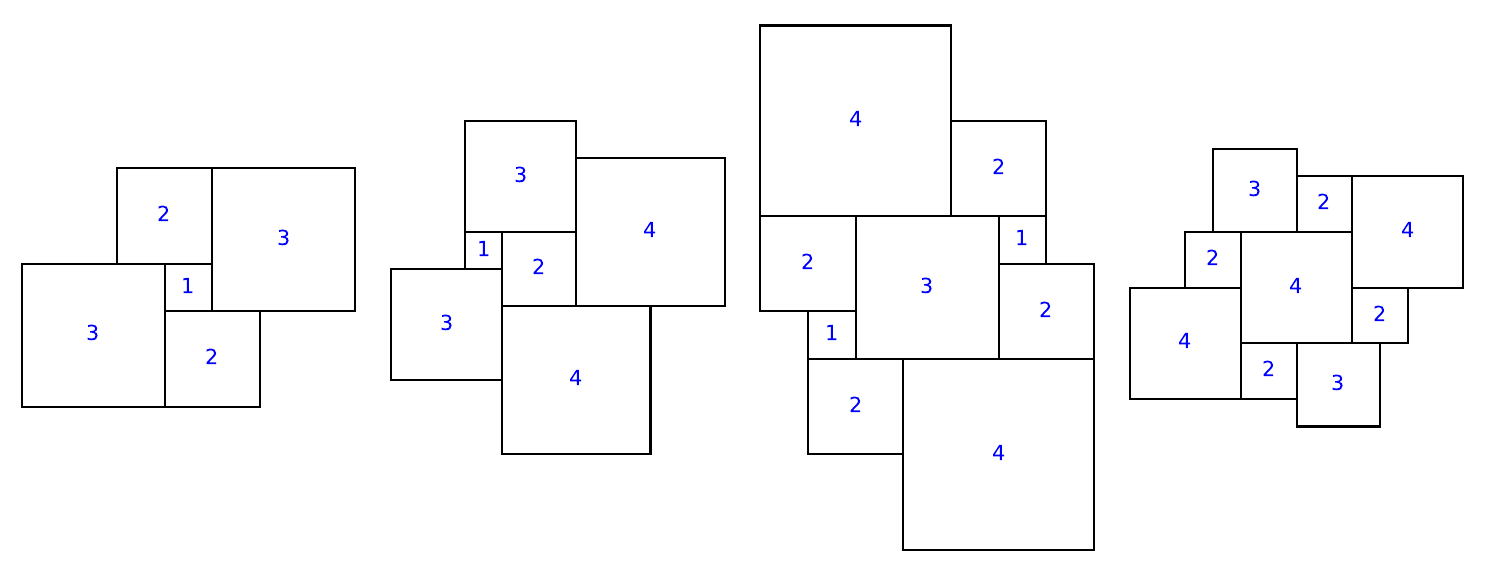}\caption{$( 2^1;3^2;2^1;3^2\,\, , \,\,3^1;4^2;4^2;3^2.1\,\, , \,\,4^2.2;1^0.2;4^2.2;1^0.2\,\, , \,\,3^1.2;4^2.2;3^1.2;4^2.2 )$}\label{fig:cor-4}
\end{figure}

\begin{figure}[H]
\centering
\includegraphics[width=.75\textwidth]{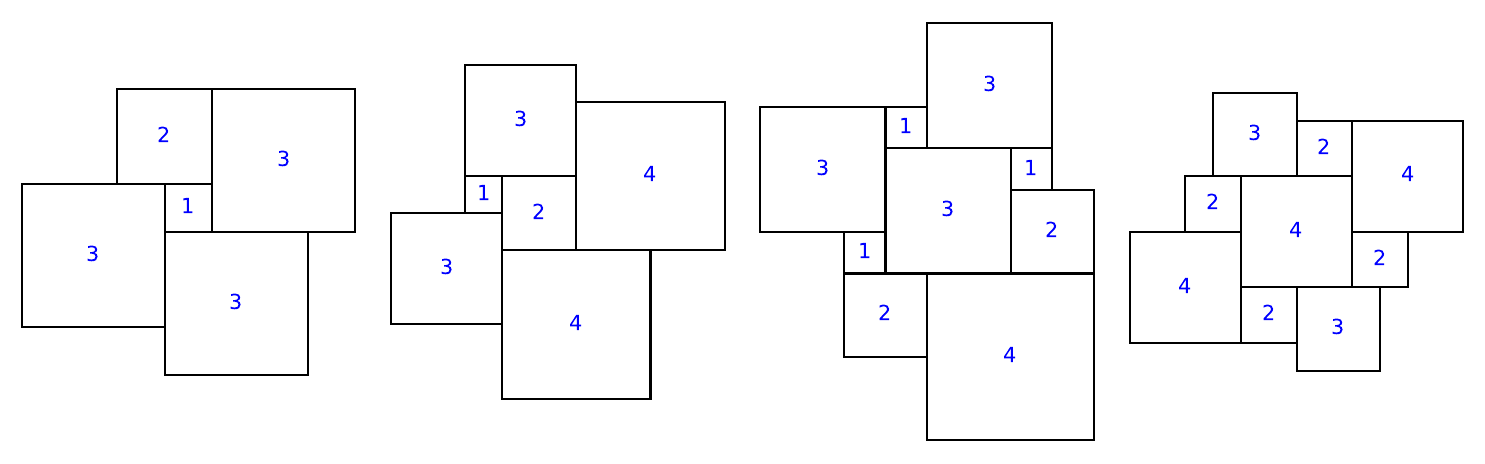}\caption{$( 2^1;3^2;3^2;3^2\,\, , \,\,3^1;4^2;4^2;3^2.1\,\, , \,\,1^0.3;1^0.2;4^2.2;1^0.3\,\, , \,\,3^1.2;4^2.2;3^1.2;4^2.2 )$}\label{fig:cor-5}
\end{figure}

\begin{figure}[H]
\centering
\includegraphics[width=.75\textwidth]{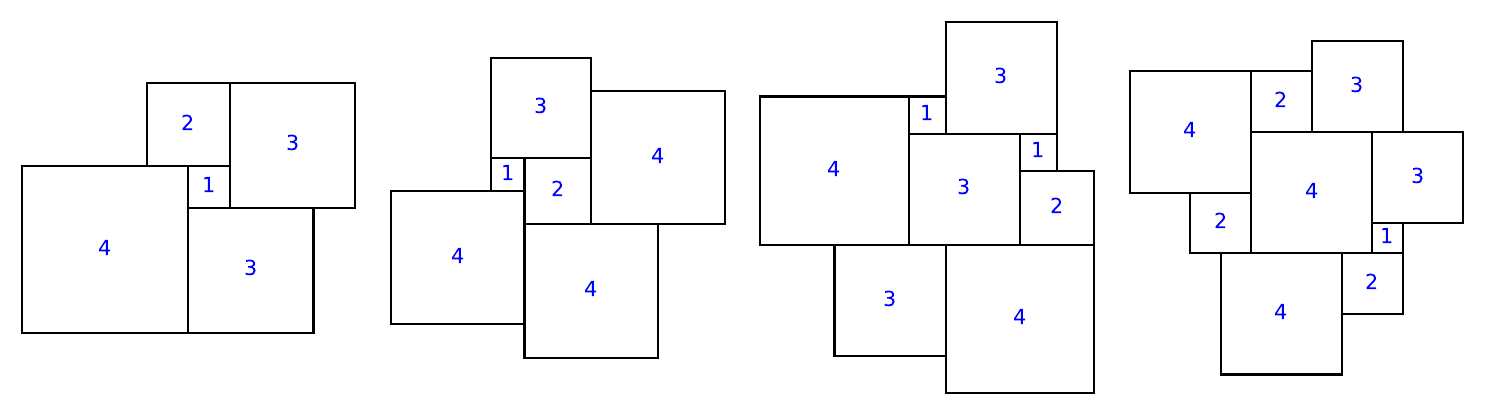}\caption{$( 2^1;3^2;3^2;4^3\,\, , \,\,3^1;4^2;4^2;4^3.1\,\, , \,\,1^0.3;1^0.2;4^2.3;4^0\,\, , \,\,2^0.3;3^0.1;2^1.4;2^0.4 )$}\label{fig:cor-6}
\end{figure}

\begin{figure}[H]
\centering
\includegraphics[width=.75\textwidth]{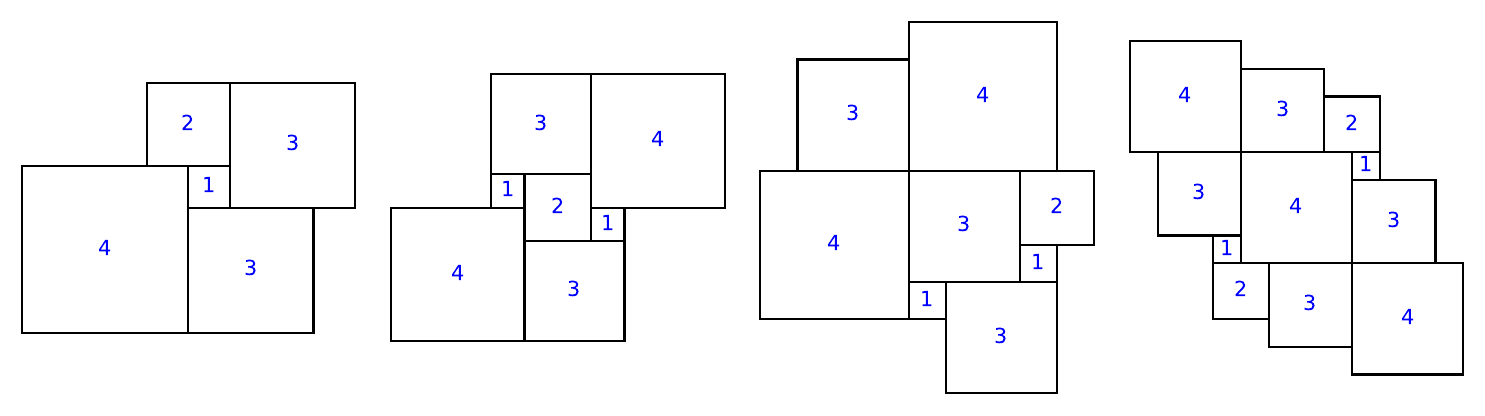}\caption{$( 2^1;3^2;3^2;4^3\,\, , \,\,3^1;4^3.1;3^1;4^3.1\,\, , \,\,3^3.4;2^0.1;3^1.1;4^1\,\, , \,\,4^4.3.2;1^0.3;4^4.3.2;1^0.3 )$}\label{fig:cor-7}
\end{figure}

\begin{figure}[H]
\centering
\includegraphics[width=.75\textwidth]{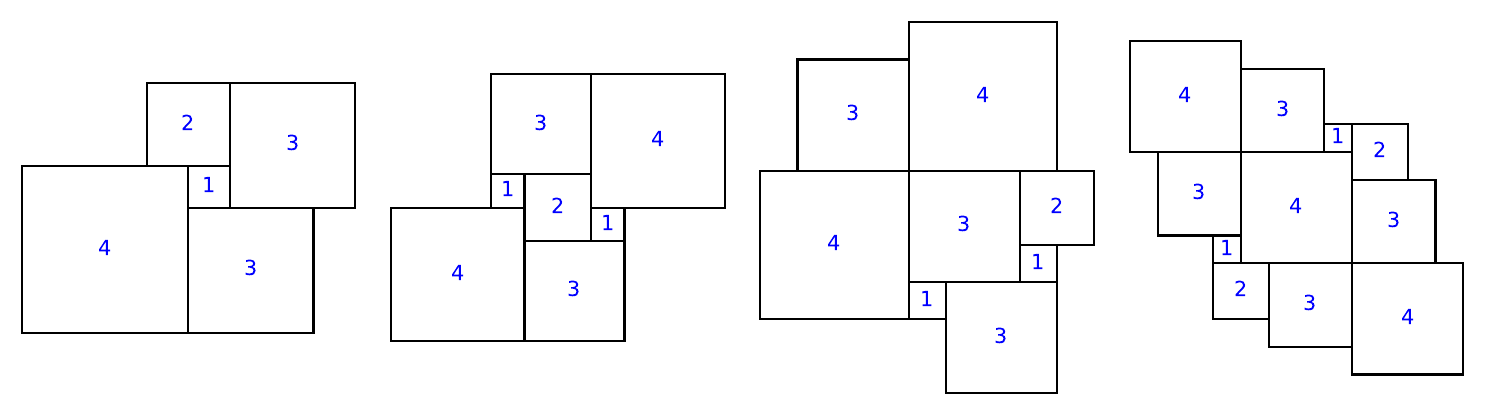}\caption{$( 2^1;3^2;3^2;4^3\,\, , \,\,3^1;4^3.1;3^1;4^3.1\,\, , \,\,3^3.4;2^0.1;3^1.1;4^1\,\, , \,\,4^4.3.1;2^1.3;4^4.3.2;1^0.3 )$}\label{fig:cor-8}
\end{figure}

\begin{figure}[H]
\centering
\includegraphics[width=.75\textwidth]{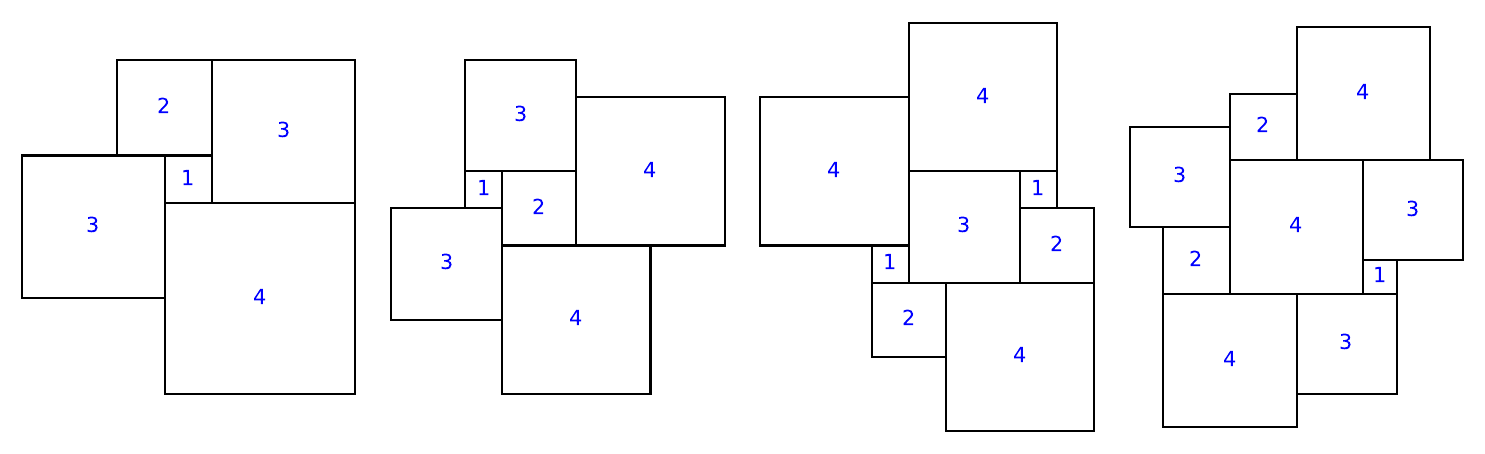}\caption{$( 2^1;3^2;4^3;3^2\,\, , \,\,3^1;4^2;4^2;3^2.1\,\, , \,\,4^0;1^0.2;4^2.2;1^0.4\,\, , \,\,2^0.4;3^0.1;3^1.4;2^0.3 )$}\label{fig:cor-9}
\end{figure}

\begin{figure}[H]
\centering
\includegraphics[width=.75\textwidth]{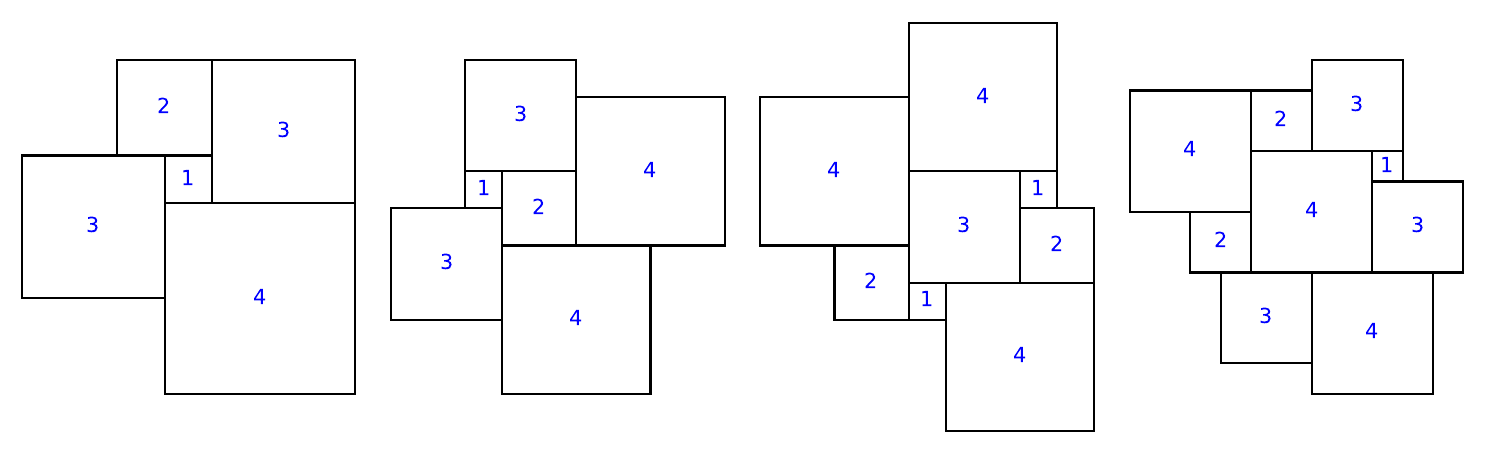}\caption{$( 2^1;3^2;4^3;3^2\,\, , \,\,3^1;4^2;4^2;3^2.1\,\, , \,\,4^0;1^0.2;4^2.1;2^1.4\,\, , \,\,2^0.3;1^0.3;4^2.3;2^0.4 )$}\label{fig:cor-10}
\end{figure}

\begin{figure}[H]
\centering
\includegraphics[width=.75\textwidth]{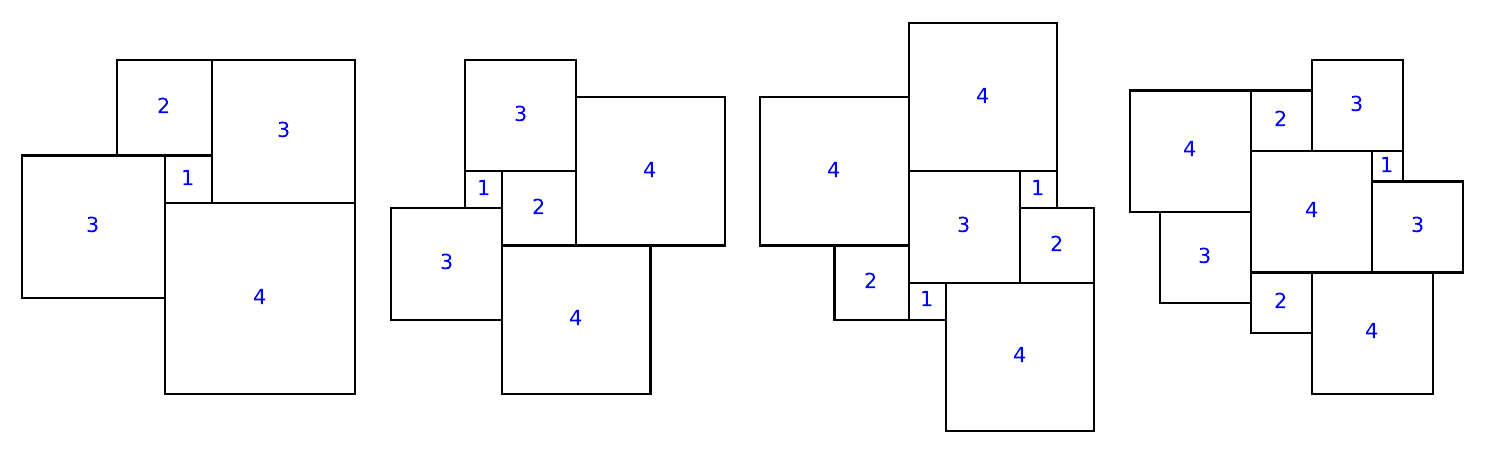}\caption{$( 2^1;3^2;4^3;3^2\,\, , \,\,3^1;4^2;4^2;3^2.1\,\, , \,\,4^0;1^0.2;4^2.1;2^1.4\,\, , \,\,2^0.3;1^0.3;4^2.2;3^1.4 )$}\label{fig:cor-11}
\end{figure}

\begin{figure}[H]
\centering
\includegraphics[width=.75\textwidth]{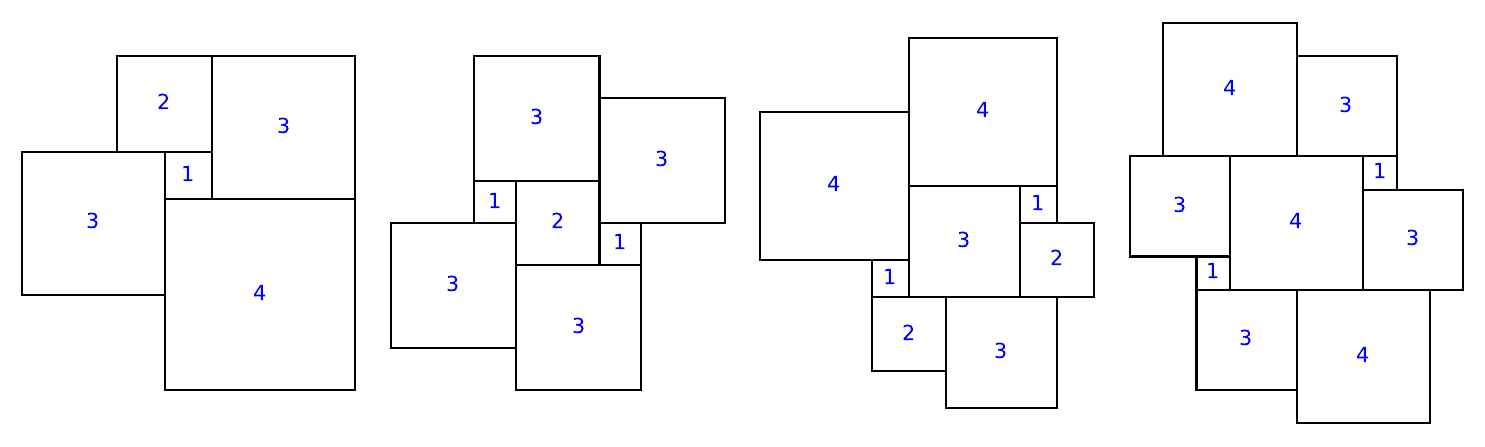}\caption{$( 2^1;3^2;4^3;3^2\,\, , \,\,3^1;3^2.1;3^1;3^2.1\,\, , \,\,4^0;1^0.2;3^1.2;1^0.4\,\, , \,\,4^2.3;1^0.3;4^2.3;1^0.3 )$}\label{fig:cor-12}
\end{figure}

\begin{figure}[H]
\centering
\includegraphics[width=.75\textwidth]{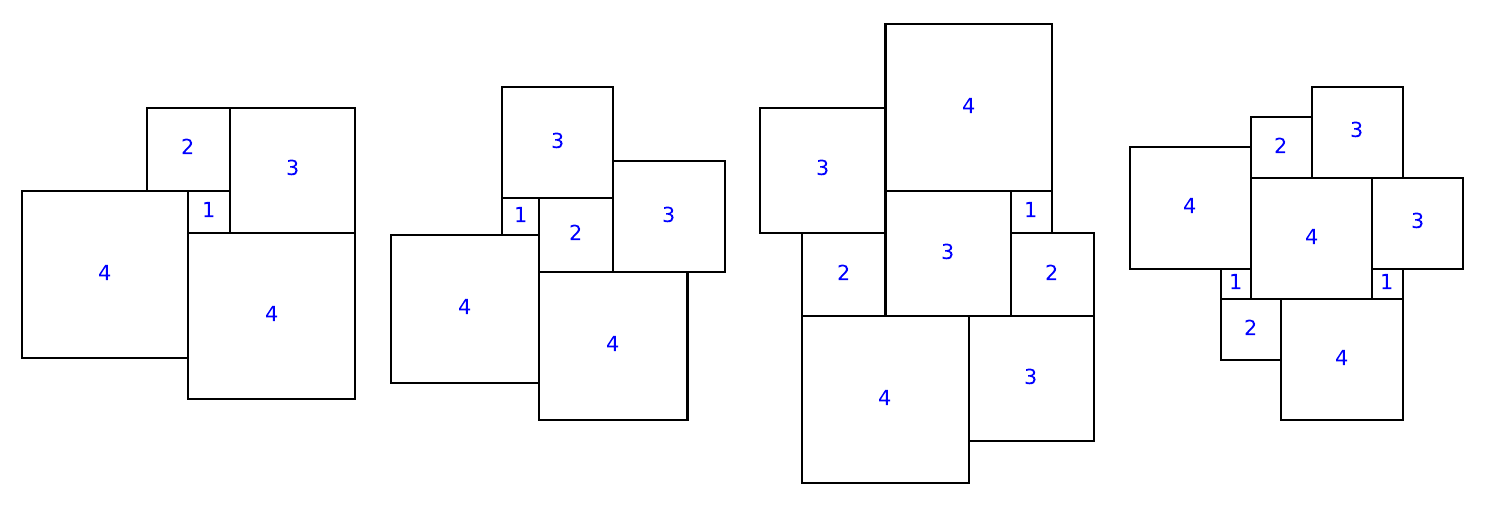}\caption{$( 2^1;3^2;4^3;4^3\,\, , \,\,3^1;3^1;4^2;4^3.1\,\, , \,\,4^0;1^0.2;3^2.4;2^0.3\,\, , \,\,2^0.3;3^0.1;4^1.2;1^0.4 )$}\label{fig:cor-13}
\end{figure}

\begin{figure}[H]
\centering
\includegraphics[width=.75\textwidth]{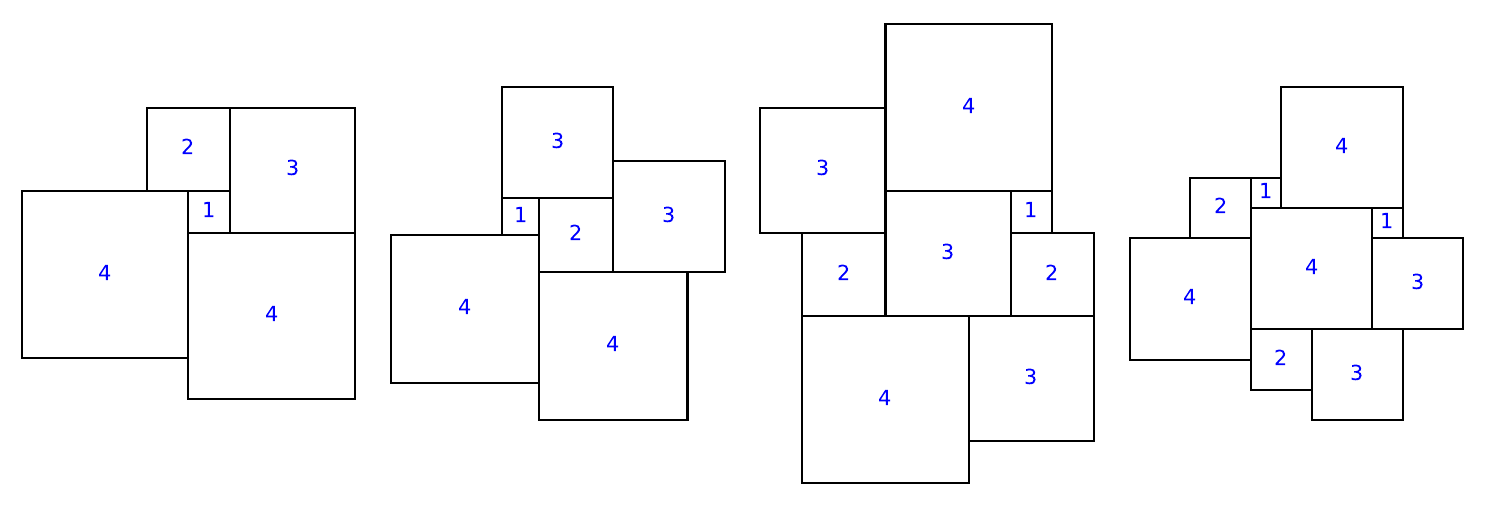}\caption{$( 2^1;3^2;4^3;4^3\,\, , \,\,3^1;3^1;4^2;4^3.1\,\, , \,\,4^0;1^0.2;3^2.4;2^0.3\,\, , \,\,1^0.4;1^0.3;3^1.2;4^1.2 )$}\label{fig:cor-14}
\end{figure}

\begin{figure}[H]
\centering
\includegraphics[width=.75\textwidth]{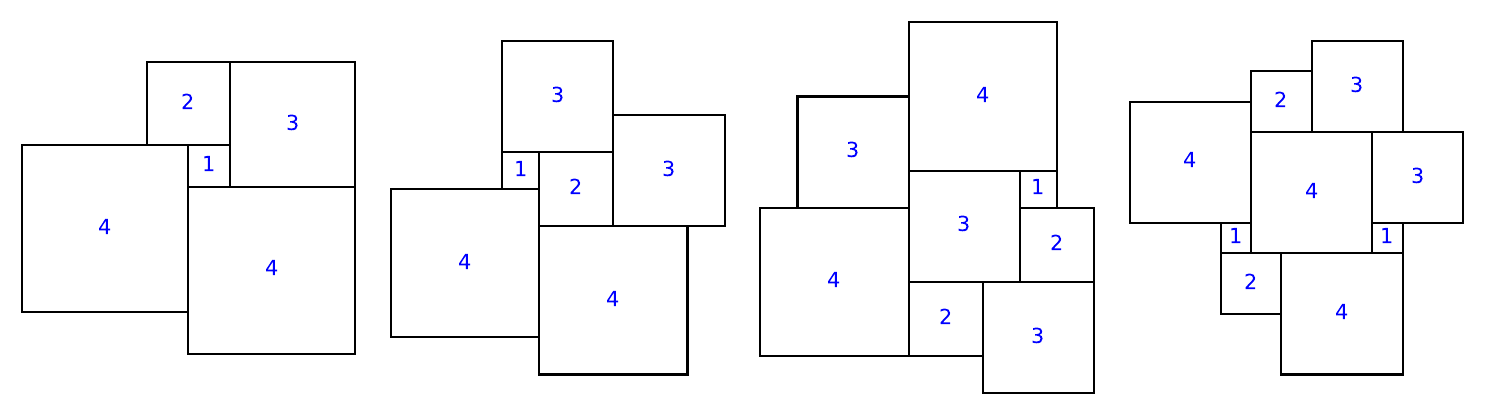}\caption{$( 2^1;3^2;4^3;4^3\,\, , \,\,3^1;3^1;4^2;4^3.1\,\, , \,\,4^0;1^0.2;3^2.2;4^2.3\,\, , \,\,2^0.3;3^0.1;4^1.2;1^0.4 )$}\label{fig:cor-15}
\end{figure}

\begin{figure}[H]
\centering
\includegraphics[width=.75\textwidth]{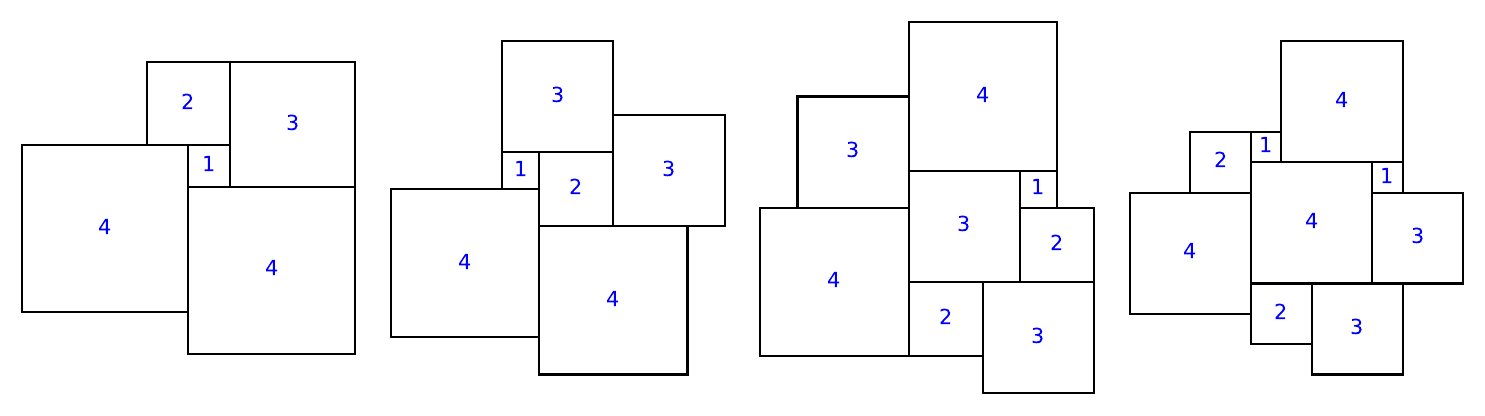}\caption{$( 2^1;3^2;4^3;4^3\,\, , \,\,3^1;3^1;4^2;4^3.1\,\, , \,\,4^0;1^0.2;3^2.2;4^2.3\,\, , \,\,1^0.4;1^0.3;3^1.2;4^1.2 )$}\label{fig:cor-16}
\end{figure}

\begin{figure}[H]
\centering
\includegraphics[width=.75\textwidth]{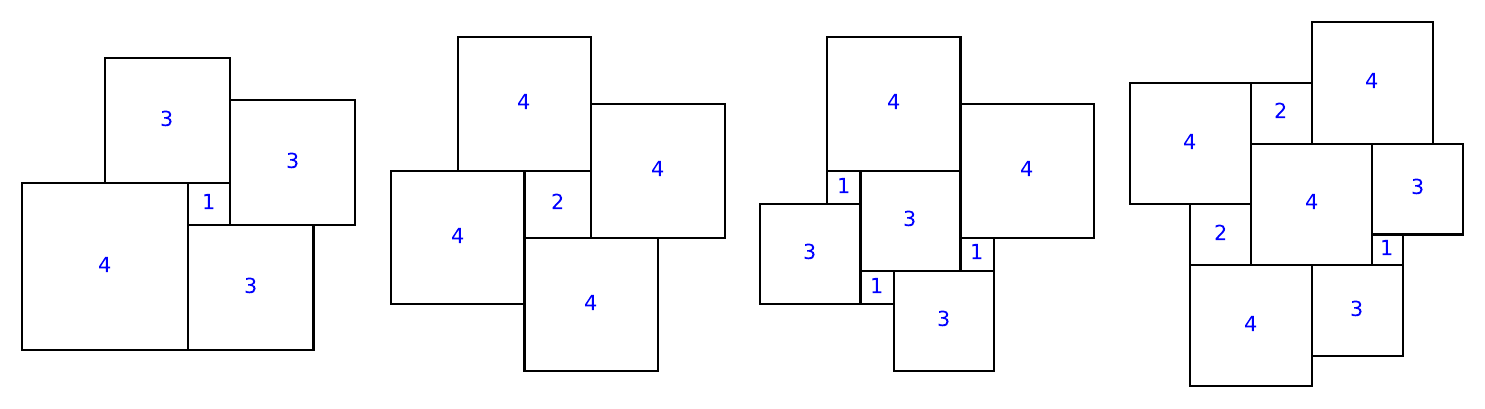}\caption{$(3^2;3^2;3^2;4^3 \,\, , \,\, 4^2;4^2;4^2;4^2 \,\, , \,\, 4^1;4^2.1;3^1.1;3^1.1 \,\, , \,\, 2^0.4;3^0.1;3^1.4;2^0.4)$}\label{fig:cor-17}
\end{figure}

\begin{figure}[H]
\centering
\includegraphics[width=.75\textwidth]{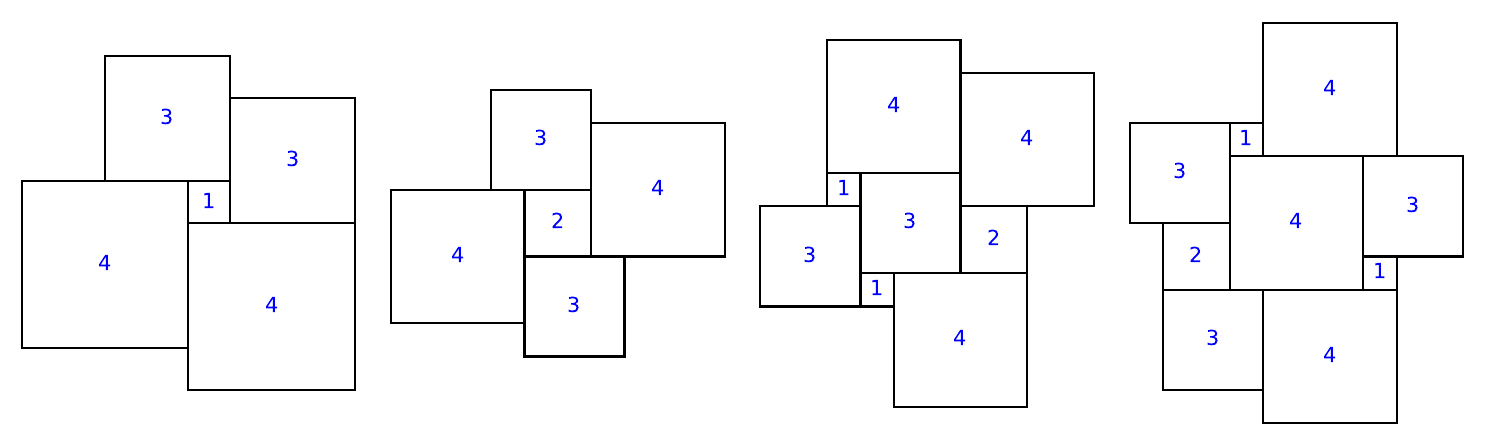}\caption{$(3^2;3^2;4^3;4^3 \,\, , \,\, 3^1;4^2;3^1;4^2 \,\, , \,\, 4^1;4^3.2;4^2.1;3^1.1 \,\, , \,\, 1^0.4;3^0.1;4^1.3;2^0.3)$}\label{fig:cor-18}
\end{figure}

\begin{figure}[H]
\centering
\includegraphics[width=.75\textwidth]{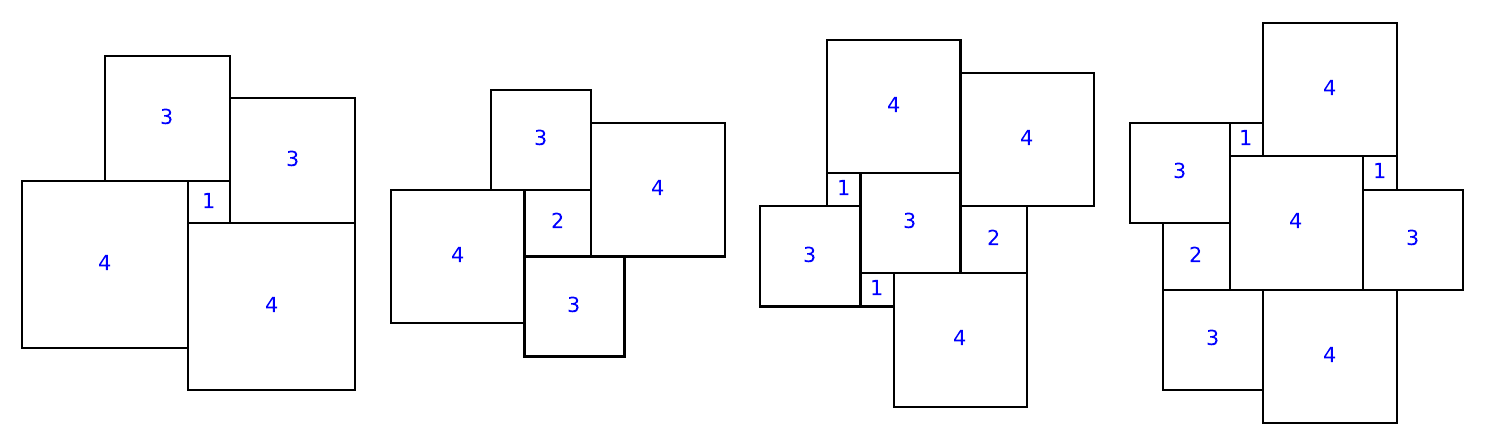}\caption{$(3^2;3^2;4^3;4^3 \,\, , \,\, 3^1;4^2;3^1;4^2 \,\, , \,\, 4^1;4^3.2;4^2.1;3^1.1 \,\, , \,\, 1^0.4;1^0.3;4^1.3;2^0.3)$}\label{fig:cor-19}
\end{figure}

\begin{figure}[H]
\centering
\includegraphics[width=.75\textwidth]{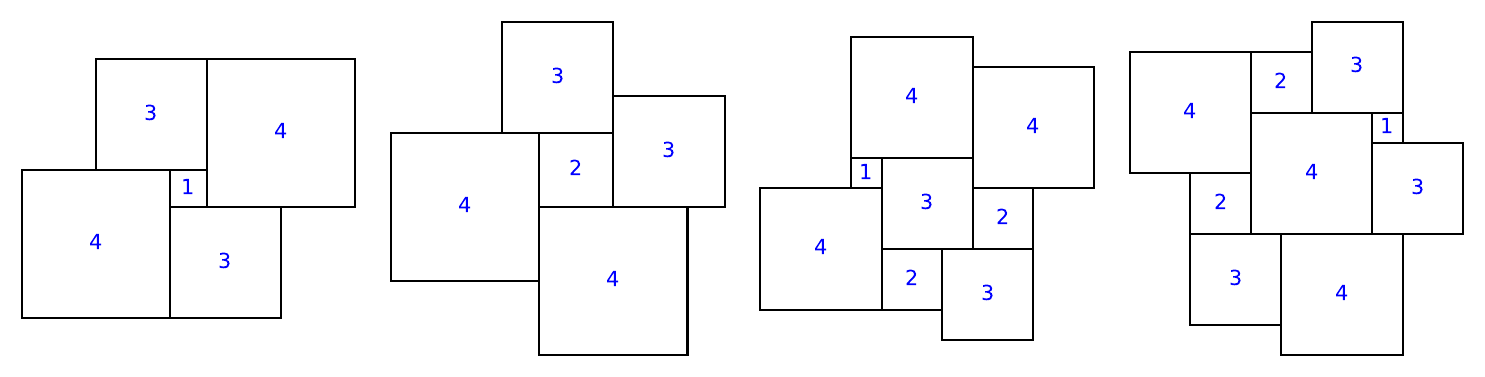}\caption{$(3^2;4^3;3^2;4^3 \,\, , \,\, 3^1;3^1;4^2;4^2 \,\, , \,\, 4^1;4^3.2;3^2.2;4^2.1 \,\, , \,\, 2^0.3;1^0.3;4^1.3;2^0.4)$}\label{fig:cor-20}
\end{figure}

\begin{figure}[H]
\centering
\includegraphics[width=.75\textwidth]{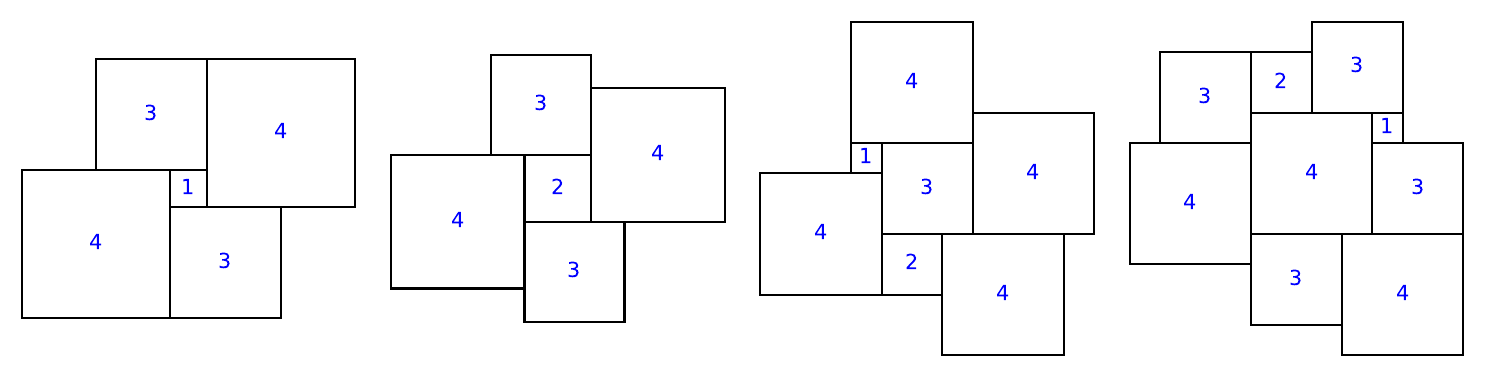}\caption{$(3^2;4^3;3^2;4^3 \,\, , \,\, 3^1;4^2;3^1;4^2 \,\, , \,\, 4^1;4^1;4^3.2;4^2.1 \,\, , \,\, 2^0.3;1^0.3;4^3.3;4^1.3)$}\label{fig:cor-21}
\end{figure}

\begin{figure}[H]
\centering
\includegraphics[width=.75\textwidth]{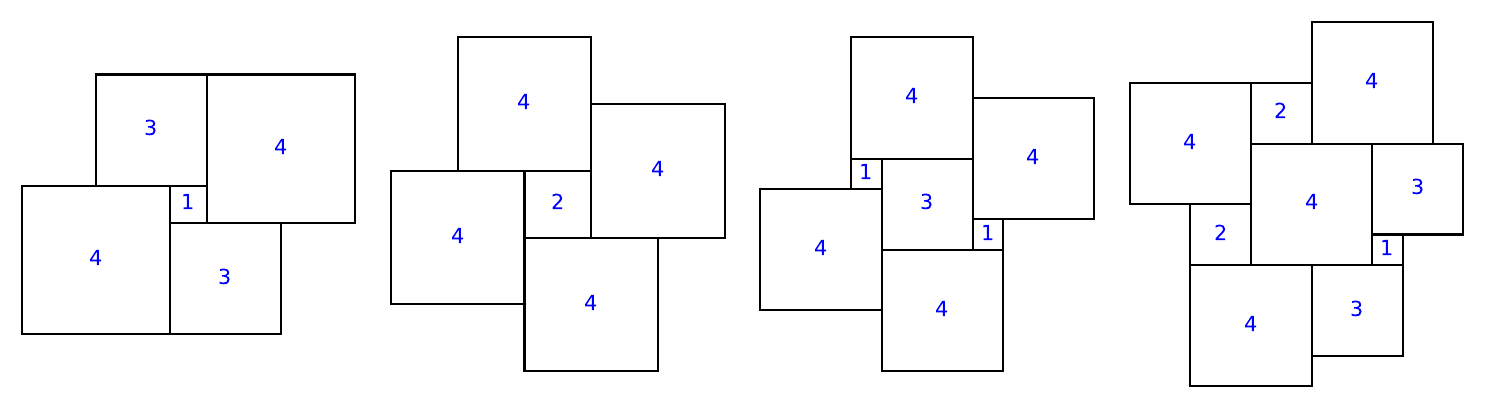}\caption{$(3^2;4^3;3^2;4^3 \,\, , \,\, 4^2;4^2;4^2;4^2 \,\, , \,\, 4^1;4^2.1;4^1;4^2.1 \,\, , \,\, 2^0.4;3^0.1;3^1.4;2^0.4)$}\label{fig:cor-22}
\end{figure}

\begin{figure}[H]
\centering
\includegraphics[width=.75\textwidth]{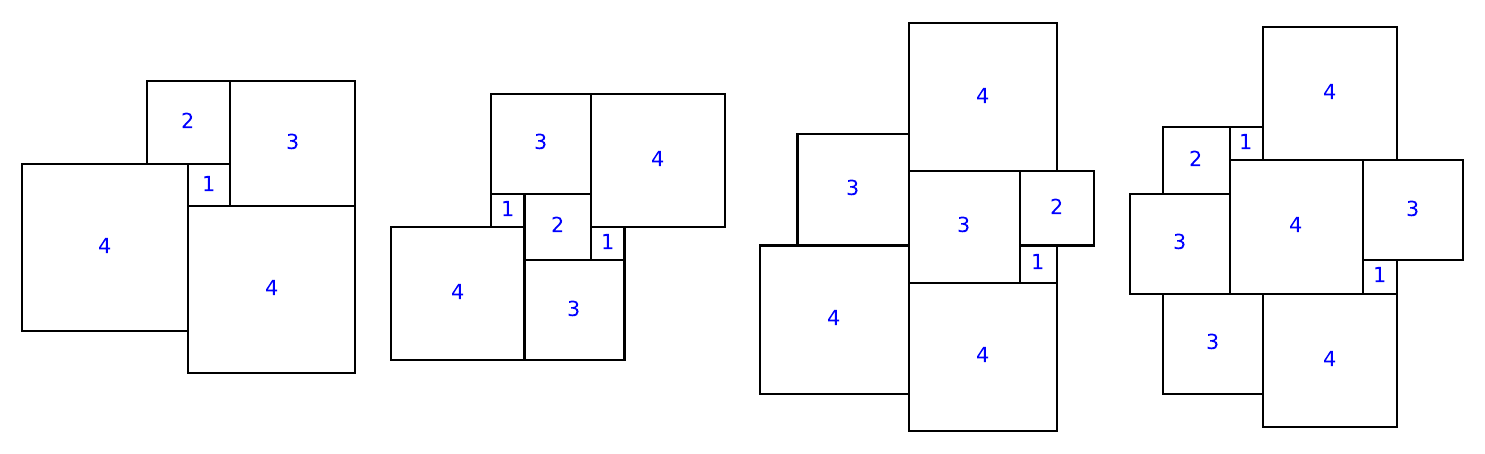}\caption{$(2^1;3^2;4^3;4^3 \,\, , \,\, 3^1;4^3.1;3^1;4^3.1 \,\, , \,\, 4^0;2^0.1;4^1;4^3.3 \,\, , \,\, 1^0.4;3^0.1;4^1.3;3^0.2)$}\label{fig:cor-23}
\end{figure}

\begin{figure}[H]
\centering
\includegraphics[width=.75\textwidth]{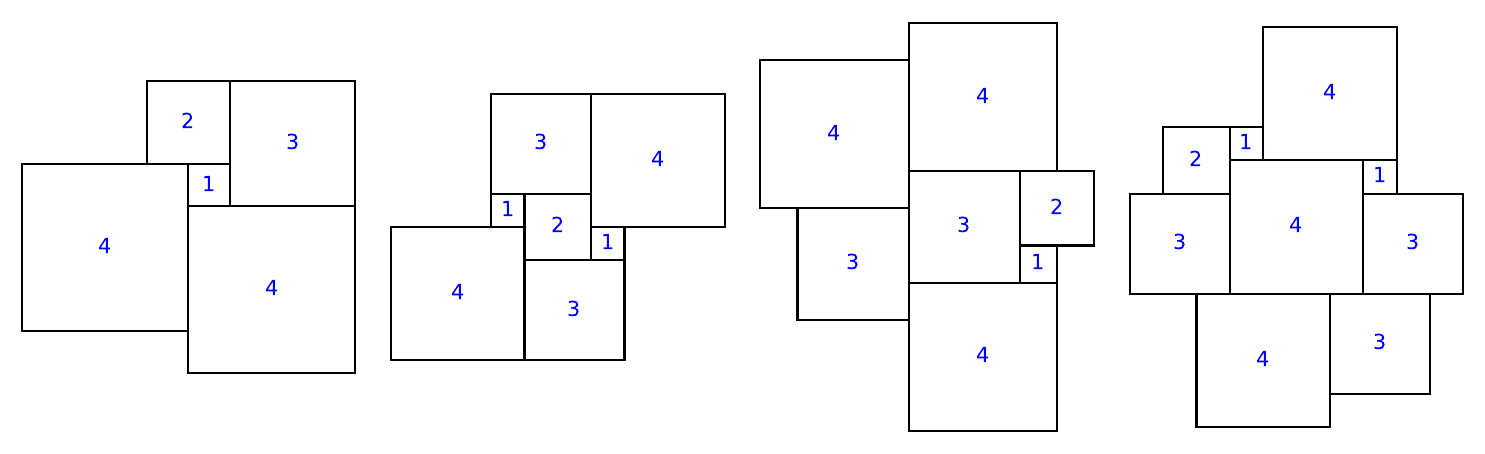}\caption{$(2^1;3^2;4^3;4^3 \,\, , \,\, 3^1;4^3.1;3^1;4^3.1 \,\, , \,\, 4^0;2^0.1;4^1;3^1.4 \,\, , \,\, 1^0.4;1^0.3;3^2.4;3^0.2)$}\label{fig:cor-24}
\end{figure}

\begin{figure}[H]
\centering
\includegraphics[width=.75\textwidth]{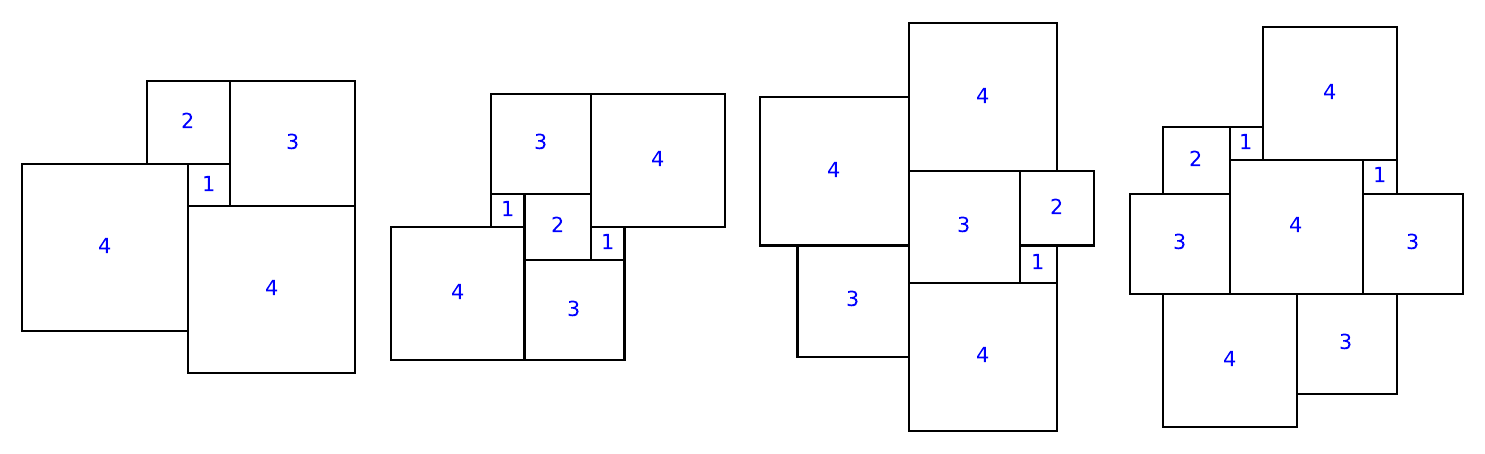}\caption{$(2^1;3^2;4^3;4^3 \,\, , \,\, 3^1;4^3.1;3^1;4^3.1 \,\, , \,\, 4^0;2^0.1;4^1;3^2.4 \,\, , \,\, 1^0.4;1^0.3;3^1.4;3^0.2)$}\label{fig:cor-25}
\end{figure}

\begin{figure}[H]
\centering
\includegraphics[width=.75\textwidth]{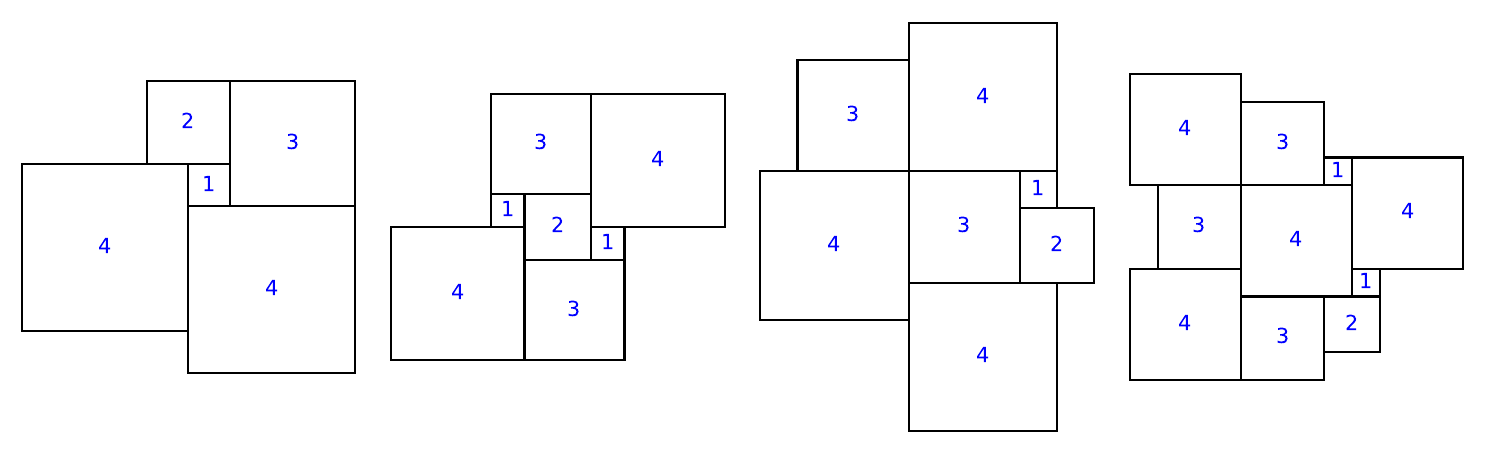}\caption{$(2^1;3^2;4^3;4^3 \,\, , \,\, 3^1;4^3.1;3^1;4^3.1 \,\, , \,\, 3^3.4;1^0.2;4^1;4^1 \,\, , \,\, 4^4.3.1;4^1.1;2^1.3;4^3.3)$}\label{fig:cor-26}
\end{figure}

\begin{figure}[H]
\centering
\includegraphics[width=.75\textwidth]{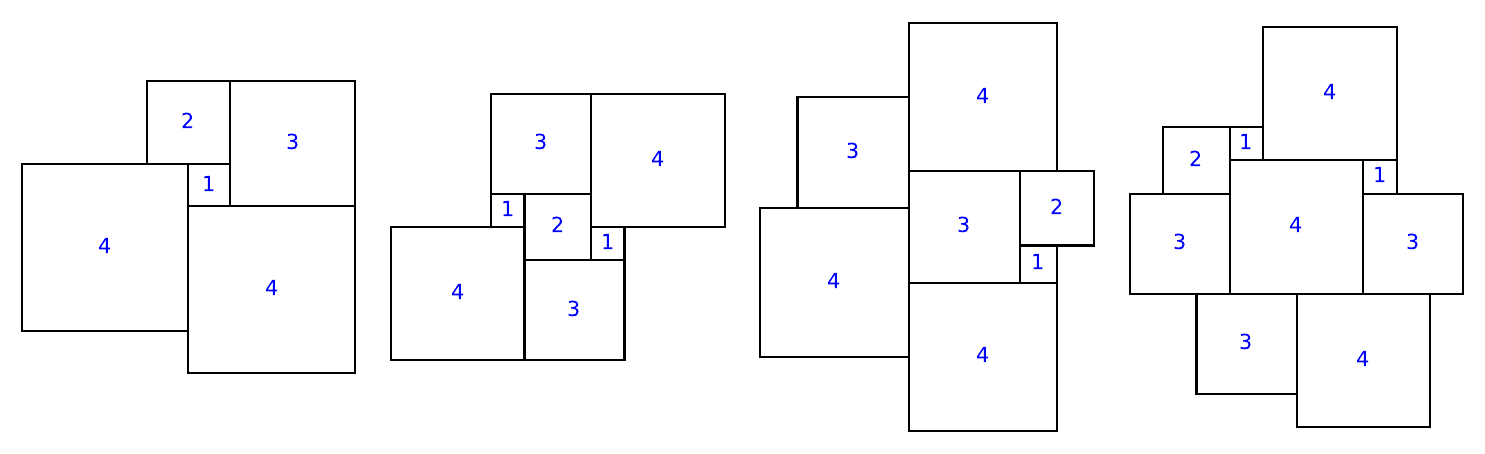}\caption{$(2^1;3^2;4^3;4^3 \,\, , \,\, 3^1;4^3.1;3^1;4^3.1 \,\, , \,\, 4^0;2^0.1;4^1;4^2.3 \,\, , \,\, 1^0.4;1^0.3;4^2.3;3^0.2)$}\label{fig:cor-27}
\end{figure}

\begin{figure}[H]
\centering
\includegraphics[width=.75\textwidth]{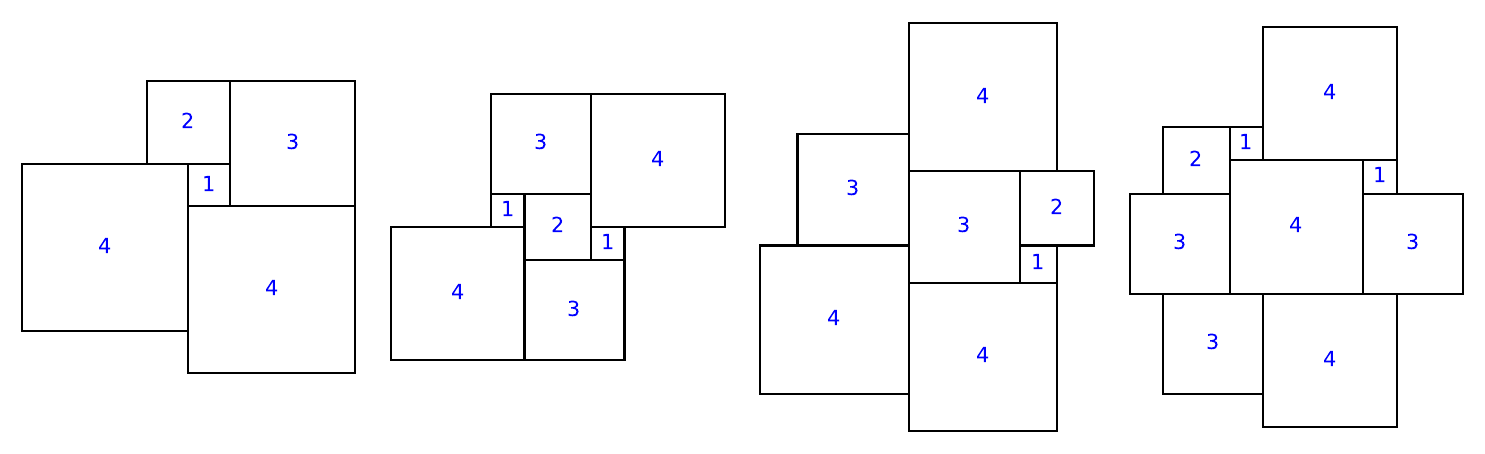}\caption{$(2^1;3^2;4^3;4^3 \,\, , \,\, 3^1;4^3.1;3^1;4^3.1 \,\, , \,\, 4^0;2^0.1;4^1;4^3.3 \,\, , \,\, 1^0.4;1^0.3;4^1.3;3^0.2)$}\label{fig:cor-28}
\end{figure}

\begin{figure}[H]
\centering
\includegraphics[width=.75\textwidth]{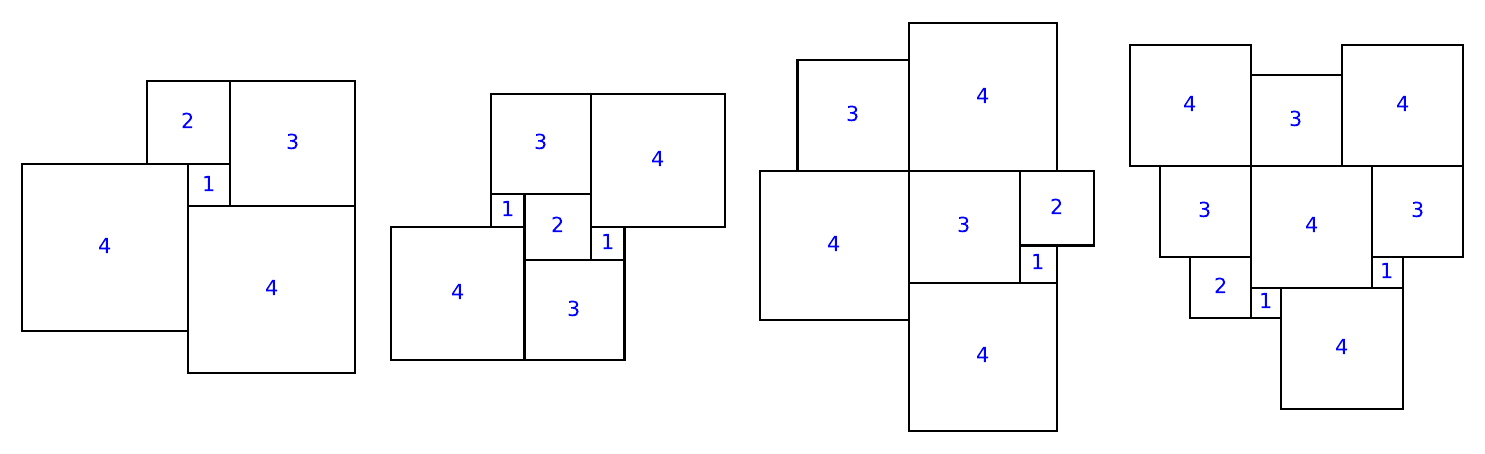}\caption{$(2^1;3^2;4^3;4^3 \,\, , \,\, 3^1;4^3.1;3^1;4^3.1 \,\, , \,\, 3^3.4;2^0.1;4^1;4^1 \,\, , \,\, 4^4.3.4;3^0.1;4^1.1;2^1.3)$}\label{fig:cor-29}
\end{figure}

\pagebreak

\section{Code}\label{apx:code}

The code is downloadable (on arXiv) as a single Jupyter (.ipynb) notebook. All of the code was executed using the SageMath 10.6.beta3 kernel.

\subsection{Side Generators}\label{apx:sides}\mbox{}\\

\begin{lstlisting}[caption={Side type 1 generator}, label={prog:side-1}, basicstyle=\footnotesize\ttfamily]
import itertools as it
import time

def Side_Type_i_Generator(Center_size,max_size):

    C = Center_size
    m = max_size

    #Using integer partitions to create a list of all potential sides of the form [0,[a0,a1,...,ak]]
    AllPotentialSides = []
    for n in range(C+1,C+m):
        
        #Generate all partitions of n using integers 1, 2, ..., m
        P = Partitions(n,max_part = m).list()
        
        #Creating all permutations of the integer partitions
        for BasePart in P:
            MoreParts = Permutations(BasePart) 
            for x in MoreParts:
                AllPotentialSides.append([0,list(x)]) #the exponent is 0 for type 1 sides
                
       
    # Filtered Sides is intialized as a copy of AllPotentialSides.
    # We will remove sides from FilteredSides when the are caught by a filter.
    FilteredSides = []
    for i in AllPotentialSides:
        FilteredSides.append(i)
    
    for L in AllPotentialSides:

        #Remove any instances where a0 = C
        if L[1][0] == C:
            FilteredSides.remove(L)
            continue
            
        #Remove any instances where L has ai = a(i+1) for some i = 0, ..., n-2
        flag = False
        for i in range(len(L[1])-1):
            if L[1][i] == L[1][i+1]:
                flag = True
                break
        if flag == True:
            FilteredSides.remove(L)
            continue

        #Remove any instances where L contains "a,1,b" or "a,2,b" since any such instance
        #violates the vortex condition for size-1 and size-2 squares
        flag = False
        for i in range(1,len(L[1])-1):
            if (L[1][i] == 1) or (L[1][i] == 2):
                flag = True
                break
        if flag == True:
            FilteredSides.remove(L)
            continue

        # The afterhang y must satisfy 0 < y < ak (where ak is the last square)
        afterhang = sum(L[1][i] for i in range(len(L[1]))) - C
        last_square = L[1][-1]
        if afterhang <= 0 or afterhang >= last_square:
            FilteredSides.remove(L)
            continue
    
    return(FilteredSides)
\end{lstlisting}

\begin{lstlisting}[caption={Side type 2 generator}, label={prog:side-2}, basicstyle=\footnotesize\ttfamily]
def Side_Type_ii_Generator(Center_size,max_size):

    C = Center_size
    m = max_size

    # Using integer partitions to create a list of all potential sides of the form [e,[a0,a1,...,ak]]
    # The partition constraint tells us to find partitions of C+1, C+2, ..., and (C + m)
    AllPotentialSides = []
    for n in range(C+1,C+m+1):
        
        #Generate all partitions of n using integers 1, 2, ..., m
        P = Partitions(n,max_part = m).list()
        
        #Creating all permutations of the integer partitions
        for BasePart in P:
            MoreParts = Permutations(BasePart) 
            
            # Creating potential sides by adding eligible exponents
            for x in MoreParts:
                if n-C <= x[0]:
                    AllPotentialSides.append([n-C,list(x)])
       
    # Filtered Sides is intialized as a copy of AllPotentialSides.
    # We will remove sides from FilteredSides when the are caught by a filter.
    
    FilteredSides = []
    for i in AllPotentialSides:
        FilteredSides.append(i)
    
    for L in AllPotentialSides:
    
        #Remove any instances where L has ai = a(i+1) for some i = 0, ..., n-2
        flag = False
        for i in range(len(L[1])-1):
            if L[1][i] == L[1][i+1]:
                flag = True
                break
        if flag == True:
            FilteredSides.remove(L)
            continue

        #Remove any instances where L contains "a,1,b" or "a,2,b" since any such instance
        #violates the vortex condition for size-1 and size-2 squares
        flag = False
        for i in range(1,len(L[1])-1):
            if (L[1][i] == 1) or (L[1][i] == 2):
                flag = True
                break
        if flag == True:
            FilteredSides.remove(L)
            continue

        #Remove all double corners of size 1 or 2:
        if (L[0] == 1 and L[1][0] == 1) or (L[0] == 2 and L[1][0] == 2):
            FilteredSides.remove(L)
            continue

        #Remove any instances in which a0 = e and a1 = C (violates unilaterality)
        if len(L[1]) > 1:
            if L[0] == L[1][0] and L[1][1] == C:
                FilteredSides.remove(L)
                continue
    
    return(FilteredSides)
\end{lstlisting}

\begin{lstlisting}[caption={Side type 3 generator}, label={prog:side-3}, basicstyle=\footnotesize\ttfamily]
def Side_Type_iii_Generator(Center_size,max_size):

    C = Center_size
    m = max_size

    # Using integer partitions to create a list of all potential sides of the form [e,[a0,a1,...,ak]]
    # The partition constraint tells us to find partitions of C.
    AllPotentialSides = []
           
    #Generate all partitions of C using integers 1, 2, ..., m
    P = Partitions(C,max_part = m).list()
        
    #Creating all permutations of the integer partitions
    for BasePart in P:
        MoreParts = Permutations(BasePart) 
        # Adding exponents
        for x in MoreParts:
            # The exponent constraint is e = 0
            AllPotentialSides.append([0,list(x)])
                
    # Filtered Sides is intialized as a copy of AllPotentialSides.
    # We will remove sides from FilteredSides when they are caught by a filter.
    FilteredSides = []
    for i in AllPotentialSides:
        FilteredSides.append(i)
    
    for L in AllPotentialSides:

        #Remove any instances where a0 = C or ak = C
        if L[1][0] == C or L[1][-1] == C:
            FilteredSides.remove(L)
            continue
            
        #Remove any instances where L has ai = a(i+1) for some i = 0, ..., n-2
        flag = False
        for i in range(len(L[1])-1):
            if L[1][i] == L[1][i+1]:
                flag = True
                break
        if flag == True:
            FilteredSides.remove(L)
            continue

        #Remove any instances where L contains "a,1,b" or "a,2,b" since any such instance
        #violates the vortex condition for size-1 and size-2 squares
        flag = False
        for i in range(1,len(L[1])-1):
            if (L[1][i] == 1) or (L[1][i] == 2):
                flag = True
                break
        if flag == True:
            FilteredSides.remove(L)
            continue
    
    return(FilteredSides)
\end{lstlisting}

\begin{lstlisting}[caption={Side type 4 generator}, label={prog:side-4}, basicstyle=\footnotesize\ttfamily]
def Side_Type_iv_Generator(Center_size,max_size):

    C = Center_size
    m = max_size

    # Using integer partitions to create a list of all potential sides of the form [e,[a0,a1,...,ak]]
    # The partition constraint tells us to find partitions of C+2, C+3, ..., and (C + 2m - 1)
    AllPotentialSides = []
    for n in range(C + 2 , C + 2*m):
        
        #Generate all partitions of n using integers 1, 2, ..., m
        P = Partitions(n,max_part = m).list()
        
        #Creating all permutations of the integer partitions
        for BasePart in P:
            MoreParts = Permutations(BasePart) 
            
            # Adding integer exponents
            for x in MoreParts:
                # The exponent constraint is 0 < e <= a0 and e is an integer
                a0 = x[0]
                for e in range(1,a0+1):
                    AllPotentialSides.append([e,list(x)])
                
    # Filtered Sides is intialized as a copy of AllPotentialSides.
    # We will remove sides from FilteredSides when the are caught by a filter.
    FilteredSides = []
    for i in AllPotentialSides:
        FilteredSides.append(i)
    
    for L in AllPotentialSides:

        #Remove any instances where L has ai = a(i+1) for some i = 0, ..., n-2
        flag = False
        for i in range(len(L[1])-1):
            if L[1][i] == L[1][i+1]:
                flag = True
                break
        if flag == True:
            FilteredSides.remove(L)
            continue

        #Remove any instances where L contains "a,1,b" or "a,2,b" since any such instance
        #violates the vortex condition for size-1 and size-2 squares
        flag = False
        for i in range(1,len(L[1])-1):
            if (L[1][i] == 1) or (L[1][i] == 2):
                flag = True
                break
        if flag == True:
            FilteredSides.remove(L)
            continue

        #Remove all double corners of size 1 or 2:
        if (L[0] == 1 and L[1][0] == 1) or (L[0] == 2 and L[1][0] == 2):
            FilteredSides.remove(L)
            continue

        #Remove any instances in which a0 = e and a1 = C (violates unilaterality)
        if len(L[1]) > 1:
            if L[0] == L[1][0] and L[1][1] == C:
                FilteredSides.remove(L)
                continue

        # The afterhang y must satisfy 0 < y < ak (where ak is the last square)
        afterhang = sum(L[1][i] for i in range(len(L[1]))) - L[0] - C
        last_square = L[1][-1]
        if afterhang <= 0 or afterhang >= last_square:
            FilteredSides.remove(L)
            continue
    
    return(FilteredSides)
\end{lstlisting}

\subsection{Corona Filters}\label{apx:stitching}\mbox{}\\

\begin{lstlisting}[caption={Check for Unilaterality Violation When Combining Sides}, label={prog:uni-glue}, basicstyle=\footnotesize\ttfamily]
def Uni_Glue(Cor):

# Uni_Glue checks if unilaterality is violated where the end square of one side meets the beginning square of the next side
# Note: Such a unilaterality violation can only occur at double corners.
# If the coronal violates unilaterality at one of these places, Uni_Glue returns True; otherwise False.

    flag = False
    if (Cor[0][1][-1] == Cor[1][1][0] and Cor[1][0] == Cor[1][1][0]) or (Cor[1][1][-1] == Cor[2][1][0] and Cor[2][0] == Cor[2][1][0]) or (Cor[2][1][-1] == Cor[3][1][0] and Cor[3][0] == Cor[3][1][0]) or  (Cor[3][1][-1] == Cor[0][1][0] and Cor[0][0] == Cor[0][1][0]):
        flag = True
    return(flag)  
\end{lstlisting}

\begin{lstlisting}[caption={Check for Vortex Violations of Size 1 and 2 Squares}, label={prog:vortex-overhang}, basicstyle=\footnotesize\ttfamily]
def Vortex_Overhang(Cor,center_size):

# This filter will make sure that when there is a 1-square or 2-square (which are vortices)
# at the beginning or end of a side, the vortex condition is not violated
# where the 1-square or 2-square meets the tile on the preceeding/next side.
# Returns True if there are any violations, False otherwise



    C = center_size
    
    for i in range(4):
        
        cur_side = Cor[i]
        prev_side = Cor[i-1]
        next_side = Cor[(i+1)%4]

        prev_side_length = len(prev_side[1])
        cur_side_length = len(cur_side[1])

        #Checking Prehang for vortex violation of size 1 or 2 squares:

        first_square = cur_side[1][0]
        prev_square = prev_side[1][prev_side_length - 1]
        prehang = cur_side[0]
        
        if (prev_square == 1 and prehang > 1):
            return(True)
        if (prev_square == 2 and prehang > 2):
            return(True)
        if (C != 1 and prehang == 1 and first_square == 2 and prev_square != 1):
            return(True)

        # Checking Afterhang for vortex violation of size 1 or 2 squares:

        last_square = cur_side[1][cur_side_length - 1]
        next_square = next_side[1][0]
        afterhang = sum(cur_side[1][j] for j in range(0,cur_side_length)) - C - prehang
        
        if (next_square == 1 and afterhang > 1):
            return(True)
        if (next_square == 2 and afterhang > 2):
            return(True)
        if (C != 1 and afterhang == 1 and last_square == 2 and next_square != 1):
            return(True)
    
    return(False) 
\end{lstlisting}

\begin{lstlisting}[caption={Symmetry Filters}, label={prog:symmetry}, basicstyle=\footnotesize\ttfamily]
def Rot_90_Filter(Corona1,Corona2):
    C1 = list(Corona1)
    C2 = list(Corona2)

    if C1[0] == C2[1] and C1[1] == C2[2] and C1[2] == C2[3] and C1[3] == C2[0]:
        return True
    else:
        return False

def Rot_180_Filter(Corona1,Corona2):
    C1 = list(Corona1)
    C2 = list(Corona2)

    if C1[0] == C2[2] and C1[1] == C2[3] and C1[2] == C2[0] and C1[3] == C2[1]:
        return True
    else:
        return False

def Rot_270_Filter(Corona1,Corona2):
    C1 = list(Corona1)
    C2 = list(Corona2)

    if C1[0] == C2[3] and C1[1] == C2[0] and C1[2] == C2[1] and C1[3] == C2[2]:
        return True
    else:
        return False

def Refl_V_Filter(Corona1,Corona2,center_size):
    C1 = list(Corona1)
    C2 = list(Corona2)
    C = center_size

    C1_top_afterhang = sum(C1[0][1][j] for j in range(0,len(C1[0][1]))) - C - C1[0][0]
    C1_bot_afterhang = sum(C1[2][1][j] for j in range(0,len(C1[2][1]))) - C - C1[2][0]

    #Creating a copy of Corona1 flipped over the vertical
    C1_flip = [[C1_top_afterhang,list(reversed(C1[0][1]))],[0,list(reversed(C1[3][1]))],[C1_bot_afterhang,list(reversed(C1[2][1]))],[0,list(reversed(C1[1][1]))]]
    
    if C1_flip == C2:
        return True
    else:
        return False  

def Refl_H_Filter(Corona1,Corona2,center_size):
    C1 = list(Corona1)
    C2 = list(Corona2)
    C = center_size

    C1_top_afterhang = sum(C1[0][1][j] for j in range(0,len(C1[0][1]))) - C - C1[0][0]
    C1_left_afterhang = sum(C1[3][1][j] for j in range(0,len(C1[3][1]))) - C - C1[3][0]

    #Creating a copy of Corona1 flipped over the horizontal
    C1_flip = [[0,list(reversed(C1[2][1]))],[0,list(reversed(C1[1][1]))],[C1_top_afterhang,list(reversed(C1[0][1]))],[C1_left_afterhang,list(reversed(C1[3][1]))]]
    
    if C1_flip == C2:
        return True
    else:
        return False  

# Check if Corona2 is really just Corona 1 reflected over the diagonal of negative slope
# This filter only applies to Archetype 1 and only occurs when n > 4

def Refl_D_Filter(Corona1,Corona2,center_size):
    C1 = list(Corona1)
    C2 = list(Corona2)
    C = center_size

    C1_bot_afterhang = sum(C1[2][1][j] for j in range(0,len(C1[2][1]))) - C - C1[2][0]

    new_top = list(reversed(C1[3][1]))
    new_top.insert(0,C1[0][0])

    new_right = list(reversed(C1[2][1]))
    new_right.pop()

    new_bot = list(reversed(C1[1][1]))
    new_bot.insert(0,C1[2][0])

    new_left = list(reversed(C1[0][1]))
    new_left.pop()

    C1_flip = [[C1[0][0],new_top],[C1_bot_afterhang,new_right],[C1[2][0],new_bot],[0,new_left]]
    
    if C1_flip == C2:
        return True
    else:
        return False  

\end{lstlisting}

\subsection{Corona Generation}\label{apx:corona_gen}\mbox{}\\

\begin{lstlisting}[caption={Generating Coronas and Applying Corona Filters}, label={prog:corona_gen}, basicstyle=\footnotesize\ttfamily]
def Corona_Generator(Center_size,max_size):

    C = Center_size
    m = max_size

    # Creating all side types (i - iv)
    S1 = Side_Type_i_Generator(C,m)
    S2 = Side_Type_ii_Generator(C,m)
    S3 = Side_Type_iii_Generator(C,m)
    S4 = Side_Type_iv_Generator(C,m)

    # Separating each side types into those whose leading square creates
    # a double corner and those that do not
    S1_Double = []
    S1_Single = []
    S2_Double = []
    S2_Single = []
    S3_Double = []
    S3_Single = []
    S4_Double = []
    S4_Single = []

    for S in S1:
        if S[0] == S[1][0]:
            S1_Double.append(S)
        else:
            S1_Single.append(S)

    for S in S2:
        if S[0] == S[1][0]:
            S2_Double.append(S)
        else:
            S2_Single.append(S)

    for S in S3:
        if S[0] == S[1][0]:
            S3_Double.append(S)
        else:
            S3_Single.append(S)

    for S in S4:
        if S[0] == S[1][0]:
            S4_Double.append(S)
        else:
            S4_Single.append(S)
    
    
    # Type 0: (iv*,iii,iv*,iii)
    T0_Coronas_All = list(it.product(S4_Double,S3_Single,S4_Double,S3_Single))
    T0_Coronas_Filtered = []
    
    for Cor in T0_Coronas_All:

        if Uni_Glue(Cor) == True:
            continue
            
        if Vortex_Overhang(Cor,C) == True:
            continue

        T0_Coronas_Filtered.append(list(Cor))

    #Getting rid of 180-degree rotationally symmetric copies
    T0_Rot180_Filtered = []
    for Cor in T0_Coronas_Filtered:
        T0_Rot180_Filtered.append(Cor)
    
    for i in range(len(T0_Coronas_Filtered)):
        for j in range(i+1,len(T0_Coronas_Filtered)):
            if Rot_180_Filter(T0_Coronas_Filtered[i],T0_Coronas_Filtered[j]) == True:
                T0_Rot180_Filtered.remove(T0_Coronas_Filtered[j])
        
 
    # Type 1: (ii*,ii,iv*,iii)
    T1_Coronas_All = list(it.product(S2_Double,S2_Single,S4_Double,S3_Single))
    T1_Coronas_Filtered = []
    
    for Cor in T1_Coronas_All:
        if Uni_Glue(Cor) == True:
            continue
        if Vortex_Overhang(Cor,C) == True:
            continue

        T1_Coronas_Filtered.append(list(Cor))

    #Getting rid of coronas that are reflections over the diagonal of negative slope of a previous corona
    T1_Refl_D_Filtered = list(T1_Coronas_Filtered)
    
    for i in range(len(T1_Coronas_Filtered)):
        for j in range(i+1,len(T1_Coronas_Filtered)):
            if Refl_D_Filter(T1_Coronas_Filtered[i],T1_Coronas_Filtered[j],C) == True:
                T1_Refl_D_Filtered.remove(T1_Coronas_Filtered[j])
            
 
    # Type 2: (ii*,ii,ii,ii)
    T2_Coronas_All = list(it.product(S2_Double,S2_Single,S2_Single,S2_Single))
    T2_Coronas_Filtered = []
    
    for Cor in T2_Coronas_All:
        if Uni_Glue(Cor) == True:
            continue
        if Vortex_Overhang(Cor,C) == True:
            continue

        T2_Coronas_Filtered.append(list(Cor))
 
    
    # Type 3: (iv*,iii,ii,ii)
    T3_Coronas_All = list(it.product(S4_Double,S3_Single,S2_Single,S2_Single))
    T3_Coronas_Filtered = []
    
    for Cor in T3_Coronas_All:
        if Uni_Glue(Cor) == True:
            continue
        if Vortex_Overhang(Cor,C) == True:
            continue

        T3_Coronas_Filtered.append(list(Cor))

    
    # Type 4: (ii*,iv,iii,ii)
    T4_Coronas_All = list(it.product(S2_Double,S4_Single,S3_Single,S2_Single))
    T4_Coronas_Filtered = []
    
    for Cor in T4_Coronas_All:
        if Uni_Glue(Cor) == True:
            continue
        if Vortex_Overhang(Cor,C) == True:
            continue

        T4_Coronas_Filtered.append(list(Cor))
 
    
    # Type 5: (ii*,ii,iv,iii)
    T5_Coronas_All = list(it.product(S2_Double,S2_Single,S4_Single,S3_Single))
    T5_Coronas_Filtered = []
    
    for Cor in T5_Coronas_All:
        if Uni_Glue(Cor) == True:
            continue
        if Vortex_Overhang(Cor,C) == True:
            continue

        T5_Coronas_Filtered.append(list(Cor))
    
    #-----------------------

    
    # Type 6: (ii,ii,ii,ii)
    T6_Coronas_All = list(it.product(S2_Single,S2_Single,S2_Single,S2_Single))
    T6_Coronas_Filtered = []

    for Cor in T6_Coronas_All:
        if Vortex_Overhang(Cor,C) == True:
            continue

        T6_Coronas_Filtered.append(list(Cor))

    #Getting rid of all rotationally symmetric copies
    T6_Rot90_Filtered = list(T6_Coronas_Filtered)
    
    for i in range(len(T6_Coronas_Filtered)):
        for j in range(i+1,len(T6_Coronas_Filtered)):
            if Rot_90_Filter(T6_Coronas_Filtered[i],T6_Coronas_Filtered[j]):
                T6_Rot90_Filtered.remove(T6_Coronas_Filtered[j])

    T6_Rot180_Filtered = list(T6_Rot90_Filtered)

    for i in range(len(T6_Rot90_Filtered)):
        for j in range(i+1,len(T6_Rot90_Filtered)):
            if Rot_180_Filter(T6_Rot90_Filtered[i],T6_Rot90_Filtered[j]):
                T6_Rot180_Filtered.remove(T6_Rot90_Filtered[j])

    T6_Rot270_Filtered = list(T6_Rot180_Filtered)

    for i in range(len(T6_Rot180_Filtered)):
        for j in range(i+1,len(T6_Rot180_Filtered)):
            if Rot_270_Filter(T6_Rot180_Filtered[i],T6_Rot180_Filtered[j]):
                T6_Rot270_Filtered.remove(T6_Rot180_Filtered[j])
 
    
    # Type 7: (i,iii,iv,i)
    T7_Coronas_All = list(it.product(S1_Single,S3_Single,S4_Single,S1_Single))
    T7_Coronas_Filtered = []

    for Cor in T7_Coronas_All:
        if Vortex_Overhang(Cor,C) == True:
            continue

        T7_Coronas_Filtered.append(list(Cor))
   
    
    # Type 8: (i,iii,ii,iv)
    T8_Coronas_All = list(it.product(S1_Single,S3_Single,S2_Single,S4_Single))
    T8_Coronas_Filtered = []

    for Cor in T8_Coronas_All:
        if Vortex_Overhang(Cor,C) == True:
            continue

        T8_Coronas_Filtered.append(list(Cor))

    #Getting rid of coronas that are reflections over the horizontal of a previous corona
    T8_Refl_H_Filtered = list(T8_Coronas_Filtered)
        
    for i in range(len(T8_Coronas_Filtered)):
        for j in range(i+1,len(T8_Coronas_Filtered)):
            if Refl_H_Filter(T8_Coronas_Filtered[i],T8_Coronas_Filtered[j],C) == True:
                T8_Refl_H_Filtered.remove(T8_Coronas_Filtered[j])
 
    
    # Type 9: (iv,iii,iv,iii)
    T9_Coronas_All = list(it.product(S4_Single,S3_Single,S4_Single,S3_Single))
    T9_Coronas_Filtered = []

    for Cor in T9_Coronas_All:
        if Vortex_Overhang(Cor,C) == True:
            continue

        T9_Coronas_Filtered.append(list(Cor))
        

    #Getting rid of coronas that are reflections over the vertical and 180-degree rotatations of a previous corona
    T9_Refl_V_Filtered = list(T9_Coronas_Filtered)
    
    for i in range(len(T9_Coronas_Filtered)):
        for j in range(i+1,len(T9_Coronas_Filtered)):
            if Refl_V_Filter(T9_Coronas_Filtered[i],T9_Coronas_Filtered[j],C) == True:
                T9_Refl_V_Filtered.remove(T9_Coronas_Filtered[j])
    T9_Rot180_Filtered = list(T9_Refl_V_Filtered)

    for i in range(len(T9_Refl_V_Filtered)):
        for j in range(i+1,len(T9_Refl_V_Filtered)):
            if Rot_180_Filter(T9_Refl_V_Filtered[i],T9_Refl_V_Filtered[j]) == True:
                T9_Rot180_Filtered.remove(T9_Refl_V_Filtered[j])

    All_Coronas = [T0_Rot180_Filtered,T1_Refl_D_Filtered,T2_Coronas_Filtered,T3_Coronas_Filtered,T4_Coronas_Filtered,T5_Coronas_Filtered,T6_Rot270_Filtered,T7_Coronas_Filtered,T8_Refl_H_Filtered,T9_Rot180_Filtered]

    # Adding the center size and archetype information to each corona.
    # This will be handy later that information is always on hand.
    # The data type for a corona will be: [ center_size, archetype, [top_side, right_side, bottom_side, left_side] ]
    # All_Coronas_Amended should be lengthed if Archetypes 10-14 are added later
    All_Coronas_Amended = [[],[],[],[],[],[],[],[],[],[]]
    
    for i in range(len(All_Coronas)):
        for j in All_Coronas[i]:
            All_Coronas_Amended[i].append([C,i,j])
    
    return(All_Coronas_Amended)



#Producing all filtered 1-, 2-, 3-, and 4-coronas of Types 1-10:

max_size = 4

Cors1 = Corona_Generator(1,max_size)
FlatCors1 = Cors1[0]+Cors1[1]+Cors1[2]+Cors1[3]+Cors1[4]+Cors1[5]+Cors1[6]+Cors1[7]+Cors1[8]+Cors1[9]

Cors2 = Corona_Generator(2,max_size)
FlatCors2 = Cors2[0]+Cors2[1]+Cors2[2]+Cors2[3]+Cors2[4]+Cors2[5]+Cors2[6]+Cors2[7]+Cors2[8]+Cors2[9]

Cors3 = Corona_Generator(3,max_size)
FlatCors3 = Cors3[0]+Cors3[1]+Cors3[2]+Cors3[3]+Cors3[4]+Cors3[5]+Cors3[6]+Cors3[7]+Cors3[8]+Cors3[9]

Cors4 = Corona_Generator(4,max_size)
FlatCors4 = Cors4[0]+Cors4[1]+Cors4[2]+Cors4[3]+Cors4[4]+Cors4[5]+Cors4[6]+Cors4[7]+Cors4[8]+Cors4[9]

All_Corona_4_Tups = list(it.product(FlatCors1,FlatCors2,FlatCors3,FlatCors4))

\end{lstlisting}

\pagebreak

\subsection{Corona Compatibility Checking}\label{apx:compat}\mbox{}\\

\begin{lstlisting}[caption={Corona Compatibibility Checking}, label={prog:compat}, basicstyle=\footnotesize\ttfamily]
def Geom_Compat_Check(corA,corB):
    

    # We are checking if CorA and CorB are compatible using a more strengent check on the geometric arrangments
    # of neighbors of squares.

    # We allow that corA = corB.
    
    a = corA[0]
    b = corB[0]
    
    TypeA = corA[1]
    TypeB = corB[1]
 
    # Precomputing some things that get used over and over in the loop to follow:
    # Finding number of squares on each side of corA and corB and storing in corresponding lists.
    # Also calculating the overhang of each side of corA and corB and storing in lists overhangsA and overhangsB.
    # Note: the prehang is just the exponent of each side which is the 0th entry of each side
    # Finding the positions of instances of "a" in corB and instances of "b" in corA
    
    corA_side_num = []
    corB_side_num = []  
    overhangsA = []
    overhangsB = []
    Pos_of_a_in_corB = []
    Pos_of_b_in_corA = []  
    
    for i in range(4):
        corA_side = corA[2][i][1]
        corB_side = corB[2][i][1]
        corA_side_num.append(len(corA_side))
        corB_side_num.append(len(corB_side))
        expA = corA[2][i][0]
        expB = corB[2][i][0]
        overhangA = sum(corA_side) - expA - a
        overhangB = sum(corB_side) - expB - b
        overhangsA.append(overhangA)
        overhangsB.append(overhangB)
        Pos_of_a_in_corB += [[i,j] for j, x in enumerate(corB_side) if x == a]
        Pos_of_b_in_corA += [[i,j] for j, x in enumerate(corA_side) if x == b]

    # Making a reflected copy of corA, corA_refl and storing its side_nums and overhangs
    corA_refl = [corA[0],corA[1]] 
    corA_refl_squares = []
    for i in range(4):
        corA_refl_squares.append([overhangsA[(-1*i) % 4],corA[2][(-1*i) % 4][1][::-1]])
    corA_refl.append(corA_refl_squares)

    # Adjusting reflected copy of corA to account for double corners
    for i in range(4):
        corA_squares = corA[2][i][1]
        expA = corA[2][i][0]
        if expA == corA_squares[0]:
            corA_refl[2][(i + 1) % 4][1].insert(0,expA)
            corA_refl[2][(i + 1) % 4][0] = expA
            L = len(corA_refl[2][i][1])
            del corA_refl[2][i][1][L - 1]

    corA_refl_side_num = []
    overhangsA_refl = []
    Pos_of_b_in_corA_refl = []
    for i in range(4):
        corA_refl_side = corA_refl[2][i][1]
        corA_refl_side_num.append(len(corA_refl_side))
        expA_refl = corA_refl[2][i][0]
        overhangA_refl = sum(corA_refl_side) - expA_refl - a
        overhangsA_refl.append(overhangA_refl)
        Pos_of_b_in_corA_refl += [[i,j] for j, x in enumerate(corA_refl_side) if x == b]
        
        
    # If one of Pos_of_a_in_corB or Pos_of_b_in_corA is empty but the other is not
    # then we know these two coronas are incompatible
    if (not Pos_of_a_in_corB and Pos_of_b_in_corA) or (Pos_of_a_in_corB and not Pos_of_b_in_corA):
        return(False)

    # If both Pos_of_a_in_corB and Pos_of_b_in_corA are empty, the two coronas are vacuously compatible, so return True
    if not Pos_of_a_in_corB and not Pos_of_b_in_corA:
        return(True)

    
    # Corona corA is the reference corona.
    # For each instance of "a" in CorB, we will check if the geometric
    # arrangement of the squares around "a" in CorB has a match in corA

    for q in Pos_of_a_in_corB:

        # Note: "q" is a 2-tuple [i,j] where i = the index of the side containing "a" and j = index of "a" on that side

        # Setting some variables based on corB
        cur_side = corB[2][q[0]][1] #Note - this is the list of squares w/o exponent
        cur_side_prehang = corB[2][q[0]][0]
        cur_side_overhang = overhangsB[q[0]]
        prev_side = corB[2][(q[0] - 1) % 4][1] #Note - this is the list of squares w/o exponent
        prev_side_prehang = corB[2][(q[0] - 1) % 4][0]
        prev_side_overhang = overhangsB[(q[0] - 1) % 4]
        next_side = corB[2][(q[0] + 1) % 4][1] #Note - this is the list of squares w/o exponent
        next_side_prehang = corB[2][(q[0] + 1) % 4][0]
        curL = corB_side_num[q[0]]
        prevL = corB_side_num[(q[0] - 1) % 4]
        nextL = corB_side_num[(q[0] + 1) % 4]


        # There are 4 major cases:
        # (1) "a" appears as the only square of cur_side
        # (2) "a" appears as a 1st but not only square of cur_side 
        # (3) "a" appears as a middle square of cur_side (so has and immediate predecessor and successor on curside)
        # (4) "a" appears as the last but not only square of cur_side (so is not 1st)
        
        # Case (1): "a" appears as the only square of cur_side
        if curL == 1:
                       
            # There are a few subcases to consider:
            # (i) "a" has prehang of 0 (and positive overhang)
            # (ii) "a" has positive prehang and 0 overhang with double corner at the end
            # (iii) "a" has positive prehang and 0 overhang but no double corner at the end
            # (iv) "a" has postive prehang and positive overhang.

            # Subcase (i)
            if cur_side_prehang == 0:

                n = next_side[0]
                p = prev_side[prevL - 1]
               
                # in corA, we must find all instances where "b" is the last square
                # of a side. If no such instances exist, the coronas are incompatible
                # and we may return False.

                pos_of_b_in_corA = []
                pos_of_b_in_corA_refl = []
                for i in range(4):
                    corA_side = corA[2][i][1]
                    corA_refl_side = corA_refl[2][i][1]
                    pos_of_b_in_corA += [[i,j] for j, x in enumerate(corA_side) if x == b and j == corA_side_num[i] - 1]
                    pos_of_b_in_corA_refl += [[i,j] for j, x in enumerate(corA_refl_side) if x == b and j == corA_refl_side_num[i] - 1]
                if not pos_of_b_in_corA and not pos_of_b_in_corA_refl:
                    return(False)

                # For each occurrence of "b" identified in pos_of_b_in_corA, we need that each of the following are satisified.
                # If it never happens that each are satisifed, corA and corB are incompatible so return False
                # (1) The overhang of the side on which "b" lies must be 0.
                # (2) The number of squares on this side must be greater than 1
                # (3) "n" must be the next to last square of the side on which "b" lies
                # (4) "p" must be the first square of the next side and that side must have postive prehang
                # (5) The overlap of "p" and "a" should be equal in both corA and corB

                found_last_b_after_n_in_corA = []
                for i in pos_of_b_in_corA:
                    # Checking (1) and (2)
                    if overhangsA[i[0]] == 0 and corA_side_num[i[0]] > 1:
                        side_of_b_in_corA = corA[2][i[0]]
                        side_after_b_in_corA = corA[2][(i[0] + 1) % 4]
                        #Checking (3) and (4)
                        if side_of_b_in_corA[1][corA_side_num[i[0]] - 2] == n and side_after_b_in_corA[1][0] == p and side_after_b_in_corA[0] > 0:
                            overlap_of_p_and_a_in_corA = p - side_after_b_in_corA[0]
                            overlap_of_p_and_a_in_corB = prev_side_overhang
                            found_last_b_after_n_in_corA.append(i)
                            break
                if not found_last_b_after_n_in_corA:
                    for i in pos_of_b_in_corA_refl:
                        # Checking (1) and (2)
                        if overhangsA_refl[i[0]] == 0 and corA_refl_side_num[i[0]] > 1:
                            side_of_b_in_corA_refl = corA_refl[2][i[0]]
                            side_after_b_in_corA_refl = corA_refl[2][(i[0] + 1) % 4]
                            #Checking (3) and (4)
                            if side_of_b_in_corA_refl[1][corA_refl_side_num[i[0]] - 2] == n and side_after_b_in_corA_refl[1][0] == p and side_after_b_in_corA_refl[0] > 0:
                                overlap_of_p_and_a_in_corA_refl = p - side_after_b_in_corA_refl[0]
                                overlap_of_p_and_a_in_corB = prev_side_overhang
                                found_last_b_after_n_in_corA.append(i)
                                break

                if not found_last_b_after_n_in_corA:
                    #print("Returning False in Case 1(i)")
                    return(False)

            # Subcase (ii): "a" has positive prehang and 0 overhang with double corner at the end
                        
            if cur_side_prehang > 0 and cur_side_overhang == 0 and next_side_prehang == next_side[0]:
                
                n = next_side[0]
                #Note: Since next_side_prehang = next_side[0] we know len(next_side) > 1
                nn = next_side[1]
                p = prev_side[prevL - 1]

                # In this case, in corA we need to find that each of the following hold, otherwise
                # we will return False:
                # (1) "n" must be the last square of a side of corA with overhang equal to 0 on that side
                # (2) "nn" must be the first square of the next side with prehang equal to "nn" on that side
                # (3) "b" must be the second square of the next side
                # (4) "p" must be the third square of the next side

                # Finding all occurances of "n" in corA that are the last square of a side of corA
                pos_of_n_in_corA = []
                pos_of_n_in_corA_refl = []
                for i in range(4):
                    corA_side = corA[2][i][1]
                    corA_refl_side = corA_refl[2][i][1]
                    pos_of_n_in_corA += [[i,j] for j, x in enumerate(corA_side) if x == n and j == corA_side_num[i] - 1]   
                    pos_of_n_in_corA_refl += [[i,j] for j, x in enumerate(corA_refl_side) if x == n and j == corA_refl_side_num[i] - 1]  
                if not pos_of_n_in_corA and not pos_of_n_in_corA_refl:
                    return(False)

                found_next_nn_in_corA = []
                for i in pos_of_n_in_corA:
                    L = corA_side_num[(i[0] + 1) % 4]
                    if L > 2 and overhangsA[i[0]] == 0:
                        next_side_in_corA = corA[2][(i[0] + 1) % 4]
                        next_side_in_corA_prehang = next_side_in_corA[0]
                        if next_side_in_corA_prehang == nn and next_side_in_corA[1][0] == nn and next_side_in_corA[1][1] == b and next_side_in_corA[1][2] == p:
                            found_next_nn_in_corA.append(i)
                            break
                
                if not found_next_nn_in_corA:
                    for i in pos_of_n_in_corA_refl:
                        L = corA_refl_side_num[(i[0] + 1) % 4]
                        if L > 2 and overhangsA_refl[i[0]] == 0:
                            next_side_in_corA_refl = corA_refl[2][(i[0] + 1) % 4]
                            next_side_in_corA_refl_prehang = next_side_in_corA_refl[0]
                            if next_side_in_corA_refl_prehang == nn and next_side_in_corA_refl[1][0] == nn and next_side_in_corA_refl[1][1] == b and next_side_in_corA_refl[1][2] == p:
                                found_next_nn_in_corA.append(i)
                                break
                
                if not found_next_nn_in_corA:
                    return(False) 

            # Subcase (iii) "a" has positive prehang and 0 overhang but no double corner at the end
            
            if cur_side_prehang > 0 and cur_side_overhang == 0 and next_side_prehang < next_side[0]:
                n = next_side[0]
                p = prev_side[prevL - 1]

                # In this case, we must check that 
                # (1) "n" is the last square of at some side of corA with positive overhang
                # (2) "b" is the first square of the next side of corA with prehang = 0
                # (3) "p" is the next square on the same side as "b" in corA
                # (4) the prehang of "n" in corB equals the overlap of "n" and "a" in corA.
                # If no such "n" is identified, then corA and corB are not compatible - return False

                pos_of_n_in_corA = []
                pos_of_n_in_corA_refl = []
                for i in range(4):
                    corA_side = corA[2][i][1]
                    corA_refl_side = corA_refl[2][i][1]
                    pos_of_n_in_corA += [[i,j] for j, x in enumerate(corA_side) if x == n and j == corA_side_num[i] - 1]
                    pos_of_n_in_corA_refl += [[i,j] for j, x in enumerate(corA_refl_side) if x == n and j == corA_refl_side_num[i] - 1]  
                if not pos_of_n_in_corA and not pos_of_n_in_corA_refl:
                    return(False)

                found_b_after_n_in_corA = []
                for i in pos_of_n_in_corA:
                    next_side_after_n = corA[2][(i[0] + 1) % 4]
                    L = corA_side_num[(i[0] + 1) % 4]
                    next_side_after_n_prehang = next_side_after_n[0]
                    
                    # Checking conditions (1) - (4)
                    if L > 1 and next_side_after_n_prehang == 0:
                        prehang_of_n_in_corB = next_side_prehang
                        overlap_of_a_and_n_in_corA = n - overhangsA[i[0]]
                        if next_side_after_n[1][0] == b and next_side_after_n[1][1] == p and prehang_of_n_in_corB == overlap_of_a_and_n_in_corA:
                            found_b_after_n_in_corA.append(i)
                            break
                if not found_b_after_n_in_corA:
                    for i in pos_of_n_in_corA_refl:
                        next_side_after_n = corA_refl[2][(i[0] + 1) % 4]
                        L = corA_refl_side_num[(i[0] + 1) % 4]
                        next_side_after_n_prehang = next_side_after_n[0]
                        
                        # Checking conditions (1) - (4)
                        if L > 1 and next_side_after_n_prehang == 0:
                            prehang_of_n_in_corB = next_side_prehang
                            overlap_of_a_and_n_in_corA_refl = n - overhangsA_refl[i[0]]
                            if next_side_after_n[1][0] == b and next_side_after_n[1][1] == p and prehang_of_n_in_corB == overlap_of_a_and_n_in_corA_refl:
                                found_b_after_n_in_corA.append(i)
                                break
                
                if not found_b_after_n_in_corA:
                    return(False)
            
            #Subcase (iv): "a" has positive prehang and positive overhang.  

            if cur_side_prehang > 0 and cur_side_overhang > 0:
                
                n = next_side[0]
                p = prev_side[prevL - 1]
                
                # Need to check the following hold for each occurance of "a" in corB:
                # We must find "n", "b", and "n" (in that order) along a side of corA.
                # "b" must be positioned along a side of corA so that the prehang and overhang
                # of "a" on "b" match in both corA and corB

                found_nbp_in_corA = []
                
                for i in Pos_of_b_in_corA: # Note: Pos_of_b_in_corA is precomputed at the beginning of the function
                    corA_side = corA[2][i[0]]
                    # First check that "b" is in the middle of a side of corA
                    if 0 < i[1] < corA_side_num[i[0]] - 1:
                        # Next check that "n" comes before "b" and "p" comes after "b"
                        if corA_side[1][i[1] - 1] == n and corA_side[1][i[1] + 1] == p:
                            prehang_of_a_on_b_in_corB = cur_side_prehang
                            overhang_of_a_on_b_in_corB = cur_side_overhang
                            prehang_of_a_on_b_in_corA = sum(corA_side[1][0:i[1]]) - corA_side[0]
                            overhang_of_a_on_b_in_corA = sum(corA_side[1][i[1]+1:corA_side_num[i[0]]]) - overhangsA[i[0]]
                            # Checking that the prehangs and overhangs match
                            if prehang_of_a_on_b_in_corB == overhang_of_a_on_b_in_corA and overhang_of_a_on_b_in_corB == prehang_of_a_on_b_in_corA:
                                found_nbp_in_corA.append(i)
                                break
                if not found_nbp_in_corA:
                    for i in Pos_of_b_in_corA_refl: # Note: Pos_of_b_in_corA_refl is precomputed at the beginning of the function
                        corA_refl_side = corA_refl[2][i[0]]
                        # First check that "b" is in the middle of a side of corA
                        if 0 < i[1] < corA_refl_side_num[i[0]] - 1:
                            # Next check that "n" comes before "b" and "p" comes after "b"
                            if corA_refl_side[1][i[1] - 1] == n and corA_refl_side[1][i[1] + 1] == p:
                                prehang_of_a_on_b_in_corB = cur_side_prehang
                                overhang_of_a_on_b_in_corB = cur_side_overhang
                                prehang_of_a_on_b_in_corA_refl = sum(corA_refl_side[1][0:i[1]]) - corA_refl_side[0]
                                overhang_of_a_on_b_in_corA_refl = sum(corA_refl_side[1][i[1]+1:corA_refl_side_num[i[0]]]) - overhangsA_refl[i[0]]
                                # Checking that the prehangs and overhangs match
                                if prehang_of_a_on_b_in_corB == overhang_of_a_on_b_in_corA_refl and overhang_of_a_on_b_in_corB == prehang_of_a_on_b_in_corA_refl:
                                    found_nbp_in_corA.append(i)
                                    break
                if not found_nbp_in_corA:
                    return(False)
    
        # Case (2) "a" appears as a 1st but not only square of cur_side.

        if curL > 1 and q[1] == 0:

            # There are 2 subcase based on the prehang of "a" (cur_side_prehang) in corB:
            # (i) 0 < cur_side_prehang <= a
            # (ii) cur_side_prehang = 0

            # Subcase (i):
            if 0 < cur_side_prehang <= a:

                n = cur_side[1]
                p = prev_side[prevL - 1]
                
                # In this subcase, we need to check that
                # "n" is the last square of a side of corA
                # the overhang of "n" on that side of corA is 0
                # "b" is the 1st square of the next side of corA after the side containing this "n" with postitive prehang
                # "p" is the 2nd square of this next side (right after "b")
                # the overlap of "a" and "b" matches in corA and corB

                
                pos_of_n_in_corA = []
                pos_of_n_in_corA_refl = []
                for i in range(4):
                    corA_side = corA[2][i][1]
                    corA_refl_side = corA_refl[2][i][1]
                    pos_of_n_in_corA += [[i,j] for j, x in enumerate(corA_side) if x == n and j == corA_side_num[i] - 1] 
                    pos_of_n_in_corA_refl += [[i,j] for j, x in enumerate(corA_refl_side) if x == n and j == corA_refl_side_num[i] - 1] 
                if not pos_of_n_in_corA and not pos_of_n_in_corA_refl:
                    return(False)

                found_b_and_p_after_n_in_corA = []
                for i in pos_of_n_in_corA:
                    next_side_after_n = corA[2][(i[0] + 1) % 4]
                    L = corA_side_num[(i[0] + 1) % 4]
                    afterhang_of_n_in_corA = overhangsA[i[0]]
                    next_side_after_n_prehang = next_side_after_n[0]
                    if L > 1 and afterhang_of_n_in_corA == 0 and next_side_after_n_prehang > 0:
                        if b == next_side_after_n[1][0] and p == next_side_after_n[1][1]:
                            overlap_of_a_and_b_in_corA = b - next_side_after_n_prehang
                            overlap_of_a_and_b_in_corB = a - cur_side_prehang
                            if overlap_of_a_and_b_in_corA == overlap_of_a_and_b_in_corB:
                                found_b_and_p_after_n_in_corA.append(i)
                                break
                if not found_b_and_p_after_n_in_corA:
                    for i in pos_of_n_in_corA_refl:
                        next_side_after_n = corA_refl[2][(i[0] + 1) % 4]
                        L = corA_refl_side_num[(i[0] + 1) % 4]
                        afterhang_of_n_in_corA_refl = overhangsA_refl[i[0]]
                        next_side_after_n_prehang = next_side_after_n[0]
                        if L > 1 and afterhang_of_n_in_corA_refl == 0 and next_side_after_n_prehang > 0:
                            if b == next_side_after_n[1][0] and p == next_side_after_n[1][1]:
                                overlap_of_a_and_b_in_corA_refl = b - next_side_after_n_prehang
                                overlap_of_a_and_b_in_corB = a - cur_side_prehang
                                if overlap_of_a_and_b_in_corA_refl == overlap_of_a_and_b_in_corB:
                                    found_b_and_p_after_n_in_corA.append(i)
                                    break

                if not found_b_and_p_after_n_in_corA:
                    return(False)
                    
            # subcase (ii)
            if cur_side_prehang == 0:

                n = cur_side[1]
                p = prev_side[prevL - 1]

                # We must check that the following all hold:
                # (1) "n" appears as the last square of a side of "a" with no overhang
                # (2) for each such "n", "b" must be the 1st (and only) square of the next side with no overhang
                # (3) "p" must be the first square of the side after "b"
                # (4) The overlap of "p" and "b" in corB must equal the prehang of "p" in corA

                pos_of_n_in_corA = []
                pos_of_n_in_corA_refl = []
                for i in range(4):
                    corA_side = corA[2][i][1]
                    corA_refl_side = corA_refl[2][i][1]
                    pos_of_n_in_corA += [[i,j] for j, x in enumerate(corA_side) if x == n and j == corA_side_num[i] - 1]  
                    pos_of_n_in_corA_refl += [[i,j] for j, x in enumerate(corA_refl_side) if x == n and j == corA_refl_side_num[i] - 1]
                    
                if not pos_of_n_in_corA and not pos_of_n_in_corA_refl:
                    return(False)

                found_b_and_p_after_n_in_corA = []
                for i in pos_of_n_in_corA:
                    next_side_after_n = corA[2][(i[0] + 1) % 4]
                    L = corA_side_num[(i[0] + 1) % 4]
                    next_side_after_n_overhang = overhangsA[(i[0] + 1) % 4]
                    afterhang_of_n_in_corA = overhangsA[i[0]]
                    if afterhang_of_n_in_corA == 0 and L == 1 and b == next_side_after_n[1][0] and next_side_after_n_overhang == 0:
                        next_next_side_after_n = corA[2][(i[0] + 2) % 4]
                        next_next_side_after_n_prehang = next_next_side_after_n[0]
                        overlap_of_p_and_b_in_corB = p - prev_side_overhang
                        if next_next_side_after_n[1][0] == p and next_next_side_after_n_prehang == overlap_of_p_and_b_in_corB:
                            found_b_and_p_after_n_in_corA.append(i)
                            break
                if not found_b_and_p_after_n_in_corA:
                    for i in pos_of_n_in_corA_refl:
                        next_side_after_n = corA_refl[2][(i[0] + 1) % 4]
                        L = corA_refl_side_num[(i[0] + 1) % 4]
                        next_side_after_n_overhang = overhangsA_refl[(i[0] + 1) % 4]
                        afterhang_of_n_in_corA_refl = overhangsA_refl[i[0]]
                        if afterhang_of_n_in_corA_refl == 0 and L == 1 and b == next_side_after_n[1][0] and next_side_after_n_overhang == 0:
                            next_next_side_after_n = corA_refl[2][(i[0] + 2) % 4]
                            next_next_side_after_n_prehang = next_next_side_after_n[0]
                            overlap_of_p_and_b_in_corB = p - prev_side_overhang
                            if next_next_side_after_n[1][0] == p and next_next_side_after_n_prehang == overlap_of_p_and_b_in_corB:
                                found_b_and_p_after_n_in_corA.append(i)
                                break

                if not found_b_and_p_after_n_in_corA:
                    return(False)
                   
        # Case (3) "a" is a middle square of a side.

        if curL > 2 and 0 < q[1] < curL - 1:           

            # There are two subcases:
            # Subcase (i): "a" is the second square of a side in corB and the first square of that side
            # forms a double corner at the beginning of that side.
            # subcase (ii) "a" is a middle square different than described in subcase (i)

            # Subcase (i):
            if q[1] == 1 and cur_side[0] == cur_side_prehang:

                n = cur_side[2]
                p = cur_side[0]
                pp = prev_side[prevL - 1]

                # We must check that the following all hold:
                # (1) "n" appears as the last square of a side of corA with no overhang
                # (2) for each such "n", "b" must be the 1st (and only) square of the next side with no overhang
                # (3) "pp" must be the first square of the side after "b" and with prehang equal to pp (so forms a double corner)
                # (4) "p" must be the second square of the side after "b" (coming immediately after "p")

                pos_of_n_in_corA = []
                pos_of_n_in_corA_refl = []
                for i in range(4):
                    corA_side = corA[2][i][1]
                    corA_refl_side = corA_refl[2][i][1]
                    pos_of_n_in_corA += [[i,j] for j, x in enumerate(corA_side) if x == n and j == corA_side_num[i] - 1]
                    pos_of_n_in_corA_refl += [[i,j] for j, x in enumerate(corA_refl_side) if x == n and j == corA_refl_side_num[i] - 1]   
                if not pos_of_n_in_corA and not pos_of_n_in_corA_refl:
                    return(False)

                found_b_and_p_after_n_in_corA = []
                for i in pos_of_n_in_corA:
                    next_side_after_n = corA[2][(i[0] + 1) % 4]
                    L = corA_side_num[(i[0] + 1) % 4]
                    afterhang_of_n_in_corA = overhangsA[i[0]]
                    if L == 1 and b == next_side_after_n[1][0] and afterhang_of_n_in_corA == 0:
                        next_next_side_after_n = corA[2][(i[0] + 2) % 4]
                        LL = corA_side_num[(i[0] + 2) % 4]
                        if LL > 1:
                            next_next_side_after_n_prehang = next_next_side_after_n[0]
                            if pp == next_next_side_after_n[1][0] and p == next_next_side_after_n[1][1] and pp == next_next_side_after_n_prehang:
                                found_b_and_p_after_n_in_corA.append(i)
                                break

                if not found_b_and_p_after_n_in_corA:
                    for i in pos_of_n_in_corA_refl:
                        next_side_after_n = corA_refl[2][(i[0] + 1) % 4]
                        L = corA_refl_side_num[(i[0] + 1) % 4]
                        afterhang_of_n_in_corA_refl = overhangsA_refl[i[0]]
                        if L == 1 and b == next_side_after_n[1][0] and afterhang_of_n_in_corA_refl == 0:
                            next_next_side_after_n = corA_refl[2][(i[0] + 2) % 4]
                            LL = corA_refl_side_num[(i[0] + 2) % 4]
                            if LL > 1:
                                next_next_side_after_n_prehang = next_next_side_after_n[0]
                                if pp == next_next_side_after_n[1][0] and p == next_next_side_after_n[1][1] and pp == next_next_side_after_n_prehang:
                                    found_b_and_p_after_n_in_corA.append(i)
                                    break
                    
                if not found_b_and_p_after_n_in_corA:
                    return(False)

            # Subcase (ii):
            else:
                n = cur_side[q[1]+1]
                p = cur_side[q[1]-1]

                # In this subcase, we need each of the following to hold:
                # (1) "n" needs to be the last square of a side of "a"
                # (2) the overhang of n on that side must be 0
                # (3) "b" should be the only square of the next side (after the side of "n")
                # (4) "a" needs to be situated along the side of "b" the same in both corA and corB.
                #     This last condition will be measured using overhangs and prehangs in both coronas

                pos_of_n_in_corA = []
                pos_of_n_in_corA_refl = []
                for i in range(4):
                    corA_side = corA[2][i][1]
                    corA_refl_side = corA_refl[2][i][1]
                    pos_of_n_in_corA += [[i,j] for j, x in enumerate(corA_side) if x == n and j == corA_side_num[i] - 1]  
                    pos_of_n_in_corA_refl += [[i,j] for j, x in enumerate(corA_refl_side) if x == n and j == corA_refl_side_num[i] - 1]  
                if not pos_of_n_in_corA and pos_of_n_in_corA_refl:
                    return(False)

                found_b_and_p_after_n_in_corA = []
                for i in pos_of_n_in_corA:
                    next_side_after_n = corA[2][(i[0] + 1) % 4]
                    L = corA_side_num[(i[0] + 1) % 4]
                    afterhang_of_n_in_corA = overhangsA[i[0]]
                    if L == 1 and b == next_side_after_n[1][0] and afterhang_of_n_in_corA == 0:
                        part_of_b_before_a_in_corB = sum(cur_side[0:q[1]]) - cur_side_prehang
                        part_of_b_after_a_in_corB = sum(cur_side[q[1]+1:curL]) - cur_side_overhang
                        next_side_after_n_prehang = next_side_after_n[0]
                        next_side_after_n_overhang = overhangsA[(i[0] + 1) % 4]
                        if part_of_b_before_a_in_corB == next_side_after_n_overhang and part_of_b_after_a_in_corB == next_side_after_n_prehang:
                            found_b_and_p_after_n_in_corA.append(i)
                            break

                if not found_b_and_p_after_n_in_corA:
                    for i in pos_of_n_in_corA_refl:
                        next_side_after_n = corA_refl[2][(i[0] + 1) % 4]
                        L = corA_refl_side_num[(i[0] + 1) % 4]
                        afterhang_of_n_in_corA_refl = overhangsA_refl[i[0]]
                        if L == 1 and b == next_side_after_n[1][0] and afterhang_of_n_in_corA_refl == 0:
                            part_of_b_before_a_in_corB = sum(cur_side[0:q[1]]) - cur_side_prehang
                            part_of_b_after_a_in_corB = sum(cur_side[q[1]+1:curL]) - cur_side_overhang
                            next_side_after_n_prehang = next_side_after_n[0]
                            next_side_after_n_overhang = overhangsA_refl[(i[0] + 1) % 4]
                            if part_of_b_before_a_in_corB == next_side_after_n_overhang and part_of_b_after_a_in_corB == next_side_after_n_prehang:
                                found_b_and_p_after_n_in_corA.append(i)
                                break

                if not found_b_and_p_after_n_in_corA:
                    return(False)

        # Case (4): "a" is a last but not 1st square of a side of corB

        if q[1] == curL - 1 and q[1] != 0:

            # There are a few subcases to consider:
            # Subcase (i): cur_side_overhang == 0 and prehang on next side is less than 1st sqaure (no double corner)
            # Subcase (ii): cur_side_overhang == 0 and prehang on next side equals 1st square (double corner)
            # Subcase (iii): cur_side_overhang > 0 and this "a" is not forming a double corner at the beginning of cur_side
            # Subcase (iv): cur_side_overhang > 0 and this "a" forms a double corner at the beginning of cur_side
            
            n = next_side[0]
            p = cur_side[q[1]-1]

            # Subcase (i):
            if cur_side_overhang == 0 and next_side_prehang < n:

                # In this subcase, each of the following must be satisfied for compatibility:
                # (1) "n" is a last square of some side of "a" with positive overhang < "n"
                # (2) "b" is the first and only square of the next side (after "n") having prehang of 0.
                # (3) "p" is the first square of the next side (after "b") having prehang of 0.
                # (4) The afterhang of "n" in corA equals the overlap of "n" and "b" in corB

                pos_of_n_in_corA = []
                pos_of_n_in_corA_refl = []
                for i in range(4):
                    corA_side = corA[2][i][1]
                    corA_refl_side = corA_refl[2][i][1]
                    pos_of_n_in_corA += [[i,j] for j, x in enumerate(corA_side) if x == n and j == corA_side_num[i] - 1]
                    pos_of_n_in_corA_refl += [[i,j] for j, x in enumerate(corA_refl_side) if x == n and j == corA_refl_side_num[i] - 1]
                if not pos_of_n_in_corA and not pos_of_n_in_corA_refl:
                    return(False)

                found_b_and_p_after_n_in_corA = []
                for i in pos_of_n_in_corA:
                    next_side_after_n = corA[2][(i[0] + 1) % 4]
                    L = corA_side_num[(i[0] + 1) % 4]
                    afterhang_of_n_in_corA = overhangsA[i[0]]
                    if L == 1 and 0 < afterhang_of_n_in_corA < n and next_side_after_n[1][0] == b:
                        next_next_side_after_n = corA[2][(i[0] + 2) % 4]
                        if next_next_side_after_n[1][0] == p and next_next_side_after_n[0] == 0:
                            overlap_of_n_and_b_in_corB = n - next_side_prehang
                            afterhang_of_n_in_corA = overhangsA[i[0]]
                            if overlap_of_n_and_b_in_corB == afterhang_of_n_in_corA:
                                found_b_and_p_after_n_in_corA.append(i)
                                break
                if not found_b_and_p_after_n_in_corA:
                    for i in pos_of_n_in_corA_refl:
                        next_side_after_n = corA_refl[2][(i[0] + 1) % 4]
                        L = corA_refl_side_num[(i[0] + 1) % 4]
                        afterhang_of_n_in_corA_refl = overhangsA_refl[i[0]]
                        if L == 1 and 0 < afterhang_of_n_in_corA_refl < n and next_side_after_n[1][0] == b:
                            next_next_side_after_n = corA_refl[2][(i[0] + 2) % 4]
                            if next_next_side_after_n[1][0] == p and next_next_side_after_n[0] == 0:
                                overlap_of_n_and_b_in_corB = n - next_side_prehang
                                afterhang_of_n_in_corA_refl = overhangsA_refl[i[0]]
                                if overlap_of_n_and_b_in_corB == afterhang_of_n_in_corA_refl:
                                    found_b_and_p_after_n_in_corA.append(i)
                                    break

                if not found_b_and_p_after_n_in_corA:
                    return(False)
                            
            # Subcase (ii): cur_side_overhang == 0 and prehang on next side equals 1st square (double corner)
            if cur_side_overhang == 0 and next_side_prehang == n:
                nn = next_side[1] # note that since n = next_side_prehang, then len(next_side) > 1

                # In this subcase, each of the following must be satisfied for compatibility:
                # (1) "n" must appear as the last square of a side of "a" with 0 overhang
                # (2) "nn" must appear as the 1st square after "n" and have prehang = nn
                # (3) "b" must appear as the 2nd (and last) square of the side of "a" containing "nn"
                # (4) "p" must appear as the 1st square of the next side (after "b") with 0 prehang

                pos_of_n_in_corA = []
                pos_of_n_in_corA_refl = []
                for i in range(4):
                    corA_side = corA[2][i][1]
                    corA_refl_side = corA_refl[2][i][1]
                    pos_of_n_in_corA += [[i,j] for j, x in enumerate(corA_side) if x == n and j == corA_side_num[i] - 1]
                    pos_of_n_in_corA_refl += [[i,j] for j, x in enumerate(corA_refl_side) if x == n and j == corA_refl_side_num[i] - 1]
                if not pos_of_n_in_corA and not pos_of_n_in_corA_refl:
                    return(False)

                found_b_and_p_after_n_in_corA = []
                for i in pos_of_n_in_corA:
                    next_side_after_n = corA[2][(i[0] + 1) % 4]
                    L = corA_side_num[(i[0] + 1) % 4]
                    afterhang_of_n_in_corA = overhangsA[i[0]]
                    next_side_after_n_prehang = next_side_after_n[0]
                    if next_side_after_n[1][0] == nn and next_side_after_n_prehang == nn and L == 2:
                        if next_side_after_n[1][1] == b:
                            next_next_side_after_n = corA[2][(i[0] + 2) % 4]
                            next_next_side_after_n_prehang = next_next_side_after_n[0]
                            if next_next_side_after_n[1][0] == p and next_next_side_after_n_prehang == 0:
                                found_b_and_p_after_n_in_corA.append(i)
                                break

                if not found_b_and_p_after_n_in_corA:
                    for i in pos_of_n_in_corA_refl:
                        next_side_after_n = corA_refl[2][(i[0] + 1) % 4]
                        L = corA_refl_side_num[(i[0] + 1) % 4]
                        afterhang_of_n_in_corA_refl = overhangsA_refl[i[0]]
                        next_side_after_n_prehang = next_side_after_n[0]
                        if next_side_after_n[1][0] == nn and next_side_after_n_prehang == nn and L == 2:
                            if next_side_after_n[1][1] == b:
                                next_next_side_after_n = corA_refl[2][(i[0] + 2) % 4]
                                next_next_side_after_n_prehang = next_next_side_after_n[0]
                                if next_next_side_after_n[1][0] == p and next_next_side_after_n_prehang == 0:
                                    found_b_and_p_after_n_in_corA.append(i)
                                    break

                if not found_b_and_p_after_n_in_corA:
                    return(False)
                        
            # Subcase (iii): cur_side_overhang > 0 and this "a" is not forming a double corner at the beginning of cur_side
            if cur_side_overhang > 0 and (q[1] > 1 or cur_side_prehang != cur_side[0]):

                # In this subcase, each of the following must be satisfied for compatibility:
                # (1) "b" must be the last square of a side of corA with postitive overhang
                # (2) "n" must be the second to last square of that same side.
                # (3) "p" must be the first square of the next side with 0 prehang.
                # (4) The overlap between "a" and "b" must be the same in both corA and corB

                
                pos_of_b_in_corA = []
                pos_of_b_in_corA_refl = []
                for i in range(4):
                    corA_side = corA[2][i][1]
                    corA_refl_side = corA_refl[2][i][1]
                    pos_of_b_in_corA += [[i,j] for j, x in enumerate(corA_side) if x == b and j == corA_side_num[i] - 1]
                    pos_of_b_in_corA_refl += [[i,j] for j, x in enumerate(corA_refl_side) if x == b and j == corA_refl_side_num[i] - 1]
                if not pos_of_b_in_corA and not pos_of_b_in_corA_refl:
                    return(False)

                found_n_before_b_in_corA = []
                for i in pos_of_b_in_corA:
                    corA_side_of_b = corA[2][i[0]][1]
                    L = corA_side_num[i[0]]
                    if L > 1:
                        if n == corA_side_of_b[i[1] - 1]:
                            next_side_after_b = corA[2][(i[0] + 1) % 4]
                            if p == next_side_after_b[1][0] and next_side_after_b[0] == 0:
                                overlap_of_a_and_b_in_corB = a - cur_side_overhang
                                overlap_of_a_and_b_in_corA = b - overhangsA[i[0]]
                                if overlap_of_a_and_b_in_corB == overlap_of_a_and_b_in_corA:
                                    found_n_before_b_in_corA.append(i)
                                    break

                if not found_n_before_b_in_corA:
                    for i in pos_of_b_in_corA_refl:
                        corA_refl_side_of_b = corA_refl[2][i[0]][1]
                        L = corA_refl_side_num[i[0]]
                        if L > 1:
                            if n == corA_refl_side_of_b[i[1] - 1]:
                                next_side_after_b = corA_refl[2][(i[0] + 1) % 4]
                                if p == next_side_after_b[1][0] and next_side_after_b[0] == 0:
                                    overlap_of_a_and_b_in_corB = a - cur_side_overhang
                                    overlap_of_a_and_b_in_corA_refl = b - overhangsA_refl[i[0]]
                                    if overlap_of_a_and_b_in_corB == overlap_of_a_and_b_in_corA_refl:
                                        found_n_before_b_in_corA.append(i)
                                        break

                if not found_n_before_b_in_corA:
                    return(False)
            
            # Subcase (iv): cur_side_overhang > 0 and this "a" forms a double corner at the beginning of cur_side
            if cur_side_overhang > 0 and q[1] == 1 and cur_side[0] == cur_side_prehang and curL == 2:
                
                pp = prev_side[prevL - 1]

                # In this subcase, each of the following must be satisfied for compatibility:
                # (1) "b" must be the last square of some side of corA with zero overhang
                # (2) "n" must be the square just before "b" on that side
                # (3) "pp" must be the first square of the side after "b" with prehang equal to "pp"
                # (4) "p" must be the second square of the side of "pp" 

                pos_of_b_in_corA = []
                pos_of_b_in_corA_refl = []
                for i in range(4):
                    corA_side = corA[2][i][1]
                    corA_refl_side = corA_refl[2][i][1]
                    pos_of_b_in_corA += [[i,j] for j, x in enumerate(corA_side) if x == b and j == corA_side_num[i] - 1]
                    pos_of_b_in_corA_refl += [[i,j] for j, x in enumerate(corA_refl_side) if x == b and j == corA_refl_side_num[i] - 1]   
                if not pos_of_b_in_corA and not pos_of_b_in_corA_refl:
                    return(False)

                found_nb_in_corA = []
                for i in pos_of_b_in_corA:
                    side_of_b_in_corA = corA[2][i[0]]
                    side_of_b_overhang_in_corA = overhangsA[i[0]]
                    L = corA_side_num[i[0]]
                    if L > 1 and side_of_b_overhang_in_corA == 0:
                        next_side_after_b_in_corA = corA[2][(i[0] + 1) % 4]
                        next_side_after_b_in_corA_prehang = next_side_after_b_in_corA[0]
                        LL = corA_side_num[(i[0] + 1) % 4]
                        if side_of_b_in_corA[1][i[1] - 1] == n and LL > 1:
                            if pp == next_side_after_b_in_corA[1][0] and pp == next_side_after_b_in_corA_prehang and p == next_side_after_b_in_corA[1][1]:
                                found_nb_in_corA.append(i)
                                break
                if not found_nb_in_corA:
                    for i in pos_of_b_in_corA_refl:
                        side_of_b_in_corA_refl = corA_refl[2][i[0]]
                        side_of_b_overhang_in_corA_refl = overhangsA_refl[i[0]]
                        L = corA_refl_side_num[i[0]]
                        if L > 1 and side_of_b_overhang_in_corA_refl == 0:
                            next_side_after_b_in_corA_refl = corA_refl[2][(i[0] + 1) % 4]
                            next_side_after_b_in_corA_refl_prehang = next_side_after_b_in_corA_refl[0]
                            LL = corA_refl_side_num[(i[0] + 1) % 4]
                            if side_of_b_in_corA_refl[1][i[1] - 1] == n and LL > 1:
                                if pp == next_side_after_b_in_corA_refl[1][0] and pp == next_side_after_b_in_corA_refl_prehang and p == next_side_after_b_in_corA_refl[1][1]:
                                    found_nb_in_corA.append(i)
                                    break
                
                if not found_nb_in_corA:
                    return(False)

    return(True)
\end{lstlisting}

\subsection{Drawing Coronas}\label{apx:draw}\mbox{}\\

\begin{lstlisting}[caption={Drawing Coronas}, label={prog:drawing}]

# draw_sq_top draws a square of side length "size" and lower left-hand corner "llc".
# It is used later to draw the squares on the top side of a corona.
def draw_sq_top(size,llc):
    p = polygon([[llc[0],llc[1]],
                 [llc[0]+size,llc[1]],
                 [llc[0]+size,llc[1]+size],
                 [llc[0],llc[1]+size]],
                fill=false,rgbcolor=(0,0,0),aspect_ratio=1,axes=false)
    t = text(size,(llc[0]+0.5*size,llc[1]+0.5*size),fontsize=10)

    return p+t

# draw_sq_bottom draws a square of side length "size" and upper right-hand corner "urc". 
# It is used later to draw the squares on the bottom side of a corona.
def draw_sq_bottom(size, urc):
    p = polygon([[urc[0],urc[1]],
                 [urc[0],urc[1]-size],
                 [urc[0]-size,urc[1]-size],
                 [urc[0]-size,urc[1]]],
                fill=false,rgbcolor=(0,0,0),aspect_ratio=1,axes=false)
    t = text(size,(urc[0]-0.5*size,urc[1]-0.5*size),fontsize=10)

    return p+t

# draw_sq_top draws a square of side length "size" and upper left-hand corner "ulc".
# It is used later to draw the squares on the right side of a corona.
def draw_sq_right(size,ulc):
    p = polygon([[ulc[0],ulc[1]],
                 [ulc[0]+size,ulc[1]],
                 [ulc[0]+size,ulc[1]-size],
                 [ulc[0],ulc[1]-size]],
                fill=false,rgbcolor=(0,0,0),aspect_ratio=1,axes=false)
    t = text(size,(ulc[0]+0.5*size,ulc[1]-0.5*size),fontsize=10)
    return p+t

# draw_sq_top draws a square of side length "size" and lower right-hand corner "lrc".
# It is used later to draw the squares on the left side of a corona.
def draw_sq_left(size,lrc):    
    p = polygon([[lrc[0],lrc[1]],
                 [lrc[0]-size,lrc[1]],
                 [lrc[0]-size,lrc[1]+size],
                 [lrc[0],lrc[1]+size]],
                fill=false,rgbcolor=(0,0,0),aspect_ratio=1,axes=false)
    t = text(size,(lrc[0]-0.5*size,lrc[1]+0.5*size),fontsize=10)
    return p+t

# The data structure for a corona signature is as follows: 
# [[ x0, [a1,a2,...]], [ x1, [b1,b2,...]], [ x2, [c1,c2,...]], [ x3, [d1,d2,...]]] 

def draw_corona(center_size,corona_sig):
    corona = Graphics()
    #Drawing the center square
    corona += draw_sq_top(center_size,[0,0])

    C = corona_sig
    
    x0 = C[0][0]
    cur_pos = [-x0,center_size]
    for i in C[0][1]:
        corona += draw_sq_top(i,cur_pos)
        cur_pos[0]+=i

    x1=C[1][0]
    cur_pos = [center_size,center_size+x1]
    for i in C[1][1]:
        corona += draw_sq_right(i,cur_pos)
        cur_pos[1]-=i

    x2=C[2][0]
    cur_pos = [center_size + x2,0]
    for i in C[2][1]:
        corona += draw_sq_bottom(i,cur_pos)
        cur_pos[0]-=i

    x3=C[3][0]
    cur_pos = [0,-x3]
    for i in C[3][1]:
        corona += draw_sq_left(i,cur_pos)
        cur_pos[1]+=i

    return(corona)
\end{lstlisting}

\subsection{Computing all Compatible 4-Tuples of Coronas}\label{apx:computing}\mbox{}\\

\begin{lstlisting}[caption={Computing Compatible 4-Tuples of Coronas}, label={prog:computing}, basicstyle=\footnotesize\ttfamily]

%%time

All_Corona_4_Tups = list(it.product(FlatCors1,FlatCors2,FlatCors3,FlatCors4))
print("len(All_Corona_4_Tups) = ",len(All_Corona_4_Tups))

num_Compatible = 0
Compatible_Coronas = []
Comp_Cors_Indices = []

all_pairs = list(it.product(range(0,4),range(0,4)))
all_triples = list(it.product(range(0,4),range(0,4),range(0,4)))

index = -1
for Q in All_Corona_4_Tups:
    index += 1
    
    #Checking if coronas in Q are pairwise compatible
    flag = True
    for i,j in all_pairs:
        if Geom_Compat_Check(Q[i],Q[j]) == False:
            flag = False
            break
    if flag == False:
        continue

    #Checking that a pair of coronas in Q do not admit a UE-2 tiling
    #If there is a pair of coronas in Q that involve only 2 sizes of squares total among them,
    #then it generates a UE-2 tiling or possibly no tiling.
    flag = False
    for i,j in all_pairs:
        if i < j:
            DE1 = list(set(list(set(it.chain(Q[i][2][0][1],Q[i][2][1][1],Q[i][2][2][1],Q[i][2][3][1])))) | set([i+1]))
            DE2 = list(set(list(set(it.chain(Q[j][2][0][1],Q[j][2][1][1],Q[j][2][2][1],Q[j][2][3][1])))) | set([j+1]))
            DE_both = list(set(DE1) | set(DE2))
            if len(DE_both) == 2:
                flag = True
                break
    if flag == True:
        continue

    #Checking that the coronas do not admit UE-3 Tilings
    #If there is a triple of coronas in Q that involve only 3 sizes of squares total among them,
    #then it generates a UE-3 tiling or possibly no tiling.
    flag = False
    for i,j,k in all_triples:
        if i < j and j < k:
            DE1 = list(set(list(set(it.chain(Q[i][2][0][1],Q[i][2][1][1],Q[i][2][2][1],Q[i][2][3][1])))) | set([i+1]))
            DE2 = list(set(list(set(it.chain(Q[j][2][0][1],Q[j][2][1][1],Q[j][2][2][1],Q[j][2][3][1])))) | set([j+1]))
            DE3 = list(set(list(set(it.chain(Q[k][2][0][1],Q[k][2][1][1],Q[k][2][2][1],Q[k][2][3][1])))) | set([k+1]))
            DE_all = list(set(DE1) | set(DE2) | set(DE3))
            #print([[i,j,k],DE_all])
            if len(DE_all) < 4:
                flag = True
                break
    if flag == True:
        continue
    
    Compatible_Coronas.append(Q)
    Comp_Cors_Indices.append(index)
    num_Compatible = num_Compatible + 1
    
print("num_Compatible = ",num_Compatible)    
\end{lstlisting}

\begin{lstlisting}[caption={Drawing Compatible 4-Tuples of Coronas}, label={prog:draw_compat}, basicstyle=\footnotesize\ttfamily]

c = 0
print("Number of Compatible Coronas = ",len(Compatible_Coronas))
for Q in Compatible_Coronas:

    print("index = ", Comp_Cors_Indices[c])
    print("count = ", c+1)
    print("1-corona = ",Q[0])
    print("2-corona = ",Q[1])
    print("3-corona = ",Q[2])
    print("4-corona = ",Q[3])
    Latex_strg = "$( "
    for i in range(4):
        for j in range(4):
            Latex_strg += str(Q[i][2][j][1][0])
            Latex_strg += "^"
            Latex_strg += str(Q[i][2][j][0])
            for k in range(1,len(Q[i][2][j][1])):
                Latex_strg += "."
                Latex_strg += str(Q[i][2][j][1][k])
            if j < 3:
                Latex_strg += ";"
        if i < 3:
            Latex_strg += r" \,\, , \,\, "
        else:
            Latex_strg += " )$"
    print(Latex_strg)

    p1 = draw_corona(Q[0][0],Q[0][2])
    p2 = draw_corona(Q[1][0],Q[1][2])
    p3 = draw_corona(Q[2][0],Q[2][2])
    p4 = draw_corona(Q[3][0],Q[3][2])

    G = graphics_array([p1, p2, p3, p4])
    
    show(G,figsize=6)
    c += 1
    name = str(c) + '-cor.pdf'
    G.save(name,aspect_ratio=1,figsize=10)  
\end{lstlisting}

\end{appendices}

\printbibliography

@book{GS1987,
  title     = {Tilings and Patterns},
  author    = {Branko Gr{\"{u}}nbaum and G. C. Shephard},
  year      = 1987,
  publisher = {W. H. Freeman and Company},
  address   = {New York}
}

@article {MannUE4,
    AUTHOR = {Mann, Casey and DiNatale, Joseph and Peirce, Emily and
              Vitercik, Ellen},
     TITLE = {Unilateral and equitransitive tilings by squares of four
              sizes},
   JOURNAL = {Ars Math. Contemp.},
  FJOURNAL = {Ars Mathematica Contemporanea},
    VOLUME = {10},
      YEAR = {2016},
    NUMBER = {1},
     PAGES = {135--167},
      ISSN = {1855-3966,1855-3974},
   MRCLASS = {52C20 (05B45)},
  MRNUMBER = {3439934},
MRREVIEWER = {Elizaveta\ Zamorzaeva},
       DOI = {10.26493/1855-3974.644.d7c},
       URL = {https://doi.org/10.26493/1855-3974.644.d7c},
}

@incollection {Schatt,
    AUTHOR = {Schattschneider, D.},
     TITLE = {Unilateral and equitransitive tilings by squares},
      NOTE = {The Branko Gr\"unbaum birthday issue},
   JOURNAL = {Discrete Comput. Geom.},
  FJOURNAL = {Discrete \& Computational Geometry. An International Journal
              of Mathematics and Computer Science},
    VOLUME = {24},
      YEAR = {2000},
    NUMBER = {2-3},
     PAGES = {519--525},
      ISSN = {0179-5376,1432-0444},
   MRCLASS = {52C20 (05B45)},
  MRNUMBER = {1758068},
MRREVIEWER = {Elizaveta\ Zamorzaeva},
       DOI = {10.1007/s004540010053},
       URL = {https://doi.org/10.1007/s004540010053},
}

@article {Bolcskei,
    AUTHOR = {B\"olcskei, Attila},
     TITLE = {Classification of unilateral and equitransitive tilings by
              squares of three sizes},
   JOURNAL = {Beitr\"age Algebra Geom.},
  FJOURNAL = {Beitr\"age zur Algebra und Geometrie. Contributions to Algebra
              and Geometry},
    VOLUME = {41},
      YEAR = {2000},
    NUMBER = {1},
     PAGES = {267--277},
      ISSN = {0138-4821},
   MRCLASS = {52C20},
  MRNUMBER = {1745596},
}

@article {Martini,
    AUTHOR = {Martini, Horst and Makai, Endre and Soltan, Valeriu},
     TITLE = {Unilateral tilings of the plane with squares of three sizes},
   JOURNAL = {Beitr\"age Algebra Geom.},
  FJOURNAL = {Beitr\"age zur Algebra und Geometrie. Contributions to Algebra
              and Geometry},
    VOLUME = {39},
      YEAR = {1998},
    NUMBER = {2},
     PAGES = {481--495},
      ISSN = {0138-4821},
   MRCLASS = {52C20},
  MRNUMBER = {1642720},
MRREVIEWER = {R.\ Blind},
}

\end{document}